\documentclass[11pt,reqno]{amsart}
\usepackage{graphicx,epsfig}
\usepackage{epstopdf}
\usepackage{pdfpages}
\usepackage{amssymb}
\usepackage{empheq}
\usepackage{cases}
\usepackage{amsthm,amsmath}
\usepackage{caption,lipsum}
\usepackage{stmaryrd}
\usepackage{tabularx}
\usepackage{color}
\usepackage{empheq,float}
\DeclareGraphicsExtensions{.eps}
\usepackage{amssymb,amsmath,graphicx,amsfonts,euscript,mathrsfs}
\usepackage{xcolor}
\usepackage{soul}
\usepackage[normalem]{ulem}

\usepackage[colorlinks]{hyperref}
\usepackage[capitalize,nameinlink,noabbrev]{cleveref}

\newtheorem{theorem}{Theorem}[section]
\newtheorem{lemma}[theorem]{Lemma}

\newtheorem{remark}[theorem]{Remark}
\theoremstyle{definition}
\newtheorem{definition}[theorem]{Definition}
\newtheorem{example}[theorem]{Example}

\def\R{\mathbb{R}}
\def\N{\mathbb{N}}
\def\Z{\mathbb{Z}}
\def\T{\mathbb{T}}
\def\C{\mathbb{C}}

\def\eps{\varepsilon}

\newcommand{\fe}{\mathrm{e}}

\newcommand{\bN}{{\mathbb N}}

\numberwithin{equation}{section}

\begin{document}
\hypersetup{pageanchor=false}
\pagenumbering{arabic}

\title[NLS with rough potential]{The cubic nonlinear Schr\"odinger equation with rough potential}

\author[N. Mauser]{Norbert J. Mauser}
\address{\hspace*{-12pt}N.~Mauser: Research Platform MMM c/o Fak. Mathematik, Univ. Wien, Oskar-Morgenstern-Platz 1, A-1090, Vienna.}
\email{mauser@courant.nyu.edu}
\urladdr{https://www.wpi.ac.at/director.html}

\author[Y. Wu]{Yifei Wu}
\address{\hspace*{-12pt}Y.~Wu: Center for Applied Mathematics, Tianjin University, 300072, Tianjin, P. R. China.}
\email{yerfmath@gmail.com}
\urladdr{http://cam.tju.edu.cn/~yfw/}

\author[X. Zhao]{Xiaofei Zhao}
\address{\hspace*{-12pt}X.~Zhao: School of Mathematics and Statistics \& Computational Sciences Hubei Key Laboratory, Wuhan University, 430072 Wuhan, China}
\email{matzhxf@whu.edu.cn}
\urladdr{http://jszy.whu.edu.cn/zhaoxiaofei/en/index.htm}



\date{}
\dedicatory{}

\begin{abstract}\noindent
We consider the cubic nonlinear Schr\"odinger equation with a spatially rough potential, a key equation in the mathematical setup for nonlinear Anderson localization. Our study comprises two main parts: new optimal results on the well-posedness analysis on the PDE level, and subsequently a new efficient numerical method, its convergence analysis and simulations that illustrate our analytical results.
In the analysis part, our results focus on understanding how the regularity of the solution is influenced by the regularity of the potential, where we provide quantitative and explicit characterizations. Ill-posedness results are also established to demonstrate the sharpness of the obtained regularity characterizations and to indicate the minimum regularity required from the potential for the NLS to be solvable.
Building upon the obtained regularity results, we design an appropriate numerical discretization for the model and establish its convergence with an optimal error bound. The numerical experiments in the end not only verify the theoretical regularity results, but also confirm the established convergence rate of the proposed scheme. Additionally, a comparison with other existing schemes is conducted to demonstrate the better accuracy of our new scheme in the case of a rough potential.

\end{abstract}
\maketitle

\tableofcontents

\section{Introduction}\label{sec:introduction}

We deal with the following cubic nonlinear Schr\"odinger equation (NLS) with a (``rough") spatial potential
\begin{equation}\label{model}
 \left\{\begin{aligned}
& i\partial_tu(t,x)+\partial_{xx} u(t,x)+\xi(x) u(t,x)
=\lambda|u(t,x)|^2u(t,x),
 &&\mbox{for}\,\,\, x\in\T\,\,\,\mbox{and}\,\,\, t>0, \\
 &u(0,x)=u_0(x), &&\mbox{for}\,\,\,  x\in\T.
 \end{aligned}\right.
\end{equation}
Here $u(t,x): \R^+\times\T\to\C$ is the complex-valued unknown ``wave function". We choose the setup on the one-dimensional torus $\T=(-\pi,\pi)$, which is an appropriate setting for efficient numerical methods.  $\xi(x):\T\to\R$ is a given real-valued potential that can ``rough" in terms of regularity.
$u_0(x)$ is a given initial wave function, and $\lambda\in\R$ a given parameter. Here $\lambda<0$ resp- $\lambda>0$ corresponds to the focusing resp. defocusing nonlinear self-interaction cases. \\
The solution of this NLS satisfies the following mass and energy conservation laws:
\begin{subequations}
\begin{align}
&\frac1{2\pi}\int_\T |u(t,x)|^2\,d x = \frac1{2\pi}\int_\T |u_0(x)|^2\,d x,  \quad\mbox{for}\,\,\, t>0, \label{mass}\\
&\int_\T \left[|\partial_x u(t,x)|^2-\xi(x)|u(t,x)|^2+\frac{\lambda}{2}|u(t,x)|^4\right]dx=
 \int_\T \left[|\partial_x u_0|^2-\xi|u_0|^2+\frac{\lambda}{2}|u_0|^4\right]dx.
\end{align}
\end{subequations}

When the potential function $\xi(x)$ is considered as a rough/random potential, (\ref{model}) is often called the disordered NLS, as it arises in the context of the so-called Anderson localization:  In 1958,  Philip W. Anderson discovered for the linear and lattice Schr\"odinger model that the waves in the solution will be localized when the potential $\xi$ is spatially random/rough enough \cite{Anderson}. Such phenomenon found important applications, e.g., for semiconductors. Note that the NLS (\ref{model}) is the continuous version of a discrete/discretized ``lattice" model.
Nonlinearities in the Schr\"odinger equation pose many deep questions for physicists and mathematicians. E.g., in the defocusing NLS case where the interaction tends to push the wave to the far field, there is a long-standing debate about whether the localization dominates in the end or the spreading dominates. There is a wide range of works to address this problem from different aspects, e.g., \cite{BourgainWang,review,prl2,Kachman,prl1,Sergey,Wang2}. Numerical simulations were done to suggest answers on the nonlinear lattice model, see e.g., \cite{review,prl2,prl1}.\\
However, the mathematical and numerical studies of (\ref{model}) are  far from done.
With this article we contribute new results to both in a setup on the torus in one space dimension for (\ref{model}) that serves as a valid truncation of localized wave dynamics up to a finite time.

For applications of the nonlinear localization model (\ref{model}) in physics, we refer to \cite{Conti,nonlinear wave,pra,PRLnew}. On the mathematical level, let us briefly review some theoretical results on (\ref{model}) related to our work. Without the potential function, i.e., $\xi\equiv0$, it is known that  \eqref{model} is globally well-posed in $H^s$ for $s\ge 0$; see e.g., \cite{Bo,Taobook}.
With $\xi(x)$ an $L^\infty$-potential, Cazenave \cite{Ca} showed the global well-posedness of (\ref{model}) for small initial data in $H^1$. For potentials that are stochastic in time but rather regular in space, we refer to \cite{debussche2} for the well-posedness theory of the model.
For less regular spatial random potentials, \cite{Labbe,CAL1,CAL0} studied the spectrum of the linear part of the operator in (\ref{model}), i.e., $-\partial_x^2+\xi$, and the results revealed the mechanism of localization in the continuous level.
Rodnianski and Schlag \cite{Schlag} analyzed the decaying property of the solution for the linear Schr\"odinger model.
For a spatial white noise potential, the work of   \cite{debussche3,debussche1} shows that (\ref{model}) has a solution almost surely in $H^1$ provided the smallness of the initial data.
For a deterministic rough potential or one precise sample of a spatial noise, the sharp regularity of the solution of (\ref{model}) is not clearly known from the existing results.

Although the numerical methods for Schr\"odinger models have been extensively developed in the literature, see e.g.,  \cite{Bao,BaoCai,Hong,Jin}, the existing studies mostly concern a smooth setup and do not work for the NLS equations with rough potentials.
A clear understanding of the regularities in (\ref{model}) is a basis for optimal numerical methods and convergence analysis, that allow for reliable
simulations of the nonlinear localization or delocalization phenomenon of waves in the model (\ref{model}).
A rough potential $\xi$ is felt in the solution $u$ of (\ref{model}) and will certainly bring serious numerical difficulties due to the simultaneous low-regularity in $\xi$ and $u$. Indeed, the low-regularity not just  affects the spatial discretization accuracy but also affects temporal discretization error, because the truncation  error of an  approximation usually involves derivatives of the functions. Popular traditional schemes for NLS models like the finite difference methods or operator/time splitting methods or exponential integrators can suffer from severe loss of accuracy when the solution is not smooth enough \cite{Bao1,dnls-fd,lownls}. In numerical experiments below, we will show their poor performance on approximating (\ref{model}) in case of rough potential.

In this work, we first address the well-posedness of the NLS equation (\ref{model}). Although in the original spirit of Anderson, the potential is generated according to some random distribution, the localization phenomenon as such is not a mere stochastic behaviour. It occurs for any rough enough potential $\xi$.
Therefore, we consider the deterministic case and focus on the effect of the roughness. In particular, we aim to understand how the regularity of the potential affects the regularity of the solution of (\ref{model}). We then continue with the numerical analysis and propose accurate discretizations for (\ref{model}). Hence, in the second part of the work, we are going to  derive an efficient tailored low-regularity integrator for solving (\ref{model}) based on the established regularity results from the first part, and then we address that under a rough potential with certain regularity, what the numerical method could offer for accuracy. The accurate computational results will in turn verify the theoretical regularity results of the solution.\\

The main theoretical results of this paper are given in the following theorems.

\subsection{Well-posedness theory}\label{sec2:well-posed}
Firstly,  let us introduce the following space $\hat b^{s,p}$ to characterize the regularity of a function $f(x)$ on $\T$ (particularly for the potential function $\xi(x)$) based on its Fourier coefficients:
\begin{equation}\label{bhat-def}
\hat b^{s,p}=\hat b^{s,p}(\T)\triangleq \{f(x): |\hat f_0|+\||k|^s\hat f_k\|_{l^p}<+\infty\},\quad f(x)=\sum\limits_{k\in\Z}\hat{f}_k \fe^{i kx},\quad x\in\T,
\end{equation}
where $s\geq0$ describes the differentiability and $1\le p\le  \infty$ gives the integrability.
When $p=2$, it is the usual Sobolev space $\hat b^{s,2}=H^s$. For simplicity, we denote $\hat b^{0,p}$ by $\hat l^p$.
A critical characteristic index $\gamma_p$ is defined as
\begin{equation}\label{gammap-def}
\gamma_p\triangleq\frac32+\frac1p.
\end{equation}
Before stating the results, we recall the following definition for the well-posedness.

\begin{definition}[Well-posedness]\label{def1}
The well-posedness of a time dependent PDE can be defined as follows:
Denote by $C_t\left(I;X_0\right)$ the space of continuous functions from the time interval $I$ to the topological space $X_0$. We say that the Cauchy problem is locally well-posed in $C_t\left(I;X_0\right)$ if the following properties hold:
\begin{enumerate}
\item There is unconditional uniqueness in $C_t\left(I;X_0\right)$ for the problem.
\item For every $u_{0}\in X_0$, there exists a strong solution defined on a maximal time interval $I=\left[0,T_{\max}\right)$, with $T_{\max}\in (0,+\infty]$.
\item The solution map $u_0\mapsto u[u_0]$ is continuous from $X_0$ to $X_0$.
\end{enumerate}
\end{definition}

\setcounter{theorem}{0}
\begin{remark}
 The well-posedness defined above is stronger than the common definition, which can be regarded as the unconditional well-posedness.
\end{remark}

Now we state our first well-posedness result for the NLS equation \eqref{model}.


\setcounter{theorem}{0}
\begin{theorem}[\bf{Well-posedness for $\hat b^{s,p}$-potential}]\label{main:thm1}
Let $\xi\in \hat b^{s,p}$ for $s\ge 0,2<p\le \infty$, so $\gamma_p\in(\frac32,2)$.
Then,
  \eqref{model} is locally well-posed in $H^{s+\gamma_p-}(\T)$ (notation $\gamma_p-$ explained in \cref{sec2:subsec1} (iii), (iv).
\end{theorem}

\setcounter{theorem}{1}
\begin{remark}\label{remark1}
In   \cref{main:thm1}, the regularity of the solution is essentially determined by the regularity of the potential.
In order to focus on the effect from the potential, we will not address the impact from rough initial data in this work.  On the other hand, requiring more regularity from the initial data will neither  improve the well-posedness result. We will explain this point in \cref{remark smooth} and also in \cref{sec:numer1}  by numerical experiments.
\end{remark}
\begin{remark}
Note that $L^1(\T)\hookrightarrow \hat l^\infty$. Therefore, \cref{main:thm1} with $s=0,p=\infty$ covers particularly the case of $\delta$-potential, i.e., $\xi(x)=\delta(x)$ in \eqref{model}. Note that the well-posedness of NLS under a delta potential is known only in $H^1$  in the literature \cite{Fukuizumi,Goodman}. Now \cref{main:thm1} states that the regularity $H^{\frac32-}$ can be attained, and this will be shown to be sharp.
 \end{remark}

Let $s\ge 0, 1\le r\le  \infty$. Denote
$$
H^{s,r}=H^{s,r}(\T)\triangleq\{f(x): \|f\|_{L^r(\T)}+\big\|(-\partial_x^2)^\frac s2 f\big\|_{L^r(\T)}<+\infty\}.
$$
Note that by Young's inequality, we have that $\hat b^{s,p} \hookrightarrow H^{s,p'}(\T)$ where $1/p'+1/p=1$. If we replace $\hat b^{s,p}(\T)$ for $s\ge 0,\,2\le p< \infty$ by a slightly stronger space  $H^{s,p'}(\T)$, then we can cover the endpoint regularity.

%
\setcounter{theorem}{1}
\begin{theorem}[\bf{Well-posedness for $H^{s,p'}$-potential}]\label{main:thm3-smoothdata}
Let $\xi\in H^{s,p'}(\T)$ for $s\ge 0,\,2\le p< \infty$ and so $\gamma_p\in(\frac32,2]$. Then, \eqref{model} is locally well-posed in  $H^{s+\gamma_p}(\T)$.
\end{theorem}
%

\subsection{Ill-posedness theory}
When one of the conditions in  \cref{def1} is violated,
 the Cauchy problem \eqref{model} is said to be ill-posed in space $X_0$. In this work, we refer to the violation of the third condition (around zero solution), which is especially relevant to valid numerical approximations. More precisely, the definition is the following.
 \setcounter{theorem}{1}
\begin{definition}[Ill-posedness]\label{def2}
Let $\epsilon>0$ and denote
\begin{equation} \label{Beps-def}
B(\epsilon) \triangleq \{u_0\in \mathcal{S}: \big\|u_0\big\|_{X_0}\le \epsilon\},\quad
\mbox{where}\ \mathcal{S}\ \mbox{is the Schwartz space}.
\end{equation}
By ``ill-posedness" of (\ref{model}) in $X_0$ we mean that  for arbitrarily small $\epsilon>0$, there exists a $u_0\in B(\epsilon)$ such that  the solution map $u_0\mapsto u[u_0]$ is discontinuous from $X_0$ to $C([0,1];X_0)$.
\end{definition}

The following result shows that for general $ \hat b^{s,p}$-potential $\xi(x)$, we can only expect the solution to be in $H^{s+\gamma_p}(\T)$ at most. This hence implies the sharpness of \cref{main:thm1}.
\setcounter{theorem}{2}
\begin{theorem}[\bf{Ill-posedness for $\hat b^{s,p}$-potential}]\label{main:thm1-ill}
For $s\geq0,\,2< p\le \infty$ and any $\gamma\ge \gamma_p$, there exists some $\xi\in \hat b^{s,p}$  such that \eqref{model} is ill-posed in $H^{s+\gamma}(\T)$.
\end{theorem}

\setcounter{theorem}{3}
\begin{remark}\label{remark smooth}
Note that the case $s=0,p=\infty$ is not covered in \cref{main:thm3-smoothdata}, and indeed here \cref{main:thm1-ill} states that the endpoint regularity does not hold for $\xi \in \hat l^\infty$. As a matter of fact, for the case of  $\xi \in \hat l^\infty$, the counterexample we choose in the proof of \cref{main:thm1-ill} later is exactly $\xi(x)=\delta(x)$. Thus, for the $\delta$-potential of \eqref{model}, the regularity $C_t^0([0,T]; H^{\frac32-}(\T))$ is sharp. Moreover, in the proof we set a smooth initial data
$$
 u_0(x)= \epsilon\big(1+2\cos(2x)\big),\quad \forall\epsilon>0.
$$
This in general means that the smoothness of initial function will not make things better.
\end{remark}

For potentials that belong to the usual Sobolev space $H^s$, i.e., $p=p'=1/2$ in   \cref{main:thm3-smoothdata}, the following result states that $H^{s+2}$ is the most regularity that we can expect in general for the solution.
\setcounter{theorem}{3}
\begin{theorem}[\bf{Ill-posedness for $H^s$-potential}]\label{main:thm2-ill}
For $s\geq0$ and any  $\gamma> 2$, there exists some $\xi\in H^s$ such that \eqref{model} is ill-posed in  $H^{s+\gamma}(\T)$.
\end{theorem}
%
%

Last but not least, we show by the following theorem that for the NLS equation \eqref{model} to be well-posed, the potential $\xi\in\hat l^\infty$ is the lowest regularity requirement.
\begin{theorem}[\bf{Minimum regularity for potential}]\label{main:thm3}
For any  $s<0$ and $1\le p\le \infty$, there exists some $\xi\in \hat b^{s,p}\setminus l^\infty$ such that  for any $\gamma\in \R$, \eqref{model} is ill-posed in $H^\gamma(\T)$.
\end{theorem}

\subsection{Numerical counterpart}
We are particularly interested in understanding the numerical issues for solving the NLS equation (\ref{model}) for the rough potential case $\xi\in \hat b^{s,p}$. Traditional numerical methods could be completely  inaccurate due to the roughness, and so the computational results are not reliable. Hence, the first goal would be a correct and accurate scheme that can work under the roughest case. Here we propose a simple and effective exponential type integration scheme  which will be derived later in the spirit of a low-regularity integrator (LRI) \cite{lownls}.

Let $\tau=\Delta t>0$ be the time step with the time grid points $t_n=n\tau$ for $n=0,1\ldots,$ and denote $u^n(x)\approx u(t_n,x)$ for the numerical solution. Our \emph{LRI scheme} reads as follows: start with the exact initial value $u^0(x)=u_0(x)$ and then update as
\begin{align}\label{NuSo-NLS}
u^{n+1}(x)=i\tau  \mathcal D_{\tau}\left[\xi(x)\>
\mathcal D_\tau [u^n(x)]\right]
+\mathcal N_\tau\left[\fe^{i\tau\partial_x^2} u^n(x)\right],\quad n=0,1,\ldots,\quad x\in\T,
\end{align}
 where the two operators $\mathcal D_\tau$ and $\mathcal N_\tau$ are defined by
\begin{equation}\label{Dtau-def}
\mathcal D_\tau [f]\triangleq  \left(\fe^{i\tau \partial_x^2}-1\right)\big(i\tau \partial_x^2\big
)^{-1}f,\qquad \mathcal N_\tau[f]\triangleq \fe^{-i\tau\lambda |P_{\leq N} f|^2}f,
\end{equation}
with $P_{\le N} f \triangleq  \sum_{|k|\le N}\hat f_k \fe^{ikx}$ ``cutting off high modes", and with {$N=\tau^{-\frac12+\eps_0}$ for any $0<\eps_0\leq\frac{1}{8(s+\gamma_p)}$}. Here $\hat f_k$ denotes the Fourier coefficient of $f=f(x)$ for $x\in\T,$ and $P_{\le N}$ is referred as the filter operator in \cite{Ignat,Zuazua}.
The derivation of the scheme (\ref{NuSo-NLS}) will be given in \cref{sec: num method}. Clearly, \eqref{NuSo-NLS} is explicit in time, and thanks to the period boundary conditions that we chose, it can be efficiently implemented under Fourier spectral discretization \cite{Trefethen} using FFT (fast Fourier transforms).  In practice, when the number of spatial grid points which is also the number of total frequencies is greater than or equal to $2N$, the filter $P_{\le N}$ becomes trivial for the implementation.
Moreover, such a filter primarily serves for theoretical results, and in practice the scheme performs very similarly without it.

For the semi-discretized (in time) LRI scheme (\ref{NuSo-NLS}),
the following theorem gives its error estimate up to a fixed time of computation.
\begin{theorem}[\bf{Error estimate for the numerical method}]\label{thm:main-N1}
Let $\xi\in \hat b^{s,p}$ with $s\ge0, 2\le p\le +\infty$ in \eqref{model}.  For any $u_0\in H^{s+2}(\T)$ and any $0<T<T_{\mathrm{max}}$, there exist constants $C>0,\tau_0>0$ such that for $0<\tau\leq\tau_0$, the numerical solution given by \eqref{NuSo-NLS}  satisfies
\begin{equation}\label{error-tau-1}
\max_{0\le t_n\le T}  \|u(t_n)-u^{n}\|_{L^2}
  \le C\tau^{\min\left\{s+\frac1p+\frac14-,1\right\}},
\end{equation}
where $\tau_0$ and $C$ depend only on $T$, $\lambda$, $\|\xi\|_{\hat b^{s,p}}$ and $\|u_0\|_{H^{s+2}}$.
\end{theorem}

As illustrated in the above result, for the ``worst'' potential case $\xi\in \hat{l}^\infty$, the scheme (\ref{NuSo-NLS}) has $\mathcal{O}(\tau^{\frac14-})$ accuracy which is indeed quite low as a convergence rate but at least can guarantee the  convergence and correctness of the computational results. The convergence order (\ref{error-tau-1}) will be verified by numerical experiments. In addition, numerical comparisons will highlight that the accuracy of (\ref{NuSo-NLS}) is indeed better than the existing schemes for (\ref{model}). To achieve a higher order accurate scheme would be an extremely challenging future task. With the reliable numerical results provided by the LRI scheme (\ref{NuSo-NLS}), we will be able to verify the theoretical  results given in \cref{sec2:well-posed}. The predicted regularity in the theorems can be exactly observed in our numerical results (see \cref{fig:thm1}-\cref{fig:thm4}).

The main difficulty and innovation of the mathematical analysis in the work is summarized here. For the well-posedness and ill-posedness theorems, the low regularity of the considered $\xi$, for instance for $\xi\in L^1$, makes the usual H\"older type estimate for the product $\xi u$ no longer available. Delicate frequency analysis has to be developed here. For the numerical analysis, the roughness limits the measure of error under the $L^2$-norm (\ref{error-tau-1}), where the algebraic property is missing. The established $L^2$-estimate here essentially benefits from the design of the approximation part $\mathcal D_{\tau}\left[\xi\>
\mathcal D_\tau [u^n]\right]$ in the scheme (\ref{NuSo-NLS}), which has a crucial inner product structure.

The rest of our paper is organized as follows. In  \cref{sec: pre}, we introduce some preliminaries for the following analysis. In \cref{sec: well}, we prove the well-posedness results for the NLS equation (\ref{model}). In \cref{sec: ill}, we prove the ill-posedness results for (\ref{model}). The numerical scheme and convergence analysis are given in \cref{sec: num method}. Numerical experiments are done in \cref{sec: num result} and conclusions are drawn in \cref{sec: con}.

\vskip1cm

 \section{Preliminaries}\label{sec: pre}
As a preparation for subsequent PDE and numerical analysis, we give some notations and lemmas that will be frequently used in the rest of the paper.

\subsection{Basic notations}\label{sec2:subsec1}
We adopt the following widely used notations from harmonic analysis and partial differential equations \cite{Bo,Taobook}:
\begin{enumerate}

\item[(i)]
For a function $f(t,x)$ which depends on $t$ and $x$, we simply denote $f(t)=f(t,\cdot)$.

\item[(ii)]
Denote $\langle k\rangle=(1+|k|^2)^{\frac12}$ for $k\in\Z$.

\item[(iii)]
The notation $a+$ stands for $a+\epsilon$ with an arbitrary small $\epsilon>0$,
and $a-$ stands for $a-\epsilon$.

\item[(iv)]
Denote  $\gamma_p=\frac32+\frac1p$ for some given $p\geq1$.

\item[(v)] For a sequence $\{a_N\}_{N\in\bN}$, denote the discrete norm as $\|a_N\|_{l^p_N}=(\sum_{N=1}^\infty |a_N|^p)^{\frac1p}$, and sometimes we will omit the subscript $N$ in $l^p_N$ for brevity.

\item[(vi)]
Denote by $C$ a generic positive constant which may has different values at different occurrences, possibly depending on the norms of the solution and $T$ but independent of the step size $\tau$ and time level $n$ in numerical analysis.

\item[(vii)]
Denote by $A\lesssim B$ or $B\gtrsim A$ the statement ``$A\leq CB$ for some constant $C>0$''.

\item[(viii)]
Denote by $A\sim B$ the statement ``$C^{-1}B\le A\leq CB$ for some constant $C>0$''. Namely, $A\sim B$ is equivalently to $A\lesssim B\lesssim A$. Moreover,
denote by $A\ll B$ or $B\gg A$ the statement $A\le C^{-1}B$ for some sufficiently large constant $C$.


\end{enumerate}

With the notations above, we often decompose a subset $E\subset \Z^2=\{(k_1,k_2):k_1,k_2\in\Z\}$ into two parts, i.e., $E=E_1\cup E_2$, with
$$
E_1=\{(k_1,k_2)\in E: |k_1|\ll |k_2|\}
\quad\mbox{and}\quad
E_2=\{(k_1,k_2)\in E: |k_1|\gtrsim |k_2|\} .
$$
This means that we consider the decomposition with
$$
E_1=\{(k_1,k_2)\in E: |k_1|< c|k_2|\}
\quad\mbox{and}\quad
E_2=\{(k_1,k_2)\in E: |k_1|\ge c|k_2|\} ,
$$
where $c>0$ is some sufficiently small constant (independent of $\tau$ and $n$) which can satisfy the requirement in our analysis.

\subsection{Fourier transform}\label{subsec1}
The inner product and the norm of $L^2(\T)$ are defined by
$$
\langle f,g\rangle \triangleq \mathrm{Re} \int_\T f(x)\overline{g(x)}\,d x
\quad\mbox{and}\quad
\|f\|_{L^2(\T)}\triangleq \sqrt{\langle f,f\rangle} .
$$
The Fourier transform of a function $f\in L^2(\T)$ is defined by
$$
\mathcal{F}_k[f] \triangleq  \frac1{2\pi}\displaystyle\int_{\T}
\fe^{- i kx}f( x)\,d x,\quad k\in\Z.
$$
For the simplicity of notation, we also denote $\hat{f}_k=\mathcal{F}_k[f]$ and $f=\mathcal{F}_k^{-1}[\hat f_k]$. The following standard properties of the Fourier transform are well known:
\begin{align*}
&\mbox{(Fourier series expansion)}&&f(x)=\sum\limits_{k\in\Z}\hat{f}_k \fe^{i kx}; \\
&\mbox{(Plancherel's identity)} &&\|f\|_{L^2(\T)}= \sqrt{2\pi}\Big(\sum\limits_{k\in\Z}|\hat f_k|^2 \Big)^\frac12; \\
&\mbox{(Parseval's identity)} &&\langle f,g\rangle  =2\pi \mathrm{Re}\sum\limits_{k\in\Z} \hat f_k \overline{\hat g_k}; \\
&\mbox{(Conversion of products to convolutions)} &&\mathcal{F}_k[fg]=\sum\limits_{k_1+k_2=k}\hat f_{k_1}\hat g_{k_2}.
\end{align*}

The Sobolev space $H^s(\T)$ with some $s\in\R$, consists of generalized functions $f=\sum\limits_{k\in\Z}\hat{f}_k \fe^{i kx} $  such that $\|f\|_{H^s}<\infty$, where
$$
\|f\|_{H^s} \triangleq   \sqrt{2\pi}\bigg(\sum_{k\in\Z} \langle k\rangle^{2s} |\hat f_k|^2 \bigg)^{\frac12}.
$$
The operator $J^s=(1-\partial_{x}^2)^\frac s2: H^{s_0}(\T)\rightarrow H^{s_0-s}(\T)$ with $s_0,s\in\R$, is defined as
$$
J^s f \triangleq  \sum_{k\in\Z} \langle k\rangle^{s} \hat f_k \fe^{ikx}, \quad\forall\, f\in H^{s_0}(\T) ,
$$
and so we have $\|f\|_{H^s(\T)} = \|J^s f \|_{L^2(\T)} $.

We often use the abbreviations $H^s(\T)=H^s$, $L^p(\T)=L^p$ and $l^p_k=l^p_{k\in\Z}$.  Moreover, we use the similar abbreviations for the spacetime norms, like $L^q_t H^{s}_x(I\times \T)=L^q_t H^{s}_x(I)$.

\subsection{Some operators}
\label{section:projection}
Some frequently used operators are defined as follows.
\begin{enumerate}

\item[(i)] Projection: for any real number $N\ge 0$,  the Littlewood--Paley projections $P_{\le N}: H^s(\T)\rightarrow H^s(\T)$ and $P_{> N}: H^s(\T)\rightarrow H^s(\T)$ are defined as
\begin{align*}
P_{\le N} f \triangleq  \mathcal{F}^{-1}_k\big( 1_{|k|\le N} \mathcal{F}_k[f] \big) = \sum_{|k|\le N}\hat f_k \fe^{ikx} ,\quad
P_{> N} f \triangleq  \mathcal{F}^{-1}_k\big( 1_{|k|> N} \mathcal{F}_k[f] \big) = \sum_{|k|> N}\hat f_k \fe^{ikx} .
\end{align*}
  \item[(ii)] The inversion of differentiation: $\partial_x^{-1}: H^s(\T)\rightarrow H^{s+1}(\T)$ is defined by
\begin{equation*}
\mathcal{F}_k[\partial_x^{-1}f]
\triangleq \Bigg\{ \aligned
    &(ik)^{-1}\hat f_k,  &&\mbox{for}\,\,\, k\ne 0,\\
    &0, &&\mbox{for}\,\,\, k= 0.
   \endaligned
\end{equation*}
  \item[(iii)] Average: define the {\it average}   of a time-dependent function $f(t)$ in the interval $[0,\tau]$ by
\begin{equation}
\mathcal M_\tau(f)\triangleq \frac1\tau \int_0^\tau f(t)\,d t.\label{time average operator}
\end{equation}
\end{enumerate}



%

\subsection{Some tool lemmas}

 \begin{lemma}[\cite{LiWu-2022}]\label{lem:average2}
Let $\alpha,\beta\in \R$. If $\alpha, \beta\ne 0$ and $s\in[0,\tau]$, then
\begin{align*}
\big|\mathcal M_\tau\big(\fe^{is(\alpha+\beta)}\big)-\mathcal M_\tau\big(\fe^{is\alpha}\big)\mathcal M_\tau\big(\fe^{is\beta}\big)\big|\lesssim  \min\left\{\left|\frac{\alpha}{\beta}\right|,\left|\frac{\beta}{\alpha}\right|,\tau|\alpha|,\tau|\beta|\right\}.
\end{align*}
If $\alpha+\beta\ne 0$, then
\begin{align*}
\big|\mathcal M_\tau\big(\fe^{is(\alpha+\beta)}\big)-\mathcal M_\tau\big(\fe^{is\alpha}\big)\mathcal M_\tau\big(\fe^{is\beta}\big)\big|\lesssim  \tau^{-1}|\alpha+\beta|^{-1}.
\end{align*}
\end{lemma}

 \begin{lemma}[Schur's test]\label{lem:schurtest}
	For any $a>0$, let sequences $\{a_N\},\,\{b_N\}\in l_{N\in2^\N}^2$, then we have
	\begin{align*}
		\sum_{N_1\le  N} \left(\frac{N_1}{N}\right)^a a_N \>b_{N_1} \lesssim  \big\|a_N\big\|_{l_N^2} \big\|b_N\big\|_{l_N^2}.
	\end{align*}
\end{lemma}

\begin{lemma}[\cite{Tao-illpose}]\label{lm:ill-tool}
Consider a {quantitatively well posed} abstract equation in spaces $D$ and $S$,
\begin{equation*}
u=L(f)+N_k(u,\ldots,u),
\end{equation*}
which means for all $f\in D,\ u_1,\ldots,u_k\in S$ and some constant $C>0$,
$$\|L(f)\|_{S}\leq C\|f\|_D,\quad \|N_k(u_1,\ldots,u_k)\|_S\leq C\|u_1\|_S\ldots\|u_k\|_S.$$
Here $(D,\|\|_{D})$ is a Banach space of initial data and $(S,\|\|_S)$ is a Banach space of spacetime functions.
Define
\begin{align*}
 A_1(f)\triangleq L(f),\quad A_n(f)\triangleq \sum_{n_1,\ldots,n_k\geq1,n_1+\ldots+n_k=n}N_k(A_{n_1}(f),\ldots,A_{n_k}(f)),\ n>1.
\end{align*}
 Then for some $C_1>0$, all $f,g\in D$ and all $n\geq1$,
$$\|A_n(f)-A_n(g)\|_S\leq  C_1^n\left(\|f\|_D+\|g\|_{D}\right)^{n-1}\|f-g\|_D.$$
\end{lemma}


\vskip 1cm

\section{Proof for well-posedness theory}\label{sec: well}
In this section, we shall give the proofs for the two main well-posednesss results, i.e., \cref{main:thm1} and \cref{main:thm3-smoothdata}. Without loss of generality, we assume $\lambda=-1$ in the NLS model (\ref{model}) for simplicity of notations.

\subsection{Proof of \cref{main:thm1}}\label{sec:Thm1}
%

Let $T$ be a fixed time with $0<T<+\infty$ which will be determined later.
By Duhamel's formula, we have that for $0\le t \le T$,
\begin{align*}
u(t)=\fe^{i(t-t_0)\partial_x^2} u(t_0)+i\int_{t_0}^t\fe^{i(t-\rho)\partial_x^2}\Big(\xi u(\rho)+|u(\rho)|^2u(\rho)\Big)\,d\rho.
\end{align*}
Denote $v(t)=\fe^{-it\partial_x^2}u(t)$ and set $t_0=0$, then it reduces to
\begin{align}\label{Duhamel-1}
v(t)=u_0+i\int_0^t\fe^{-i\rho\partial_x^2}\Big(\xi u(\rho)+|u(\rho)|^2u(\rho)\Big)\,d\rho.
\end{align}

We denote the operator $\Phi$ by
\begin{align}\label{def:Phi}
\Phi(v)\triangleq u_0+i\int_0^t\fe^{-i\rho\partial_x^2}\Big(\xi u(\rho)+|u(\rho)|^2u(\rho)\Big)\,d\rho,
\end{align}
and
$$
R_0\triangleq \max\{2\big\|u_0\big\|_{H^{s+\gamma}_x}, 1\}.
$$
Then  we aim to show that for any $v\in L^\infty_tH^{s+\gamma}_x((0,T)\times\T)$ with
\begin{align}\label{assum}
\big\|v\big\|_{L^\infty_tH^{s+\gamma}_x}\le R_0,
\end{align}
it holds that
\begin{align}\label{est:bounded}
\big\|\Phi(v)\big\|_{L^\infty_tH^{s+\gamma}_x}\le & R_0.
\end{align}
Here and below, we use the abbreviation for the norms $L^\infty_tH^{s+\gamma}_x=L^\infty_tH^{s+\gamma}_x([0,T]\times \T)$.
Moreover, for any $v_1,v_2\in L^\infty_tH^{s+\gamma}_x((0,T)\times\T)$ with
$$
\big\|v_j\big\|_{L^\infty_tH^{s+\gamma}_x}\le R_0, \quad j=1,2,
$$
it holds that
\begin{align}\label{est:contraction}
\big\|\Phi(v_1)-\Phi(v_2)\big\|_{L^\infty_tH^{s+\gamma}_x}\le & \theta \big\|v_1-v_2\big\|_{L^\infty_tH^{s+\gamma}_x},
\quad \mbox{ for some }\theta \in (0,1).
\end{align}
The proof of \eqref{est:contraction} is similar as the proof of \eqref{est:bounded}, we only consider the latter.
Then it can be reduced to the following lemma.
\begin{lemma}\label{lem:Phi} Let $s\ge 0, 2<p\le +\infty, \gamma=\gamma_p-$ and denote  $\varepsilon_0=\min\{\frac12(\gamma_p-\gamma),\frac18\}$. Then  for any $v$ satisfying \eqref{assum}, and any $N_0>0$, there exists a function $C_0(N_0)>0$ which is dependent on $\|\xi\|_{\hat b^{s,p}}$ such that
\begin{align*}
\big\|\Phi(v)\big\|_{L^\infty_tH^{s+\gamma}_x}\le \frac12 R_0+ \Big[C_0(N_0)T+N_0^{-\frac{\varepsilon_0}2}\Big]\big(R_0+R_0^3\big).
\end{align*}
\end{lemma}

The proof of   \cref{lem:Phi} will be split into several parts.
Firstly, we consider the term
$$
\Phi^1(u)\triangleq i\int_0^t\fe^{-i\rho\partial_x^2}\Big(\xi u(\rho)\Big)\,d\rho,
$$
and we have  the following estimate for it.
\begin{lemma}\label{lem:Phi-i} Under the assumption of \cref{lem:Phi},
\begin{align*}
\big\|\Phi^1(v)\big\|_{L^\infty_tH^{s+\gamma}_x}\le  \Big[C_0(N_0)T+N_0^{-\frac{\varepsilon_0}2}\Big]\big(R_0+R_0^3\big).
\end{align*}
\end{lemma}
\begin{proof}
Taking Fourier transform for $\Phi^1(v)$, we have that for any $k\in \Z$, its Fourier coefficients denoted by $I_k$, equal to
\begin{align*}
I_k\triangleq i\int_0^t\sum\limits_{k_1+k_2=k}\fe^{i\rho\phi(k,k_2)}\hat \xi_{k_1}
\hat v_{k_2}(\rho)\,d\rho.
\end{align*}
Here we denote the phase functions $\phi$ by
$$
\phi(k,k_j)\triangleq |k|^2-|k_j|^2.
$$
Then
\begin{align*}
\left\|i\int_0^t\fe^{-i\rho\partial_x^2}\Big(\xi u(\rho)\Big)\,d\rho\right\|_{H^{s+\gamma}}
=\big\|\langle k\rangle^{s+\gamma} I_k\big\|_{l_k^2}.
\end{align*}
Note that $s+\gamma>\frac12$, then by Cauchy-Schwartz's inequality, we have
\begin{align}
\big\|\langle k\rangle^{s+\gamma} I_k\big\|_{l^2_k(\{|k|\le N_0\})}
\lesssim & N_0^{s+\gamma}  \big\|\hat\xi_k\big\|_{l^\infty_k} \int_0^t \sum\limits_{k_2}| v_k|\,d\rho\nonumber\\
\lesssim &t N_0^{s+\gamma} \big\|\hat\xi_k\big\|_{l^\infty_k} \|v\|_{L^\infty_tH^{s+\gamma}_x}.\label{est:low-frq}
\end{align}
Hence, we only consider the piece
$
\big\|\langle k\rangle^{s+\gamma} I_k\big\|_{l^2_k(\{|k|> N_0\})}.
$
Now we split $I$ into two parts as
$$
I_k=I_{1,k}+I_{2,k},
$$
where
\begin{align*}
I_{1,k}\triangleq i\int_0^t\sum\limits_{\substack{k_1+k_2=k\\ \phi(k,k_2)=0}}\hat \xi_{k_1}
\hat v_{k_2}(\rho)\,d\rho;\quad
I_{2,k}\triangleq i\int_0^t\sum\limits_{\substack{k_1+k_2=k\\ \phi(k,k_2)\ne 0}}\fe^{i\rho\phi(k,k_2)}\hat \xi_{k_1}
\hat v_{k_2}(\rho)\,d\rho.
\end{align*}

We first  estimate  $I_{1,k}$. Note that $\phi(k,k_2)=0$ is equivalent to $k=k_2$ or $k=-k_2$, and so
\begin{align*}
I_{1,k}= i\int_0^t \hat \xi_0 \hat v_{k}(\rho)\,d\rho
+ i\int_0^t \hat \xi_{2k}\>\hat v_{-k}(\rho)\,d\rho.
\end{align*}
Therefore,
\begin{align*}
\big\|\langle k\rangle^{s+\gamma} I_{1,k}\big\|_{l^2_k}
\lesssim &\int_0^t \big\|\langle k\rangle^{s+\gamma}  \hat \xi_0 \hat v_{k}(\rho)\big\|_{l^2_k}\,d\rho
+ \int_0^t \big\|\langle k\rangle^{s+\gamma} \hat \xi_{2k}\>\hat v_{-k}(\rho)\big\|_{l^2_k}\,d\rho\\
\lesssim & t \big\|\hat \xi_k\big\|_{l^\infty_k} \big\|\langle k\rangle^{s+\gamma}\hat v_k\big\|_{L^\infty_t l^2_k}
\lesssim t \big\|\hat \xi_k\big\|_{l^\infty_k} \|v\|_{L^\infty_tH^{s+\gamma}_x}.
\end{align*}

Now we start to estimate  $I_{2,k}$.  By integration-by-parts and noting that
\begin{align*}
\partial_tv(t)=\fe^{-it\partial_x^2}\left(\xi u+|u|^2u\right),
\end{align*}
we have
\begin{align*}
I_{2,k}= & i\sum\limits_{\substack{k_1+k_2=k\\ \phi(k,k_2)\ne 0}}\fe^{i\rho\phi(k,k_2)} \frac{\hat \xi_{k_1}
\hat v_{k_2}(\rho)}{i\phi(k,k_2)}\Big|_0^t
-i\int_0^t\sum\limits_{\substack{k_1+k_2+k_3=k\\ \phi(k,k_2+k_3)\ne 0}}
\frac{\fe^{i\rho\phi(k,k_3)}}{i\phi(k,k_2+k_3)}
\hat \xi_{k_1} \hat \xi_{k_2}
\hat v_{k_3}(\rho)\,d\rho\\
& - i\int_0^t\sum\limits_{\substack{k_1+k_2=k\\ \phi(k,k_2)\ne 0}}
\fe^{i\rho k^2}\frac{1}{i\phi(k,k_2)}
\hat \xi_{k_1} \widehat{\big(|u|^2u\big)}_{k_2}\,d\rho\\
\triangleq &
I_{21,k}+I_{22,k}+I_{23,k} .
\end{align*}

\subsubsection{Estimates on $I_{21,k}$.} For this term, by using the dual and Parseval's identity, we obtain
\begin{align}\label{781}
\big\|\langle k\rangle^{s+\gamma}I_{21,k}\big\|_{l^2(|k|\geq N_0)}=&\sup\limits_{h:\|h\|_{l^2}=1}\big\langle\langle k\rangle^{s+\gamma}I_{21,k}, h_k\big\rangle\nonumber\\
=&2\pi \sup\limits_{h:\|h\|_{l^2}=1} \mbox{Re} \sum\limits_{\substack{k_1+k_2=k\\ \phi(k,k_2)\ne 0,|k|\geq N_0 }}  \langle k\rangle^{s+\gamma}
\fe^{\rho\phi(k,k_2)} \frac{\hat \xi_{k_1}
\hat v_{k_2}(\rho)}{i\phi(k,k_2)}\Big|_0^t
\> \overline{h_k}\nonumber\\
\leq&2\pi \sup\limits_{h:\|h\|_{l^2}=1}\sup\limits_{t\in [0,T]} \sum\limits_{\substack{k_1+k_2=k\\ \phi(k,k_2)\ne 0,|k|\geq N_0, |k_1|\lesssim|k_2| }}\langle k\rangle^{s+\gamma}\frac{\big|\hat \xi_{k_1}\big|
\big|\hat v_{k_2}(\rho)\big|}{|\phi(k,k_2)|}|h_k|\nonumber\\
&+2\pi\sup\limits_{h:\|h\|_{l^2}=1}\sup\limits_{t\in [0,T]} \sum\limits_{\substack{k_1+k_2=k\\ \phi(k,k_2)\ne 0,|k|\geq N_0, |k_1|\gg|k_2| }}\langle k\rangle^{s+\gamma}\frac{\big|\hat \xi_{k_1}\big|
\big|\hat v_{k_2}(\rho)\big|}{|\phi(k,k_2)|}|h_k|\nonumber\\
\triangleq &
I_{211}+I_{212}.
\end{align}
For short, we omit  $\sup\limits_{h:\|h\|_{l^2}=1}\sup\limits_{t\in [0,T]} $ in the front, and turn to check the estimate hold for any $h$ such that $\|h\|_{l^2}= 1$ and any $t\in [0,T]$
in the proof.

$\bullet$ {\bf Estimates on $I_{211}$.} Note that $|\phi(k,k_2)|=(k-k_2)(k+k_2)$, and $\phi\neq0$ implies $|k\pm k_2|\geq1$. Thus,
\begin{align*}
|\phi(k,k_2)|=|k-k_2||k+k_2|\geq |k|\geq N_0.
\end{align*}
Then, combining with the estimate
$$
|\phi(k,k_2)|=|k_1(k_1+2k_2)|\sim \langle k_1\rangle\langle k_1+2k_2\rangle,
$$
we have
\begin{align*}
I_{211}
\lesssim &
\sum\limits_{\substack{k_1+k_2=k\\ \phi(k,k_2)\ne 0,|k|\geq N_0, |k_1|\lesssim|k_2| }} \langle k_2\rangle^{s+\gamma}\frac 1{|\phi(k,k_2)|}\big|\hat \xi_{k_1}\big||\hat v_{k_2}||h_k|\\
\lesssim &
N_0^{-\varepsilon_0} \sum\limits_{\substack{k_1+k_2=k\\ \phi(k,k_2)\ne 0,|k|\geq N_0, |k_1|\lesssim|k_2| }} \langle k_2\rangle^{s+\gamma}\frac 1{|\phi(k,k_2)|^{1-\varepsilon_0}}\big|\hat \xi_{k_1}\big||\hat v_{k_2}||h_k|\\
\lesssim &
N_0^{-\varepsilon_0}\sum_{k_2}\sum_{k_1}\frac {1}{\langle k_1\rangle^{1-\varepsilon_0}}  \frac {1}{\langle k_1+2k_2\rangle^{1-\varepsilon_0}}\big|\hat \xi_{k_1}\big|\langle k_2\rangle^{s+\gamma}|\hat v_{k_2}||h_{k_1+k_2}|\\
\lesssim &
N_0^{-\varepsilon_0}\big\|\hat \xi_{k}\big\|_{l_{k}^{\infty}}\sum_{k_2}\sum_{k_1}\left(\frac {1}{\langle k_1\rangle^{2-2\varepsilon_0}}  +\frac {1}{\langle k_1+2k_2\rangle^{2-2\varepsilon_0}}\right)\langle k_2\rangle^{s+\gamma}|\hat v_{k_2}||h_{k_1+k_2}|\\
\lesssim &
N_0^{-\varepsilon_0}\big\|\hat \xi_{k}\big\|_{l_{k}^{\infty}}\sum_{k_1}\frac {1}{\langle k_1\rangle^{2-2\varepsilon_0}}\sum_{k_2}\langle k_2\rangle^{s+\gamma}|\hat v_{k_2}||h_{k_1+k_2}|\\
&+N_0^{-\varepsilon_0}\big\|\hat \xi_{k}\big\|_{l_{k}^{\infty}}\sum_{\tilde{k}_1}\frac {1}{\langle\tilde{k}_1\rangle^{2-2\varepsilon_0}}\sum_{k_2}\langle k_2\rangle^{s+\gamma}|\hat v_{k_2}||h_{\tilde{k}_1-k_2}|,
\end{align*}
where in the last inequality, we have used the change of the variable $\tilde k_1=k_1+2k_2$.
By the Cauchy-Schwartz inequality, we obtain that
\begin{align}\label{782}
|I_{211}|\lesssim  & N_0^{-\varepsilon_0}\big\|\hat \xi_{k}\big\|_{l_{k}^{\infty}}\|v\|_{H^{s+\gamma}}\|h_k\|_{l_k^2}
\lesssim  N_0^{-\varepsilon_0}\|\xi\|_{\hat{b}^{s,p}}\|v\|_{H^{s+\gamma}}.
\end{align}

$\bullet$ {\bf Estimates on  $I_{212}$.} When $|k_1|\gg|k_2|$, we   have
$$
|\phi(k,k_2)|\gtrsim |k_1|^2.
$$
Then  by the Cauchy-Schwartz inequality, we find
\begin{align*}
I_{212}\lesssim &\sum\limits_{\substack{k_1+k_2=k\\ \phi(k,k_2)\ne 0,|k|\geq N_0, |k_1|\gg|k_2| }}\langle k_1\rangle^{\gamma+s-2}\big|\hat \xi_{k_1}\big||\hat v_{k_2}||h_{k_1+k_2}|\\
\lesssim &\sum_{k_2}\Big(\sum_{k_1}\big(\langle k_1\rangle^s\big|\hat \xi_{k_1}\big|\big)^p\Big)^{\frac 1p}\Big(\sum_{k_1:|k_1|\gtrsim N_0}\big(\langle k_1\rangle^{\gamma-2}|h_{k_1+k_2}|\big)^{p'}\Big)^{\frac 1{p'}}|\hat v_{k_2}|.
\end{align*}
Then by further using the Cauchy-Schwartz inequality,  we get
\begin{align*}
I_{212}
\lesssim & \|\xi\|_{\hat{b}^{s,p}}\|v\|_{H^{s+\gamma}}\|h\|_{l^2}\big\|\langle k_1\rangle^{\gamma-2}\big\|_{l^r(|k_1|\gtrsim N_0)}\sum_{k_2}|\hat v_{k_2}|,\nonumber
\end{align*}
where $r$ satisfies that $\frac 1r=\frac 12-\frac 1p$ and so $(\gamma-2)r<-1$. Moreover, $r+s>\frac12$ and so we get
\begin{equation}\label{783}
I_{212}
\lesssim N_0^{-\varepsilon_0}\|\xi\|_{\hat{b}^{s,p}}\|v\|_{L^\infty_tH^{s+\gamma}_x}.
\end{equation}

Collecting the estimates \eqref{781}-\eqref{783}, we obtain
\begin{align*}
\big\|\langle k\rangle^{s+\gamma}I_{21,k}\big\|_{l^2(|k|\geq N_0)}\lesssim N_0^{-\varepsilon_0}\|\xi\|_{\hat{b}^{s,p}}\|v\|_{H^{s+\gamma}}.
\end{align*}

\subsubsection{Estimates on $I_{22,k}$.}\label{subsub1} For this term, we further split it into the following two parts as
\begin{align*}
I_{22,k}= &-i\int_0^t\sum\limits_{\substack{k_1+k_2+k_3=k\\ \phi(k,k_2+k_3)\ne 0,\phi(k,k_3)=0}}
\frac{\fe^{i\rho\phi(k,k_3)}}{i\phi(k,k_2+k_3)}
\hat \xi_{k_1} \hat \xi_{k_2}
\hat v_{k_3}(\rho)\,d\rho\\
&-i\int_0^t\sum\limits_{\substack{k_1+k_2+k_3=k\\ \phi(k,k_2+k_3)\ne 0,\phi(k,k_3)\ne 0}}
\frac{\fe^{i\rho\phi(k,k_3)}}{i\phi(k,k_2+k_3)}
\hat \xi_{k_1} \hat \xi_{k_2}
\hat v_{k_3}(\rho)\,d\rho\\
\triangleq &
I_{221,k}+I_{222,k}.
\end{align*}

$\bullet$ {\bf Estimates on $I_{221,k}$.}  Note that $\phi(k,k_3)=0$ implies that $k=k_3$ or $k=-k_3$. Therefore,
\begin{align*}
I_{221,k}= &-i\int_0^t\sum\limits_{k_1: \phi(k,-k_1+k)\ne 0}
\frac1{i\phi(k,-k_1+k)}
\hat \xi_{k_1} \hat \xi_{-k_1}
\hat v_{k}(\rho)\,d\rho\\
&-i\int_0^t\sum\limits_{\substack{k_1+k_2=2k\\ \phi(k,k_2-k)\ne 0}}
\frac1{i\phi(k,k_2-k)}
\hat \xi_{k_1} \hat \xi_{k_2}
\hat v_{-k}(\rho)\,d\rho\\
\triangleq &
I_{2211,k}+I_{2212,k}.
\end{align*}

{\emph{1) On $I_{2211,k}$}.}  Since $\phi(k,-k_1+k)=k_1(2k-k_1)$, it can be rewritten as
\begin{align*}
I_{2211,k}= &-i\int_0^t\sum\limits_{\substack{k_1\neq 0\\ 2k-k_1\ne 0}}
\frac1{ik_1(2k-k_1)}
\hat \xi_{k_1} \hat \xi_{-k_1}
\hat v_{k}(\rho)\,d\rho.
\end{align*}
Therefore, we have that
\begin{align*}
\big\|\langle k\rangle^{s+\gamma} I_{2211,k}\big\|_{l^2_k}
= &\Big\|\langle k\rangle^{s+\gamma}\int_0^t\sum\limits_{\substack{k_1\neq 0\\ 2k-k_1\ne 0}}
\frac1{|k_1||2k-k_1|}
\big|\hat \xi_{k_1}\big|| \hat \xi_{-k_1} |
|\hat v_{k}(\rho)|\,d\rho\Big\|_{l^2_k}\\
\lesssim &
\big\|\hat \xi_{k}\big\|_{l^\infty_k}^2
\Big\|\langle k\rangle^{s+\gamma}\int_0^t\sum\limits_{\substack{k_1\neq 0\\ 2k-k_1\ne 0}}
\left(\frac1{|k_1|^2}+\frac1{|2k-k_1|^2}\right)
|\hat v_{k}(\rho)|\,d\rho\Big\|_{l^2_k}\\
\lesssim &
\big\|\hat \xi_{k}\big\|_{l^\infty_k}^2
\int_0^t\big\|\langle k\rangle^{s+\gamma}
|\hat v_{k}(\rho)|\big\|_{l^2_k}\,d\rho
\lesssim
t\big\|\hat \xi_k\big\|_{l^\infty_k}^2\|v\|_{L^\infty_tH^{s+\gamma}_x}.
\end{align*}

{\emph{2) On $I_{2212,k}$.}} Since $\phi(k,k_2-k)=k_1 k_2$, it can be rewritten as
\begin{align*}
I_{2212,k}= &-i\int_0^t\sum\limits_{\substack{k_1+k_2=2k\\ k_1\ne 0,k_2\ne 0}}
\frac1{ik_1k_2}
\hat \xi_{k_1} \hat \xi_{k_2}
\hat v_{-k}(\rho)\,d\rho.
\end{align*}
Therefore, we have that
\begin{align*}
\big\|\langle k\rangle^{s+\gamma} I_{2212,k}\big\|_{l^2_k}
= &\Big\|\langle k\rangle^{s+\gamma}\int_0^t\sum\limits_{\substack{k_1+k_2=2k\\ k_1\ne 0,k_2\ne 0}}
\frac1{|k_1||k_2|}
\hat \xi_{k_1} \hat \xi_{k_2}
\hat v_{-k}(\rho)\,d\rho\Big\|_{l^2_k}\\
\lesssim &
\big\|\hat \xi_{k}\big\|_{l^\infty_k}^2
\Big\|\langle k\rangle^{s+\gamma}\int_0^t\sum\limits_{\substack{k_1\neq 0\\ 2k-k_1\ne 0}}
\left(\frac1{|k_1|^2}+\frac1{|2k-k_1|^2}\right)
|\hat v_{-k}(\rho)|\,d\rho\Big\|_{l^2_k}\\
\lesssim &
t\big\|\hat \xi_k\big\|_{l^\infty_k}^2\|v\|_{L^\infty_tH^{s+\gamma}_x}.
\end{align*}

$\bullet$ {\bf Estimates on $I_{222,k}$.} By integration-by-parts, we obtain that
\begin{align*}
&I_{222,k}\\
=&
-i\sum\limits_{\substack{k_1+k_2+k_3=k\\ \phi(k,k_2+k_3)\ne 0,\phi(k,k_3)\ne 0}}
\frac{\fe^{i\rho\phi(k,k_3)}}{i\phi(k,k_2+k_3)\cdot i\phi(k,k_3)}
\hat \xi_{k_1} \hat \xi_{k_2}
\hat v_{k_3}(\rho)\Big|_0^t\\
& +
i\int_0^t\sum\limits_{\substack{k_1+k_2+k_3+k_4=k\\ \phi(k,k_2+k_3+k_4)\ne 0,\phi(k,k_3+k_4)\ne 0}}
\frac{\fe^{i\rho\phi(k,k_4)}}{i\phi(k,k_2+k_3+k_4)\cdot i\phi(k,k_3+k_4)}
\hat \xi_{k_1} \hat \xi_{k_2}  \hat \xi_{k_3}
\hat v_{k_4}(\rho)\,d\rho\\
&+
i\int_0^t \sum\limits_{\substack{k_1+k_2+k_3=k\\ \phi(k,k_2+k_3)\ne 0,\phi(k,k_3)\ne 0}}
\frac{\fe^{i\rho k^2}}{i\phi(k,k_2+k_3)\cdot i\phi(k,k_3)}
\hat \xi_{k_1} \hat \xi_{k_2}
\widehat{\big(|u|^2u\big)}_{k_3}\,d\rho\\
\triangleq &
I_{2221,k}+I_{2222,k}+I_{2223,k}.
\end{align*}

{\emph{1) On $I_{2221,k}$.}} Note that
$$
\phi(k,k_2+k_3)=k_1(2k-k_1);\quad
\phi(k,k_3)=(k-k_3)(k+k_3).
$$
Moreover, we use the formula that when $\phi(k,j)\ne 0$,
\begin{align*}
\frac1{\phi(k,j)}=\frac1{2k}\Big(\frac1{k-j}+\frac1{k+j}\Big)
\end{align*}
and obtain that
\begin{align*}
&\frac{1}{\phi(k,k_3)}=\frac 1{2k}\Big[\frac 1{k-k_3}+\frac1{k+k_3}\Big],\quad \frac{1}{\phi(k,k_2+k_3)}=\frac 1{2k}\Big[\frac 1{k-(k_2+k_3)}+\frac1{k+k_2+k_3}\Big].
\end{align*}
Then by the Cauchy-Schwartz inequality, we obtain that when $\phi(k,k_2+k_3)\ne 0$,
\begin{align}
\left|\frac1{\phi(k,k_2+k_3)}\right|
= &
 \frac1{|k_1|^{\frac12\varepsilon_0}|2k-k_1|^{\frac12\varepsilon_0}}\frac{2^{1-\frac12\varepsilon_0}}{|k|^{1-\frac12\varepsilon_0}}\left(\frac1{|k_1|^{1-\frac12\varepsilon_0}}+\frac1{|2k-k_1|^{1-\frac12\varepsilon_0}}\right)\notag\\
\lesssim &
\langle k\rangle^{-1+\frac12\varepsilon_0} \left(\frac1{\langle k_1\rangle^{1+\frac12\varepsilon_0}}+\frac1{\langle 2k-k_1\rangle^{1+\frac12\varepsilon_0}}\right);\label{phi-2-3}
\end{align}
and similarly,
\begin{align}\label{phi-0-3}
\left|\frac1{\phi(k,k_3)}\right|
\lesssim &
\langle k\rangle^{-1+\frac12\varepsilon_0} \left(\frac1{\langle k-k_3\rangle^{1+\frac12\varepsilon_0}}+\frac1{\langle k+k_3\rangle^{1+\frac12\varepsilon_0}}\right).
\end{align}
Together with \eqref{phi-2-3} and \eqref{phi-0-3}, and using $k_1=k-k_2-k_3$ it gives that
\begin{align}\label{phi-20-3}
\frac{1}{|\phi(k,k_2+k_3)||\phi(k,k_3)|}\lesssim &\langle k\rangle^{-2+\varepsilon_0}\left[\frac 1 {\langle k-k_3\rangle^{1+\frac {1}{2}\varepsilon_0}}+\frac 1 {\langle k+k_3\rangle^{1+\frac {1}{2}\varepsilon_0}}\right]\notag\\
&\cdot \left[\frac 1 {\langle k-(k_2+k_3)\rangle^{1+\frac {1}{2}\varepsilon_0}}+\frac 1 {\langle k+k_2+k_3\rangle^{1+\frac {1}{2}\varepsilon_0}}\right].
\end{align}
Now we are ready to estimate $I_{2221,k}$. By duality, we have that
\begin{align*}
\Big\|\langle k\rangle^{s+\gamma}|I_{2221,k}|\Big\|_{l_k^2(|k|\geq N_0)}
=&\sup\limits_{h:\|h\|_{l^2}=1}\big\langle \langle k\rangle^{s+\gamma}|I_{2221,k}|, h_k\big\rangle\nonumber\\
=&\sup\limits_{h:\|h\|_{l^2}=1}\sup\limits_{t\in [0,T]}\sum\limits_{\substack{k_1+k_2+k_3=k\\ \phi(k,k_2+k_3)\ne 0,\phi(k,k_3)\ne 0}}\frac{\langle k\rangle^{s+\gamma}}{|\phi(k,k_2+k_3)||\phi(k,k_3)|}\\
&\quad \cdot |\hat \xi_{k_1} | |\hat \xi_{k_2} ||\hat v_{k_3} ||h_k|.
\end{align*}
As before, we  omit  $\sup\limits_{h:\|h\|_{l^2}=1}\sup\limits_{t\in [0,T]}$ in the front. Then by \eqref{phi-20-3}, the trivial inequality
$
\langle k\rangle^s\le \langle k_1\rangle^s\langle k_2\rangle^s\langle k_3\rangle^s,
$
 and the Cauchy-Schwartz inequality, we  obtain
\begin{align*}
&\Big\|\langle k\rangle^{s+\gamma}|I_{2221,k}|\Big\|_{l_k^2(|k|\geq N_0)} \\
\lesssim&\sum\limits_{\substack{k_1+k_2+k_3=k,|k|\geq N_0\\ \phi(k,k_2+k_3)\ne 0,\phi(k,k_3)\ne 0}}\frac{\langle k\rangle^{\gamma}}{|\phi(k,k_2+k_3)||\phi(k,k_3)|}\langle k_1\rangle^s|\hat \xi_{k_1} |\langle k_2\rangle^s|\hat \xi_{k_2} |\langle k_3\rangle^s|\hat v_{k_3} ||h_k| \\
\lesssim&\sum\limits_{\substack{k_1+k_2+k_3=k,|k|\geq N_0\\ \phi(k,k_2+k_3)\ne 0,\phi(k,k_3)\ne 0}}\langle k\rangle^{-2+\varepsilon_0+\gamma}\left[\frac 1 {\langle k-k_3\rangle^{1+\frac {1}{2}\varepsilon_0}}+\frac 1 {\langle k+k_3\rangle^{1+\frac {1}{2}\varepsilon_0}}\right] \\
&\cdot \left[\frac 1 {\langle k-(k_2+k_3)\rangle^{1+\frac {1}{2}\varepsilon_0}}+\frac 1 {\langle k+k_2+k_3\rangle^{1+\frac {1}{2}\varepsilon_0}}\right]\langle k_1\rangle^s|\hat \xi_{k_1} |\langle k_2\rangle^s|\hat \xi_{k_2} |\langle k_3\rangle^s|\hat v_{k_3} ||h_k|.
\end{align*}
Denote
$$\zeta_{k_3,k}^1=\langle k\rangle^{-2+\varepsilon_0+\gamma}\left[\frac 1 {\langle k-k_3\rangle^{1+\frac {1}{2}\varepsilon_0}}+\frac 1 {\langle k+k_3\rangle^{1+\frac {1}{2}\varepsilon_0}}\right]|h_k|
$$
and
$$
\zeta_{k_3,k}^2=\sum_{k_2} \left[\frac 1 {\langle k-(k_2+k_3)\rangle^{1+\frac {1}{2}\varepsilon_0}}+\frac 1 {\langle k+k_2+k_3\rangle^{1+\frac {1}{2}\varepsilon_0}}\right]\langle k-k_2-k_3\rangle^s|\hat \xi_{k-k_2-k_3} |\langle k_2\rangle^s\big|\hat \xi_{k_2}\big|.
$$
Then we have that
\begin{align*}
\Big\|\langle k\rangle^{s+\gamma}|I_{2221,k}|\Big\|_{l_k^2(|k|\geq N_0)}
\lesssim&\sum\limits_{\substack{k,k_3\\ |k|\geq N_0}}\zeta_{k_3,k}^1\>\zeta_{k_3,k}^2 \> |\langle k_3\rangle^s|\hat v_{k_3} | \\
\lesssim&\sum_{k_3}\langle k_3\rangle^s|\hat v_{k_3} |\cdot \left(\sum_{k:|k|\geq N_0}\left(\zeta_{k_3,k}^1\right)^{p'}\right)^{\frac 1{p'}}\cdot\left(\sum_{k:|k|\geq N_0}(\zeta_{k_3,k}^2)^{p}\right)^{\frac 1{p}}.
\end{align*}
For convenience, we further denote
$$
J_{k_3}=\Big\|\zeta_{k_3,k}^1\Big\|_{l^{p'}_k (|k|\geq N_0)},\quad
\tilde{J}_{k_3}=\Big\|\zeta_{k_3,k}^2\Big\|_{l_k^p (|k|\geq N_0)}.
$$
Then it gives that
\begin{align}\label{711-2}
\Big\|\langle k\rangle^{s+\gamma}|I_{2221,k}|\Big\|_{l_k^2(|k|\geq N_0)}
\lesssim
&\sum_{k_3}\langle k_3\rangle^s|\hat v_{k_3}|\>J_{k_3}\>\tilde{J}_{k_3},
\end{align}
Therefore, it reduces to estimate $J_{k_3}$ and $\tilde{J}_{k_3}$.

For the term $J_{k_3}$,
\begin{align}\label{711-3}
 J_{k_3}
 \lesssim &\|h_k\|_{l_k^2}\Big\|\langle k\rangle^{-2+\varepsilon_0+\gamma}\left[\frac 1 {\langle k-k_3\rangle^{1+\frac {1}{2}\varepsilon_0}}+\frac 1 {\langle k+k_3\rangle^{1+\frac {1}{2}\varepsilon_0}}\right]\Big\|_{l_k^r(|k|\geq N_0)}
  \lesssim  N_0^{-\varepsilon_0}\|h\|_{l^2},
 \end{align}
where we have used the relationship that $r$ satisfies that $\frac 1r=\frac 12-\frac 1p$ and $(\gamma-2)r<-1$.

For the term $\tilde{J}_{k_3}$, we change the variable and write
\begin{align}\label{711-4}
\tilde{J}_{k_3}=&\Big\|\sum_{k_2}\Big( \frac 1 {\langle k-(k_2+k_3)\rangle^{1+\frac {1}{2}\varepsilon_0}}+\frac 1 {\langle k+k_2+k_3\rangle^{1+\frac {1}{2}\varepsilon_0}}\Big)\langle k-k_2-k_3\rangle^s|\hat \xi_{k-k_2-k_3} |\langle k_2\rangle^s\big|\hat \xi_{k_2}\big|\Big\|_{l_k^p}\nonumber\\
\lesssim& \Big\|\sum_{k_1}\frac 1 {\langle k_1\rangle^{1+\frac {1}{2}\varepsilon_0}}\langle k_1\rangle^s|\hat \xi_{k_1} |\langle k-k_1-k_3\rangle^s|\hat \xi_{k-k_1-k_3}|\Big\|_{l_k^p}\nonumber\\
&+\Big\|\sum_{\tilde{k}_1}\frac 1 {\langle\tilde{k}_1\rangle^{1+\frac {1}{2}\varepsilon_0}}\langle 2k-\tilde{k}_1\rangle^s|\hat \xi_{2k-\tilde{k}_1} |\langle \tilde{k}_1-k-k_3\rangle^s|\hat \xi_{\tilde{k}_1-k-k_3}|\Big\|_{l_k^p}\nonumber\\
\lesssim&\|\langle k\rangle^s\hat \xi_{k}\|_{l_k^p}\|\langle k\rangle^s\hat \xi_{k}\|_{l_k^{\infty}}
\lesssim\|\xi\|_{\hat{b}^{s,p}}^2.
\end{align}
Hence, by the estimates \eqref{711-2}-\eqref{711-4}, we have
\begin{align*}
\Big\|\langle k\rangle^{s+\gamma}|I_{2221,k}|\Big\|_{l_k^2(|k|\geq N_0)}\lesssim& N_0^{-\varepsilon_0}\|h\|_{l^2}\|\xi\|_{\hat{b}^{s,p}}^2\sum_{k_3} \langle k_3\rangle^s|\hat v_{k_3} |
\lesssim N_0^{-\varepsilon_0}\|h\|_{l^2}\|\xi\|_{\hat{b}^{s,p}}^2\|v\|_{H^{s+\gamma}}.
\end{align*}

 {\emph{2) On $I_{2222,k}$.}} Similarly as \eqref{phi-20-3}, we have that
\begin{align}\label{phi-234}
\frac{1}{|\phi(k,k_2+k_3+k_4)||\phi(k,k_3+k_4)|}\lesssim &\langle k\rangle^{-2+\varepsilon_0}\left[\frac 1 {\langle k_1\rangle^{1+\frac {1}{2}\varepsilon_0}}+\frac 1 {\langle 2k-k_1\rangle^{1+\frac {1}{2}\varepsilon_0}}\right]\notag\\
&\cdot \left[\frac 1 {\langle k-(k_3+k_4)\rangle^{1+\frac {1}{2}\varepsilon_0}}+\frac 1 {\langle k+k_3+k_4\rangle^{1+\frac {1}{2}\varepsilon_0}}\right].
\end{align}
Then by the duality, \eqref{phi-234} and Cauchy-Schwartz's inequality, we have that (again, from the second line we omit  $\sup\limits_{h:\|h\|_{l^2}=1}$ in the front)
\begin{align*}
&\Big\|\langle k\rangle^{s+\gamma}|I_{2222,k}|\Big\|_{l_k^2(|k|\geq N_0)}
=\sup\limits_{h:\|h\|_{l^2}=1}\big\langle\langle k\rangle^{s+\gamma}|I_{2222,k}|, h_k\big\rangle \\
\lesssim&\sum\limits_{\substack{k_1+k_2+k_3+k_4=k\\ \phi(k,k_2+k_3+k_4)\ne 0,\phi(k,k_3+k_4)\ne 0}}\int_0^t
\frac{\langle k\rangle^{s+\gamma}\big|\hat \xi_{k_1}\big|| \hat \xi_{k_2}| | \hat \xi_{k_3}|
|\hat v_{k_4}(\rho)||h_k|}{\phi(k,k_2+k_3+k_4)\cdot \phi(k,k_3+k_4)}d\rho  \\
\lesssim&\sum\limits_{\substack{k_1+k_2+k_3+k_4=k\\ \phi(k,k_2+k_3+k_4)\ne 0,\phi(k,k_3+k_4)\ne 0}}\int_0^t\langle k\rangle^{-2+\varepsilon_0+\gamma}\left[\frac 1 {\langle k-(k_3+k_4)\rangle^{1+\frac {1}{2}\varepsilon_0}}+\frac 1 {\langle k+k_3+k_4\rangle^{1+\frac {1}{2}\varepsilon_0}}\
\right]\\
&\qquad \cdot \left[\frac 1 {\langle k_1\rangle^{1+\frac {1}{2}\varepsilon_0}}+\frac 1 {\langle 2k-k_1\rangle^{1+\frac {1}{2}\varepsilon_0}}\right] \langle k_1\rangle^s|\hat \xi_{k_1} |\langle k_2\rangle^s|\hat \xi_{k_2} |\langle k_3\rangle^s|\hat \xi_{k_3} |\langle k_4\rangle^s|\hat v_{k_4} | |h_k|d\rho,
\end{align*}
where in the last step we have used the inequality
$
\langle k\rangle^s\le \langle k_1\rangle^s\langle k_2\rangle^s\langle k_3\rangle^s.
$
Then we further get that
\begin{align*}
 &\Big\|\langle k\rangle^{s+\gamma}|I_{2222,k}|\Big\|_{l_k^2(|k|\geq N_0)}\\
\lesssim&\int_0^t\sum_{k_4}\sum_{k}\sum_{k_3}\sum_{k_1}\left[\frac 1 {\langle k_1\rangle^{1+\frac {1}{2}\varepsilon_0}}+\frac 1 {|2k-k_1\rangle^{1+\frac {1}{2}\varepsilon_0}}\right]\langle k_1\rangle^s\big|\hat \xi_{k_1}\big|\langle k-k_1-k_3-k_4\rangle^s|\hat \xi_{k-k_1-k_3-k_4} |\\
&\qquad \cdot \left[\frac 1 {\langle k-(k_3+k_4)\rangle^{1+\frac {1}{2}\varepsilon_0}}+\frac 1 {\langle k+k_3+k_4\rangle^{1+\frac {1}{2}\varepsilon_0}}\
\right]\langle k_3\rangle^s|\hat \xi_{k_3} |\langle k_4\rangle^s|\hat v_{k_4} | |h_k|  \langle k\rangle^{-2+\varepsilon_0+\gamma}d\rho\\
\lesssim&\big\|\langle k\rangle^s|\hat \xi_{k} |\big\|_{l_k^{\infty}}^2\int_0^t\sum_{k_4}\sum_{k}\sum_{k_3} \left[\frac 1 {\langle k-(k_3+k_4)\rangle^{1+\frac {1}{2}\varepsilon_0}}+\frac 1 {\langle k+k_3+k_4\rangle^{1+\frac {1}{2}\varepsilon_0}}\
\right]\\
&\qquad\qquad\qquad\quad \cdot\langle k_3\rangle^s|\hat \xi_{k_3} |\langle k_4\rangle^s|\hat v_{k_4} | |h_k|  \langle k\rangle^{-2+\varepsilon_0+\gamma}d\rho.
\end{align*}
We denote
\begin{align*}
M_k=\sum_{k_3} \left[\frac 1 {\langle k-(k_3+k_4)\rangle^{1+\frac {1}{2}\varepsilon_0}}+\frac 1 {\langle k+k_3+k_4\rangle^{1+\frac {1}{2}\varepsilon_0}}\
\right]\langle k_3\rangle^s|\hat \xi_{k_3} |,
\end{align*}
then we have that
\begin{align}
\Big\|\langle k\rangle^{s+\gamma}|I_{2222,k}|\Big\|_{l_k^2(|k|\geq N_0)}
\label{711-6}
\lesssim&\big\|\langle k\rangle^s|\hat \xi_{k} |\big\|_{l_k^{\infty}}^2\int_0^t\sum_{k_4}\sum_{k} M_k \langle k_4\rangle^s|\hat v_{k_4} | |h_k| \langle k\rangle^{-2+\varepsilon_0+\gamma}d\rho.
\end{align}
Now, we estimate the term $M_k$ as follows
\begin{align*}
M_k\lesssim &\Bigg\|\left[\frac 1 {\langle k-(k_3+k_4)\rangle^{1+\frac {1}{2}\varepsilon_0}}+\frac 1 {\langle k+k_3+k_4\rangle^{1+\frac {1}{2}\varepsilon_0}}\
\right]^{\frac 1{p'}}\Bigg\|_{l_{k_3}^{p'}}\\
&\cdot  \Bigg\|\left[\frac 1 {\langle k-(k_3+k_4)\rangle^{1+\frac {1}{2}\varepsilon_0}}+\frac 1 {\langle k+k_3+k_4\rangle^{1+\frac {1}{2}\varepsilon_0}}\
\right]^{\frac 1p}\langle k_3\rangle^s|\hat \xi_{k_3} |\Bigg\|_{l_{k_3}^{p}}\\
\lesssim &\Bigg\|\left[\frac 1 {\langle k-(k_3+k_4)\rangle^{1+\frac {1}{2}\varepsilon_0}}+\frac 1 {\langle k+k_3+k_4\rangle^{1+\frac {1}{2}\varepsilon_0}}\
\right]^{\frac 1p}\langle k_3\rangle^s|\hat \xi_{k_3} |\Bigg\|_{l_{k_3}^{p}}.
\end{align*}
This gives that
\begin{align}\label{711-7}
\big\| M_k\big\|_{l^p_k}
\lesssim
 &\Bigg\|\left[\frac 1 {\langle k-(k_3+k_4)\rangle^{1+\frac {1}{2}\varepsilon_0}}+\frac 1 {\langle k+k_3+k_4\rangle^{1+\frac {1}{2}\varepsilon_0}}\
\right]^{\frac 1p}\langle k_3\rangle^s|\hat \xi_{k_3} |\Bigg\|_{l_{k, k_3}^{p}}\notag\\
\lesssim
 &\Bigg\|\left\|\frac 1 {\langle k-(k_3+k_4)\rangle^{1+\frac {1}{2}\varepsilon_0}}+\frac 1 {\langle k+k_3+k_4\rangle^{1+\frac {1}{2}\varepsilon_0}}\
\right\|_{l_k^1}^{\frac 1p}\langle k_3\rangle^s|\hat \xi_{k_3} |\Bigg\|_{l_{k_3}^{p}}\notag\\
\lesssim &
\|\xi\|_{\hat{b}^{s,p}}.
\end{align}
Insert the estimate \eqref{711-7} into \eqref{711-6}, we have
\begin{align*}
\Big\|\langle k\rangle^{s+\gamma}|I_{2222,k}|\Big\|_{l_k^2(|k|\geq N_0)}
\lesssim&\big\|\langle k\rangle^s|\hat \xi_{k} |\big\|_{l_k^{\infty}}^2\int_0^t\sum_{k_4} \big\| M_k\big\|_{l^p_k}\Big\|\langle k\rangle^{-2+\varepsilon_0+\gamma}\Big\|_{l_k^r} \|h\|_{l^2}
\langle k_4\rangle^s|\hat v_{k_4} |d\rho\\
\lesssim &\|\xi\|_{\hat{b}^{s,\infty}}^2\|\xi\|_{\hat{b}^{s,p}} \int_{0}^t\sum_{k_4}\langle k_4\rangle^s|\hat v_{k_4} |d\rho
\lesssim  t\|\xi\|_{\hat{b}^{s,p}}^3\|v\|_{L_t^{\infty}H^{s+\gamma}_x},
\end{align*}
where in the last step we have used that $\gamma>\frac12$.

{ \emph{3) On $I_{2223,k}$.}} Similarly as treating the term $I_{2222, k}$, we can obtain
\begin{align*}
\Big\|\langle k\rangle^{s+\gamma}|I_{2223,k}|\Big\|_{l_k^2(|k|\geq N_0)}
\lesssim & t\|\xi\|_{\hat{b}^{s,p}}^2\||u|^2u\|_{L_t^{\infty}H_x^{s+\gamma}}
\lesssim  t\|\xi\|_{\hat{b}^{s,p}}^2\|u\|_{L_t^{\infty}H_x^{s+\gamma}}^3.
\end{align*}
Hence, we obtain the estimate on $I_{22,k}$ as follows
\begin{align*}
\Big\|\langle k\rangle^{s+\gamma}|I_{22,k}|\Big\|_{l_k^2(|k|\geq N_0)}\lesssim &
t\big\|\hat \xi_k\big\|_{l^\infty_k}^2\|v\|_{L^\infty_tH^{s+\gamma}_x}^2+ N_0^{-\varepsilon_0}\|\xi\|_{\hat{b}^{s,p}}^2\|v\|_{L^\infty_tH^{s+\gamma}_x}\\
&+t\|\xi\|_{\hat{b}^{s,p}}^3\|v\|_{L^\infty_tH^{s+\gamma}_x}+t\|\xi\|_{\hat{b}^{s,p}}^2\|u\|_{L^\infty_tH^{s+\gamma}_x}^3.
\end{align*}
This completes the estimates of $I_{222}$ and then finishes \cref{subsub1}.

\subsubsection{Estimates on $I_{23,k}$.} The estimation of  this term is very similar as the term $I_{21}$, so we almost repeat the estimates of the latter. Firstly, we  split it into the following two parts:
\begin{align*}
I_{23,k}= & - i\int_0^t\sum\limits_{\substack{k_1+k_2=k\\ \phi(k,k_2)\ne 0,|k_1|\lesssim |k_2|}}\fe^{i\rho k^2}\frac{1}{i\phi(k,k_2)}
\hat \xi_{k_1} \widehat{\big(|u|^2u\big)}_{k_2}\,d\rho\\
&   - i\int_0^t\sum\limits_{\substack{k_1+k_2=k\\ \phi(k,k_2)\ne 0,|k_1|\gg |k_2|}}\fe^{i\rho k^2}\frac{1}{i\phi(k,k_2)}
\hat \xi_{k_1} \widehat{\big(|u|^2u\big)}_{k_2}\,d\rho\\
\triangleq &
I_{231,k}+I_{232,k}.
\end{align*}

$\bullet$ {\bf Estimates on $I_{231,k}$.}  Since
$$
\frac{1}{|\phi(k,k_2)|}
\le  \frac1{2\langle k_1\rangle^2}+\frac1{2\langle k_1+2k_2\rangle^2}.
$$
Hence,   we have that
\begin{align*}
\langle k\rangle^{s+\gamma}|I_{231,k}|\lesssim &\int_0^t \sum\limits_{\substack{k_1+k_2=k\\ \phi(k,k_2)\ne 0,|k_1|\lesssim|k_2| }}\langle k\rangle^{s+\gamma}\frac 1{|\phi(k,k_2)|}\big|\hat \xi_{k_1}\big|\Big|\widehat{\big(|u|^2u\big)}_{k_2}\Big|\,d\rho\\
\lesssim &\int_0^t \sum\limits_{\substack{k_1+k_2=k\\ \phi(k,k_2)\ne 0,|k_1|\lesssim|k_2| }}\langle k_2\rangle^{s+\gamma}\left[\frac {1}{\langle k_1\rangle^2}  +\frac {1}{\langle k_1+2k_2\rangle^2}\right]\big|\hat \xi_{k_1}\big|\Big|\widehat{\big(|u|^2u\big)}_{k_2}\Big|\,d\rho.
\end{align*}
Then treating similarly as $I_{211,k}$ and by Kato-Ponce's inequality, we have that
\begin{align*}
\big\|\langle k\rangle^\gamma I_{231,k}\big\|_{l^2_k}
\lesssim &
\int_0^t \big\|\hat \xi_k\big\|_{l^\infty_k}\big\||u|^2u\big\|_{L^\infty_tH^{s+\gamma}_x}\,d\rho
\lesssim
t \big\|\hat \xi_k\big\|_{l^\infty_k}\|u\|_{L^\infty_tH^{s+\gamma}_x}^3.
\end{align*}

$\bullet$ {\bf Estimates on $I_{232,k}$.}  Under the restriction in the summation, it gives  that
$$|\phi(k,k_2)|\gtrsim  \langle k_1\rangle^2.
$$
Then we have that
\begin{align*}
\langle k\rangle^{s+\gamma}|I_{232,k}|\lesssim
\int_0^t\sum\limits_{\substack{k_1+k_2=k\\ \phi(k,k_2)\ne 0,|k_1|\gg|k_2| }}\langle k_1\rangle^{\gamma+s-2}\big|\hat \xi_{k_1}\big|\Big|\widehat{\big(|u|^2u\big)}_{k_2}\Big|\,d\rho.
\end{align*}
Again, treating similarly as $I_{211}$ and by Kato-Ponce's inequality, we have that
\begin{align*}
\big\|\langle k\rangle^{s+\gamma} I_{232,k}\big\|_{l^2_k}
\lesssim &
\int_0^t \big\|\xi\big\|_{\hat b^{s,p}}\big\||u|^2u\big\|_{L^\infty_tH^{s+\gamma}_x}\,d\rho
\lesssim
t  \big\|\xi\big\|_{\hat b^{s,p}}\|u\|_{L^\infty_tH^{s+\gamma}_x}^3.
\end{align*}
Combining with the estimates on $I_{231}$ and $I_{232}$, we obtain that
\begin{align*}
\big\|\langle k\rangle^{s+\gamma} I_{23,k}\big\|_{l^2_k}
\lesssim &
t  \big\|\xi\big\|_{\hat b^{s,p}}\|u\|_{L^\infty_tH^{s+\gamma}_x}^3.
\end{align*}

Collecting the estimates on $I_1$ and $I_{21}-I_{23}$ above, we have that there exists a constant $C>0$ which depends on $\|\xi\|_{\hat b^{s,p}}$ such that
\begin{align*}
 \big\|\langle k\rangle^{s+\gamma} I_k\big\|_{L^\infty_tl^2_k([0,T]\times \{|k|> N_0\})}
\le &
C\big(T+N_0^{-\frac{\varepsilon_0}{2}})\big(\|v\|_{L^\infty_t  H^{s+\gamma}_x}+\|v\|_{L^\infty_t  H^{s+\gamma}_x}^3\big).
\end{align*}
This together with \eqref{est:low-frq} give that
\begin{align*}
\left\|\Phi^1(u)\right\|_{L^\infty_t H^{s+\gamma}_x([0,T])}
= &
\big\|\langle k\rangle^{s+\gamma} I_k\big\|_{L^\infty_t  l^2_k}\\
\le &
C\Big(TN_0^{s+\gamma}+N_0^{-\frac{\varepsilon_0}{2}}\Big)\big(\|v\|_{L^\infty_t  H^{s+\gamma}_x}+\|v\|_{L^\infty_t  H^{s+\gamma}_x}^3\big).
\end{align*}
This finishes the proof of \cref{lem:Phi-i}.
\end{proof}

\begin{proof}[Proof of \cref{lem:Phi}]
By Kato-Ponce's inequality, we have
\begin{align*}
\left\|i\int_0^t\fe^{-it'\partial_x^2}\Big( |u(t')|^2u(t')\Big)\,dt'\right\|_{L^\infty_t H^{s+\gamma}_x([0,T])}
\le &
CT\|u\|_{L^\infty_t  H^{s+\gamma}_x}^3.
\end{align*}
This combining with \eqref{def:Phi} and  \cref{lem:Phi-i} yields that under the assumption  \eqref{assum},
\begin{align*}
\big\|\Phi(v)\big\|_{L^\infty_t H^{s+\gamma}_x([0,T])}
\le & \|u^0\|_{H^{s+\gamma}_x}+C\big(TN_0^{s+\gamma}+N_0^{-\frac{\varepsilon_0}{2}})\big(R_0+R_0^3\big)\\
\le & \frac12 R_0+C\big(TN_0^{s+\gamma}+N_0^{-\frac{\varepsilon_0}{2}})\big(R_0+R_0^3\big).
\end{align*}
This proves the lemma.
\end{proof}

Now we choosing $N_0$ and $T$ in   \cref{lem:Phi} such that
$$
N_0^{\frac{\varepsilon_0}{2}}= 8\big(1+ R_0^2\big);\quad
T =\big[8 C_0(N_0)\big(1+ R_0^2\big)\big]^{-1},
$$
then we obtain \eqref{est:bounded} and \eqref{est:contraction} with $\theta=\frac34$, and thus complete the proof of   \cref{main:thm1}.
\qed

\subsection{Proof of \cref{main:thm3-smoothdata}}

In the following, we denote  $T$ (using the same notation) to be a positive constant which is smaller than the lifespan obtained in  \cref{main:thm1}.

We denote $\varphi=\partial_tu$, then in order to prove that $u\in C_t^0([0,T]; H^{s+\frac32+\frac1p}(\T))$, it is sufficient to
prove that $\varphi\in C_t^0([0,T]; H^{s-\frac12+\frac1p}(\T))$. Indeed, note that
$$
\partial_{xx} u=-i\varphi-\xi u+|u|^2u.
$$
This gives that
\begin{equation*}
\begin{aligned}
\|u\|_{H^{s+\frac32+\frac1p}}
&\lesssim
\|\varphi\|_{H^{s-\frac12+\frac1p}}+\big\|\xi u\big\|_{H^{s-\frac12+\frac1p}}+\big\||u|^2u\big\|_{H^{s-\frac12+\frac1p}}.
\end{aligned}
\end{equation*}
Note that  $H^{s,p'}\hookrightarrow \hat b^{s,p}$, then from the proof of Theorem \ref{main:thm1}, we have that
\begin{align}\label{est:u-}
\|u\|_{L^\infty_tH^{s+\frac32+\frac1p-}_x([0,T])}
\le 2\|u_0\|_{H^{s+2}}.
\end{align}
By Sobolev embedding $L^{p'}\hookrightarrow H^{-\frac12+\frac1p}$ and \eqref{est:u-}, we further have that
\begin{equation}\label{est:u-psi}
\begin{aligned}
\|u\|_{H^{s+\frac32+\frac1p}}
&\lesssim
\|\varphi\|_{H^{s-\frac12+\frac1p}}+\big\|J^s(\xi u)\big\|_{L^{p'}}+\big\|J^s(|u|^2u)\big\|_{L^{p'}}\\
&\lesssim
\|\varphi\|_{H^{s-\frac12+\frac1p}}
	 + \big\|J^s\xi \big\|_{L^{p'}}\cdot\|u\|_{L^\infty}\\
	 &\quad +\|\xi\|_{L^{p'}}\cdot\big\|J^s u\big\|_{L^{\infty}}+\big\|J^su\big\|_{L^{p'}}\|u\|_{L^\infty}^2\\
&\lesssim
\|\varphi\|_{H^{s-\frac12+\frac1p}}
	 + \big\|J^s\xi \big\|_{L^{p'}}\cdot\|u\|_{H^{s+\frac32+\frac1p-}}\\
	 &\quad +\|\xi\|_{L^{p'}}\cdot\big\|u\big\|_{H^{s+\frac32+\frac1p-}}+\|u\|_{H^{s+\frac32+\frac1p-}}^3\\
&\lesssim
\|\varphi\|_{H^{s-\frac12+\frac1p}}
	 + \big\|\xi \big\|_{H^{s,p'}}\cdot\|u_0\|_{H^{s+2}} +\|u_0\|_{H^{s+2}}^3.
\end{aligned}
\end{equation}
This proves the claim that  $\varphi\in C_t^0([0,T]; H^{s-\frac12+\frac1p}(\T))$ implies $u\in C_t^0([0,T]; H^{s+\frac32+\frac1p}(\T))$.

Moreover, from \eqref{model},  $\varphi $ satisfies the following equation
	\begin{equation}\label{model-ut}
		\left\{\begin{aligned}
			& i\partial_t\varphi+\partial_{xx} \varphi+\xi \varphi=\mathcal{O}(u^2\varphi),
			&&\mbox{for}\,\,\, x\in\T\,\,\,\mbox{and}\,\,\, t\in(0,T] , \\
			&\varphi(0,x)=i(\partial_{xx}u_0+\xi u_0-|u_0|^2u_0)\triangleq \varphi_0, &&\mbox{for}\,\,\,  x\in\T.
		\end{aligned}\right.
	\end{equation}
Here we denote $\mathcal{O}(u^2\varphi)=2|u|^2\varphi+u^2 \bar \varphi$.
From Duhamel's formula, we have
\begin{align}
	\label{solution-v}
	\varphi(t)=\fe^{it\partial_x^2}\varphi_0+i\int_{0}^t\fe^{i(t-\rho)\partial_x^2}\Big(\xi \varphi(\rho)+\mathcal{O}\big(u^2\varphi(\rho)\big)\Big)\,d\rho.
\end{align}
Accordingly, we denote that
$$
\Psi(\varphi)=
\fe^{it\partial_x^2}\varphi_0+i\int_{0}^t\fe^{i(t-\rho)\partial_x^2}\Big(\xi \varphi(\rho)+\mathcal{O}\big(u^2\varphi(\rho)\big)\Big)\,d\rho.
$$
In this proof, we denote $a=\frac12-\frac1p$ for short. Firstly, similarly as \eqref{est:u-psi}, we have
\begin{align}
	\|\varphi_0\|_{H^{-a+s}} \lesssim \|u_0\|_{H^{s-a+2}}
	 + \big\|\xi \big\|_{H^{s,p'}}\cdot\|u_0\|_{H^{s+2}} +\|u_0\|_{H^{s+2}}^3.
\end{align}
Hence, we have $\varphi_0 \in H^{-a+s}.$
Similar as  the proof of   \cref{main:thm1}, the proof of Theorem \ref{main:thm3-smoothdata} is reduced to the following lemma.
\begin{lemma}
There exist a constant $\varepsilon_0$ and a positive function $C(\cdot)$ such that  for any $N_0>0$, any $t\in [0,T]$,
\begin{align}\label{est:Psi}
	\big\|\Psi(\varphi)\big\|_{H^{-a+s}}\leqslant \big\|\varphi_0\big\|_{H^{-a+s}}+\big[tC(N_0)+N_0^{-\varepsilon_0}\big]\cdot\big(\big\|\varphi\big\|_{L^\infty_tH^{-a+s}_x}+\big\|\varphi\big\|^3_{L^\infty_tH^{-a+s}_x}\big).
\end{align}
\end{lemma}
\begin{proof}
We split it into the following two parts.

  $\bullet $ {\textbf{Estimate on} $\big\|\int_{0}^t\fe^{i(t-\rho)\partial_x^2}[\xi \varphi(\rho)]\,d\rho\big\|_{H^{-a+\rho}}.$}
Denote $w=e^{-it\partial_x^2}\varphi, $ then it reduces to
\begin{align}
	I\triangleq \int_{0}^t\fe^{-i\rho\partial_x^2}[\xi \cdot \fe^{i\rho\partial_x^2}w(\rho)]\,d\rho.
\end{align}
Since $H^{s,p'}\hookrightarrow \hat b^{s,p}$, arguing similarly as the proof of \cref{lem:Phi-i}, we have that (for which the details are omitted here)
\begin{align}\label{est:I}
	\big\|I\big\|_{H^{-a+s}}
	\lesssim &\big[tC(N_0)+N^{-\varepsilon_0}\big]\cdot\big[\|w\|_{L^\infty_tH^{-a+s}_x}+\|w\|^3_{L^\infty_tH^{-a+s}_x}\big]\notag\\
	=
	&\big[tC(N_0)+N^{-\varepsilon_0}\big]\cdot\big[\|\varphi\|_{L^\infty_tH^{-a+s}_x}+\|\varphi\|^3_{L^\infty_tH^{-a+s}_x}\big].
\end{align}

$\bullet $ {\textbf{Estimate on} $\big\|\int_{0}^t\fe^{i(t-\rho)\partial_x^2}[\mathcal{O}(u^2\varphi)]\,d\rho\big\|_{H^{-a+\rho}}.$}
 If $-a+s\ge 0$,  then by Kato-Ponce's Inequality and \eqref{est:u-},
 \begin{align}\label{est:a-s+0}
 	\Big\|\int_{0}^t\fe^{i(t-\rho)\partial_x^2}[\mathcal{O}(u^2\varphi)]\,d\rho\Big\|_{H^{-a+s}}
 	\lesssim &\int_{0}^t\big\|J^{-a+s}(u^2\varphi)\big\|_{L^2}\,d\rho\notag\\
	\lesssim &\|\varphi\|_{L^\infty_tH^{-a+s}_x}\|u\|_{L^\infty_tH^{2-a+s-}_x}^2\notag \\
	\lesssim & \|u_0\|_{H^{s+2}}^2\|\varphi\|_{L^\infty_tH^{-a+s}_x} .
 \end{align}
If $-a+s<0$,
using dual argument, we note that
\begin{align}\label{dual}
	\big\|J^{-a+s}(u^2\varphi)\big\|_{L^2}=\sup_{h:\|h\|_{l^2}=1}\big\langle J^{-a+s}(u^2\varphi), h\big\rangle.
\end{align}
By H\"{o}lder's inequality, we have
\begin{align}
	&
	\big\langle J^{-a+s}(u^2\varphi), h\big\rangle\nonumber\\
	=&\sum\limits_N\langle P_NJ^{-a+s}(u^2\varphi), P_Nh\big\rangle
	\lesssim
	\sum\limits_N	N^{-a+s}\big\|P_N(u^2\varphi)\big\|_{L^2}\cdot\big\|P_Nh\big\|_{L^2}
	\notag\\
	\lesssim &
	\sum\limits_{N\gtrsim M} N^{-a+s}\big\|P_N(u^2P_M \varphi)\big\|_{L^2}\cdot\big\|P_Nh\big\|_{L^2}
	+\sum\limits_{N\ll M} N^{-a+s}\big\|P_N(u^2P_M \varphi)\big\|_{L^2}\cdot\big\|P_Nh\big\|_{L^2}\notag\\
	\triangleq & S_1+S_2.\label{est:s1-s2}
\end{align}
For $S_1$, by Schur's test in \cref{lem:schurtest} and \eqref{est:u-}, we have
\begin{align}\label{est:s1}
	S_1&=\sum\limits_{N\gtrsim M} N^{-a+s}\big\|P_N(u^2P_M \varphi)\big\|_{L^2}\cdot\big\|P_Nh\big\|_{L^2}
	\notag\\
	&\lesssim
	\sum\limits_{N\gtrsim M} N^{-a+s} \|u\|_{H^{\frac12+}}^2\big\|P_M\varphi\big\|_{L^2}\cdot\big\|P_Nh\big\|_{L^2}
	\notag\\
	&\lesssim
	\|u_0\|_{H^{s+2}}^2\sum\limits_{N\gtrsim M}\left(\frac{M}{N}\right)^{-a+s}\big\|P_M\varphi\big\|_{H^{-a+s}}\big\|P_Nh\big\|_{L^2}
	\notag\\
	&\lesssim
	\|u_0\|_{H^{s+2}}^2\big\|\varphi\big\|_{H^{-a+s}}\big\|h\big\|_{L^2}.
\end{align}
For $S_2$, note that when $N\ll M$,
$
P_N(u^2P_M \varphi)=P_N\big(u P_{>\frac M2}u P_M \varphi\big).
$
Then using H\"{o}lder's inequality and \eqref{est:u-}, we have
\begin{align*}
	S_2
	&=\sum\limits_{N\ll M} N^{-a+s}\big\|P_N(u^2P_M \varphi)\big\|_{L^2}\cdot\big\|P_Nh\big\|_{L^2}
	\\
	&\lesssim
	\sum\limits_{N\ll M} N^{-a+s}\|u\|_{L^\infty}\big\|P_{>\frac{M}{2}}u\big\|_{L^\infty}\big\|P_M\varphi\big\|_{L^2}\big\|P_Nh\big\|_{L^2}
	\\
	&\lesssim
	\sum\limits_{N\ll M}N^{-a+s}\big\|u\big\|_{H^{\frac{1}{2}+}}\big\|P_{>\frac{M}{2}}u\big\|_{H^{\frac{1}{2}+}}\big\|P_M\varphi\big\|_{L^2}\big\|P_Nh\big\|_{L^2}
	\\
	&\lesssim
	\sum\limits_{N\ll M}M^{-2(s+\frac1p)-\frac12+}N^{-a+s}\big\|u\big\|_{H^{\frac{1}{2}+}}\big\|P_{>\frac{M}{2}}u\big\|_{H^{s+\frac1p+\frac32+}}\big\|P_M\varphi\big\|_{H^{-a+s}}\big\|P_Nh\big\|_{L^2}
	\\
	&\lesssim
	\sum\limits_{N\ll M}M^{-2(s+\frac1p)-\frac12+}N^{-a+s}\|u_0\|_{H^{s+2}}^2\big\|\varphi\big\|_{H^{-a+s}}\big\|h\big\|_{L^2}.
\end{align*}
Note that $-2(s+\frac1p)-\frac12<0, -a+s<0$, by Cauchy-Schwartz's inequality, we obtain that 	
\begin{align}
\label{est:s2}
	S_2
	&\lesssim
	\|u_0\|_{H^{s+2}}^2\big\|\varphi\big\|_{H^{-a+s}}\big\|h\big\|_{L^2}.
\end{align}
Inserting  \eqref{est:s1} and \eqref{est:s2} into \eqref{est:s1-s2}, and then by \eqref{dual} we have that
$$
\big\|J^{-a+s}(u^2\varphi)\big\|_{L^2}\lesssim \|u_0\|_{H^{s+2}}^2\big\|\varphi\big\|_{H^{-a+s}}.
$$
This implies that
\begin{align*}
 	\Big\|\int_{0}^t\fe^{-i(t-\rho)\partial_x^2}[\mathcal{O}(u^2\varphi)]\,d\rho\Big\|_{H^{-a+s}}
 	\lesssim &\int_{0}^t\Big\|\fe^{-i(t-\rho)\partial_x^2}[\mathcal{O}(u^2\varphi)]\Big\|_{H^{-a+s}}\,d\rho\notag\\
	\lesssim & t\|u_0\|_{H^{s+2}}^2\|\varphi\|_{L^\infty_tH^{-a+s}_x} .
 \end{align*}
 This together with \eqref{est:a-s+0} and \eqref{est:I}, yields \eqref{est:Psi}.
 \end{proof}

 \vskip1cm

\section{Proof for ill-posedness theory}\label{sec: ill}
In this section, we are going to prove the three ill-posedness results, i.e., \cref{main:thm1-ill}, \cref{main:thm2-ill} and \cref{main:thm3}. Again, we assume $\lambda=-1$ in (\ref{model}) for simplicity. In order to apply the tool \cref{lm:ill-tool}, we consider the following.

Let $f=f(x)$ be a time-independent function. Denote
\begin{align*}
A_2(f)\triangleq &
i\int_0^t\fe^{-i\rho\partial_x^2}\Big(\xi \fe^{i\rho\partial_x^2}f+\big|\fe^{i\rho\partial_x^2}f\big|^2\fe^{i\rho\partial_x^2}f\Big)\,d\rho.
\end{align*}
Taking Fourier transform, we have that for any $k\in \Z$,
\begin{align*}
\widehat{\big(A_2(f)\big)}_k
=&
i\int_0^t\Big(\sum\limits_{k_1+k_2=k}\fe^{i\rho(k^2-k_2^2)}\hat\xi_{k_1} \hat f_{k_2}+\sum\limits_{k_1+k_2+k_3=k}\fe^{i\rho(k^2+k_1^2-k_2^2-k_3^2)} \widehat{(\overline{f})}_{k_1} \hat f_{k_2} \hat f_{k_3}\Big)\,d\rho\\
=&
i\sum\limits_{k_1+k_2=k}\int_0^t\fe^{i\rho(k^2-k_2^2)}\,d\rho\> \hat\xi_{k_1} \hat f_{k_2}\\
& +i \sum\limits_{k_1+k_2+k_3=k}\int_0^t\fe^{i\rho(k^2+k_1^2-k_2^2-k_3^2)} \,d\rho\>\widehat{(\overline{f})}_{k_1} \hat f_{k_2} \hat f_{k_3}.
\end{align*}
The first term is equal to
\begin{align*}
i\sum\limits_{\substack{k_1+k_2=k\\ k^2-k_2^2=0}} \int_0^t\fe^{i\rho(k^2-k_2^2)}\,d\rho\> \hat\xi_{k_1} \hat f_{k_2}
+i\sum\limits_{\substack{k_1+k_2=k\\ k^2-k_2^2\ne 0}} \int_0^t\fe^{i\rho(k^2-k_2^2)}\,d\rho\> \hat\xi_{k_1} \hat f_{k_2}.
\end{align*}
Note that the set
\begin{align*}
&\{(k_1,k_2): k_1+k_2=k, k^2-k_2^2=0\}\\
=&
\{(k_1,k_2): k_1=0, k_2=k\}
\cup
\{(k_1,k_2): k_1=2k, k_2=-k\}\setminus
\{(k_1,k_2): k_1=k_2=0\}.
\end{align*}
We further get that
\begin{align*}
&i\sum\limits_{k_1+k_2=k} \int_0^t\fe^{i\rho(k^2-k_2^2)}\,d\rho\> \hat\xi_{k_1} \hat f_{k_2}\\
=&
it \hat\xi_{0} \hat f_{k}+it \hat\xi_{2k} \hat f_{-k}-it \hat\xi_0 \hat f_0
+i\sum\limits_{\substack{k_1+k_2=k\\ k^2-k_2^2\ne 0}} \frac{1}{i(k^2-k_2^2)}\left(\fe^{it(k^2-k_2^2)}-1\right)\> \hat\xi_{k_1} \hat f_{k_2}.
\end{align*}
This implies that
\begin{align}
\widehat{\big(A_2(f)\big)}_k
=&
it \hat\xi_{0} \hat f_{k}+it \hat\xi_{2k} \hat f_{-k}-it \hat\xi_0 \hat f_0\notag\\
& +i\sum\limits_{\substack{k_1+k_2=k\\ k^2-k_2^2\ne 0}} \frac{1}{i(k^2-k_2^2)}\Big(\fe^{it(k^2-k_2^2)}-1\Big)\> \hat\xi_{k_1} \hat f_{k_2}\notag\\
&  +i \sum\limits_{k_1+k_2+k_3=k}\int_0^t\fe^{i\rho(k^2+k_1^2-k_2^2-k_3^2)} \,d\rho\>\widehat{(\overline{f})}_{k_1} \hat f_{k_2} \hat f_{k_3}.\label{Form-A2}
\end{align}

\subsection{Proof of  \cref{main:thm1-ill}}

\

$\bullet$ {\bf Case 1: $p=\infty$.}
We only need to show the ill-posedness in $H^{s+\frac32}$.
To do this, we define the initial data
\begin{align}\label{def:intialdata-1}
u_0\triangleq \epsilon .
\end{align}
Then
$
\|u_0\|_{H^2}= \epsilon.
$
Moreover, we define
\begin{align}\label{def:potential-1}
\xi\triangleq  1+ \sum\limits_{k\in \Z\setminus \{0\}}|k|^{-s} \fe^{ikx}.
\end{align}
Then
$
\|\xi\|_{\hat b^{s,\infty}}=1.
$
In particular, $\xi=2\pi \delta$ when $s=0$.

 Then by \eqref{Form-A2}, in which the first three terms vanish, we have that
\begin{subequations}
\begin{align}
\widehat{\big(A_2(u_0)\big)}_k
=&
i\sum\limits_{k_1+k_2=k} \frac{1}{i(k^2-k_2^2)}\Big(\fe^{it(k^2-k_2^2)}-1\Big)\> \hat\xi_{k_1} \widehat{(u_0)}_{k_2}\label{A2u-1}\\
&\quad +i \sum\limits_{k_1+k_2+k_3=k}\int_0^t\fe^{i\rho(k^2+k_1^2-k_2^2-k_3^2)} \,d\rho\>\widehat{(\overline{u_0})}_{k_1} \widehat{(u_0)}_{k_2} \widehat{(u_0)}_{k_3}. \label{A2u-2}
\end{align}
\end{subequations}

For \eqref{A2u-1}, according to the definition of $u_0$ and $\xi$ we have
\begin{align}\label{A2u-1-1}
\eqref{A2u-1}
=&
\epsilon\frac{1}{k^{2+s}}\Big(\fe^{itk^2}-1\Big).
\end{align}
Now we denote the set
$$
\mathbb K\triangleq \big\{M_0(2j+1): j\in\Z^+, 1\le j\le M_1\big\},
$$
where $M_0,M_1\ge 10$ are some large constants determined later.
Then for $k\in \mathbb K$, we have
\begin{align*}
k^2
=&
M_0^2(4j^2+4j+1).
\end{align*}
Taking
\begin{align}\label{def:time-1}
t\triangleq \frac\pi2 M_0^{-2},
\end{align}
then
$
 tk^2= 2(j^2+j)\pi+\frac\pi2.
$
This yields that
$
\fe^{itk^2}=i.
$
This further gives \eqref{A2u-1-1} that for any $k\in \mathbb K$,
\begin{align}\label{A2u-1-2}
\eqref{A2u-1}
=&
\epsilon(i-1) \frac{1}{k^{2+s}}.
\end{align}
Now we consider the estimate  on \eqref{A2u-1}. By \eqref{A2u-1-2}, it reads
\begin{align*}
\big\|(1+|k|^2)^{\frac12(s+\frac32)} \eqref{A2u-1}\big\|_{l^2_k}^2
\ge &
\big\|(1+|k|^2)^{\frac12(s+\frac32)} \eqref{A2u-1}\big\|_{l^2_k(\mathbb K)}^2 \\
=&\epsilon^2 \sum\limits_{k\in \mathbb K}(1+|k|^2)^{s+\frac32} \Big|(i-1) \frac{1}{k^{2+s}}\Big|^2\\
=&2 \epsilon^2 \sum\limits_{k\in \mathbb K}(1+|k|^2)^{s+\frac32} k^{-(4+2s)}.
\end{align*}
Note that in $\mathbb K$,
$$
(1+|k|^2)^{s+\frac32} k^{-(4+2s)}\ge |k|^{-1}.
$$
Therefore,
we  get  that
\begin{align*}
\big\|(1+|k|^2)^{\frac12(s+\frac32)} \eqref{A2u-1}\big\|_{l^2_k}^2
\ge
&2\epsilon^2 \sum\limits_{k\in \mathbb K}|k|^{-1}
=
2\epsilon^2M_0^{-1} \sum\limits_{j=1}^{M_1}(2j+1)^{-1}
\ge
2\epsilon^2M_0^{-1} \ln (2M_1+1).
\end{align*}

For \eqref{A2u-2}, by Cauchy-Schwartz's inequality, we have
\begin{align}\label{est:A2u-2}
&\big\|(1+|k|^2)^{\frac12(s+\frac32)} \eqref{A2u-2}\big\|_{l^2_k}^2\notag\\
=
&\sum\limits_{k}(1+|k|^2)^{s+\frac32}\Big|\sum\limits_{k_1+k_2+k_3=k}\int_0^t\fe^{i\rho(k^2+k_1^2-k_2^2-k_3^2)} \,d\rho\>\widehat{(\overline{u_0})}_{k_1} \widehat{(u_0)}_{k_2} \widehat{(u_0)}_{k_3}\Big|^2\notag\\
= &
t^2\epsilon^6.
\end{align}

Together with the estimates on \eqref{A2u-1} and \eqref{A2u-2} above, we obtain
\begin{align*}
\big\|A_2(u_0)\big\|_{H^{s+\frac32}}=&
\Big\|(1+|k|^2)^{\frac12(s+\frac32)}\widehat{\big(A_2(u_0)\big)}_k\Big\|_{l^2_k}\\
\ge &
\sqrt2\epsilon M_0^{-\frac12} \sqrt{\ln (2M_1+1)}
-t\epsilon^3\\
= &
\sqrt2\epsilon M_0^{-\frac12} \sqrt{\ln (2M_1+1)}
-\frac\pi 2M_0^{-2}\epsilon^3
\to  +\infty,\quad \mbox{when } M_1\to +\infty.
\end{align*}
Then by Lemma \ref{lm:ill-tool}, we obtain the ill-posedness in $H^{s+\frac32}$, and thus finish the proof of Theorem \ref{main:thm1-ill} in the case of $p=\infty$.

\noindent $\bullet$ {\bf Case 2: $2<p<\infty$.}
As before, we define the initial data
\begin{align*}
u_0\triangleq \epsilon .
\end{align*}
Now, we define
\begin{align}
\xi\triangleq 1+ \sum\limits_{k\in \Z\setminus\{0\}}\frac1{|k|^{s+\frac1p}(\ln |k|)^\alpha} \fe^{ikx},
\end{align}
where $\alpha>\frac1p$ will be decided later.
Then
$$
\|\xi\|_{\hat b^{s,p}}\lesssim 1.
$$
Therefore, the same estimates as \eqref{A2u-1-2} give that
\begin{align*}
\big\|(1+|k|^2)^{\frac12(s+\frac32+\frac1p)} \eqref{A2u-1}\big\|_{l^2_k}^2
\ge &
\big\|(1+|k|^2)^{\frac12(s+\frac32+\frac1p)} \eqref{A2u-1}\big\|_{l^2_k(\mathbb K)}^2 \\
=&2 \epsilon^2 \sum\limits_{k\in \mathbb K}(1+|k|^2)^{s+\frac32+\frac1p} k^{-(4+2s+\frac2p)}(\ln |k|)^{-2\alpha}\\
\ge &2 \epsilon^2 \sum\limits_{k=1}^{M_1}k^{-1}(\ln |k|)^{-2\alpha}.
\end{align*}
This together with \eqref{est:A2u-2} yields that
\begin{align*}
\big\|A_2(u_0)\big\|_{H^{s+\frac32+\frac1p}}=&
\Big\|(1+|k|^2)^{\frac12(s+\frac32+\frac1p)}\widehat{\big(A_2(u_0)\big)}_k\Big\|_{l^2_k}\\
\ge &
\sqrt2 \epsilon \Big(\sum\limits_{k=1}^{M_1}k^{-1}(\ln k)^{-2\alpha}\Big)^\frac12-\frac\pi 2M_0^{-2}\epsilon^3.
\end{align*}
Choosing $\frac1p<\alpha<\frac12$, then
$$
\sum\limits_{k=1}^{M_1}k^{-1}(\ln k)^{-2\alpha}\to   +\infty,\quad \mbox{when } M_1\to +\infty.
$$
Then the estimates above give that
\begin{align*}
\big\|A_2(u_0)\big\|_{H^{s+\frac32+\frac1p}}
\to & +\infty,\quad \mbox{when } M_1\to +\infty.
\end{align*}
The proof is done by applying \cref{lm:ill-tool}. \qed

\subsection{Proof of  \cref{main:thm2-ill}}

As in the proof of  \cref{main:thm1-ill}, we define the initial data
\begin{align*}
u_0\triangleq \epsilon,
\end{align*}
and
\begin{align}
\xi\triangleq 1+ \sum\limits_{k\in \Z\setminus\{0\}}\frac1{|k|^{s+\beta}} \fe^{ikx},
\end{align}
where $\beta>\frac12$ will be decided later.
Then
$
\|\xi\|_{H^s}\lesssim 1.
$
Given $\gamma>2$, arguing similarly as above, we have that
\begin{align*}
\big\|(1+|k|^2)^{\frac12(s+\gamma)} \eqref{A2u-1}\big\|_{l^2_k}^2
\ge &
\big\|(1+|k|^2)^{\frac12(s+\gamma)} \eqref{A2u-1}\big\|_{l^2_k(\mathbb K)}^2 \\
=&2 \epsilon^2 \sum\limits_{k\in \mathbb K}(1+|k|^2)^{s+\gamma} |k|^{-(4+2s+2\beta)}
\ge 2 \epsilon^2 \sum\limits_{k=1}^{M_1}|k|^{2(\gamma-2)-2\beta}.
\end{align*}
Choosing $\beta$ such that $2\beta-1<2(\gamma-2)$, then
\begin{align*}
\big\|(1+|k|^2)^{\frac12(s+\gamma)} \eqref{A2u-1}\big\|_{l^2_k}^2
\ge &2 \epsilon^2 M_1^{2(\gamma-2)-2\beta+1}.
\end{align*}
This together with \eqref{est:A2u-2}  yields that
\begin{align*}
\big\|A_2(u_0)\big\|_{H^{s+\frac32+\frac1p}}=&
\Big\|(1+|k|^2)^{\frac12(s+\frac32+\frac1p)}\widehat{\big(A_2(u_0)\big)}_k\Big\|_{l^2_k}\\
\ge &
\sqrt2 \epsilon M_1^{\gamma-2-\beta+\frac12}-\frac\pi 2M_0^{-2}\epsilon^3
\to  +\infty,\quad \mbox{when } M_1\to +\infty.
\end{align*}
The proof is done by applying \cref{lm:ill-tool}.

\subsection{Proof of \cref{main:thm3}}
Note that the ill-posedness in $H^{\gamma}$ for $\gamma\ge \frac32$ has been essentially presented in \cref{main:thm1-ill}, so here we only need to further consider $\gamma<\frac32$.

Define the initial data
\begin{align}\label{def:intialdata-2}
u_0\triangleq  \epsilon\sum\limits_{k:|k-N|\le 10} N^{-\gamma} \fe^{ikx},
\end{align}
then
$$
\|u_0\|_{H^{\gamma}}\sim \epsilon.
$$
Define the potential
\begin{align}\label{def:potential-2}
\xi\triangleq  \sum\limits_{k:|k|> 10} \ln(|k|) \> \fe^{ikx},
\end{align}
then for any $\epsilon>0$,
$$
 \big\|\hat \xi_k\big\|_{\hat b^{-\epsilon, \infty}}<+\infty,\quad \mbox{and}\quad
 \big\|\hat \xi_k\big\|_{l^\infty_k}=+\infty.
$$

By \eqref{Form-A2}, we have that
\begin{align*}
\widehat{\big(A_2(u_0)\big)}_k
=&
it \hat\xi_{0} \widehat{(u_0)}_{k}+it \hat\xi_{2k} \widehat{(u_0)}_{-k}-it \hat\xi_0 \widehat{(u_0)}_0\\
&\quad+i\sum\limits_{\substack{k_1+k_2=k\\ k^2-k_2^2\ne 0}} \frac{1}{i(k^2-k_2^2)}\Big(\fe^{it(k^2-k_2^2)}-1\Big)\> \hat\xi_{k_1} \widehat{(u_0)}_{k_2} \\
&\quad +i \sum\limits_{k_1+k_2+k_3=k}\int_0^t\fe^{i\rho(k^2+k_1^2-k_2^2-k_3^2)} \,d\rho\>\widehat{(\overline{u_0})}_{k_1} \widehat{(u_0)}_{k_2} \widehat{(u_0)}_{k_3}.
\end{align*}
According to the definition of $\xi$ in \eqref{def:potential-2}, we have $\hat \xi_0=0$ and so
\begin{subequations}
 \begin{align}
\widehat{\big(A_2(u_0)\big)}_k
=&
it \hat\xi_{2k} \widehat{(u_0)}_{-k}\label{A2-u0-1}\\
&+i\sum\limits_{\substack{k_1+k_2=k\\ k^2-k_2^2\ne 0}} \frac{1}{i(k^2-k_2^2)}\Big(\fe^{it(k^2-k_2^2)}-1\Big)\> \hat\xi_{k_1} \widehat{(u_0)}_{k_2}\label{A2-u0-2}\\
& +i \sum\limits_{k_1+k_2+k_3=k}\int_0^t\fe^{i\rho(k^2+k_1^2-k_2^2-k_3^2)} \,d\rho\>\widehat{(\overline{u_0})}_{k_1} \widehat{(u_0)}_{k_2} \widehat{(u_0)}_{k_3}. \label{A2-u0-3}
\end{align}
\end{subequations}

$\bullet$ {\bf Lower  and upper bounds on  \eqref{A2-u0-1}.}
By \eqref{def:intialdata-2} and \eqref{def:potential-2}, we have that
$$
\widehat{u_0}_k=N^{-\gamma}, \mbox{ for } |k-N|\le 10;\quad
\widehat{u_0}_k=0, \mbox{ for } |k-N|> 10;
$$
and
$$
\hat{\xi}_k=\ln(|k|), \mbox{ for } |k|> 10.
$$
This gives  
\begin{align*}
\big\|\langle k\rangle^{\gamma} \eqref{A2-u0-1}\big\|_{l^2_k}^2
=&
t^2\sum\limits_{k} \langle k\rangle^{2\gamma}
 \big|\hat\xi_{2k}\big|^2 \big| \widehat{(u_0)}_{-k}\big|^2\\
=& t^2\sum\limits_{k:|k-N|\le 10} \langle k\rangle^{2\gamma}
\ln(|k|)^2 N^{-2\gamma}.
\end{align*}
By choosing $N$ large enough, we further get that
\begin{align}\label{est:A-u0-1}
\frac12 t^2 (\ln N)^2\le \big\|\langle k\rangle^{\gamma} \eqref{A2-u0-1}\big\|_{l^2_k}^2
\le &
20 t^2 (\ln N)^2.
\end{align}

$\bullet$ {\bf Upper bound on  \eqref{A2-u0-2}.}
We split \eqref{A2-u0-2} into the following two parts:
\begin{subequations}
\begin{align}
\eqref{A2-u0-2}
=&
i\sum\limits_{\substack{k_1+k_2=k\\ k^2-k_2^2\ne 0,|k_2|\gtrsim |k|}} \frac{1}{i(k^2-k_2^2)}\Big(\fe^{it(k^2-k_2^2)}-1\Big)\> \hat\xi_{k_1} \widehat{(u_0)}_{k_2}\label{A2-u0-2-1}\\
&+i\sum\limits_{\substack{k_1+k_2=k\\ k^2-k_2^2\ne 0,|k_2|\ll |k|}} \frac{1}{i(k^2-k_2^2)}\Big(\fe^{it(k^2-k_2^2)}-1\Big)\> \hat\xi_{k_1} \widehat{(u_0)}_{k_2}. \label{A2-u0-2-2}
\end{align}
\end{subequations}

For \eqref{A2-u0-2-1}, we first note that if $k^2-k_2^2\ne 0$, then
\begin{align}\label{phs-2}
\big|k^2-k_2^2\big|\ge |k|.
\end{align}
Indeed, $\big|k^2-k_2^2\big|=|k+k_2||k-k_2|$.
If $k\cdot k_2\ge 0$, then $|k+k_2|\ge |k|$ and $|k-k_2|\ge 1$, thus we have \eqref{phs-2}.
If $k\cdot k_2< 0$, then $|k-k_2|\ge |k|$ and $|k+k_2|\ge 1$, thus we also have \eqref{phs-2}.
Moreover, \eqref{phs-2} combining with the trivial bound
$$
\big|k^2-k_2^2\big|=|k+k_2||k-k_2|\ge |k-k_2|,
$$
give  
\begin{align}\label{phs-3}
\big|k^2-k_2^2\big|\ge |k|^\frac34|k-k_2|^\frac14.
\end{align}

Therefore, by  \eqref{def:intialdata-2}, \eqref{def:potential-2} and \eqref{phs-3}, we have that
\begin{align*}
\langle k\rangle^{\gamma} \big| \eqref{A2-u0-2-1}\big|
=&
\Bigg|\langle k\rangle^{\gamma}\sum\limits_{\substack{k_1+k_2=k\\ k^2-k_2^2\ne 0,|k_2|\gtrsim |k|\\|k_2-N|\le 10}} \frac{1}{i(k^2-k_2^2)}\Big(\fe^{it(k^2-k_2^2)}-1\Big)\> \ln(|k_1|) N^{-\gamma}\Bigg|\\
\le &
\Bigg|\sum\limits_{\substack{k_1+k_2=k\\ k^2-k_2^2\ne 0,|k_2|\gtrsim |k|\\|k_2-N|\le 10}} \frac{1}{i(k^2-k_2^2)}\Big(\fe^{it(k^2-k_2^2)}-1\Big)\> \ln(|k_1|) \Bigg|\\
\le &
|k|^{-\frac34}\sum\limits_{k_2: |k_2-N|\le 10} \ln(|k-k_2|)|k-k_2|^{-\frac14}
\lesssim
|k|^{-\frac34}.
\end{align*}
This implies  
\begin{align}\label{est:A2-u0-2-1}
\big\|\langle k\rangle^{\gamma}  \eqref{A2-u0-2-1}\big\|_{l^2_k}
\lesssim
\big\||k|^{-\frac34}\big\|_{l^2_k}
\lesssim
1.
\end{align}

For \eqref{A2-u0-2-2}, under the restriction in the summation, we have
$
|k^2-k_2^2|\gtrsim |k|^2
$
and $|k|\sim |k_1|$. Therefore,
\begin{align*}
\langle k\rangle^{\gamma}  \big|\eqref{A2-u0-2-2}\big|
=&
\Bigg|\langle k\rangle^{\gamma}\sum\limits_{\substack{k_1+k_2=k\\ k^2-k_2^2\ne 0,|k_2|\ll |k|\\|k_2-N|\le 10}} \frac{1}{i(k^2-k_2^2)}\Big(\fe^{it(k^2-k_2^2)}-1\Big)\> \ln(|k_1|) N^{-\gamma}\Bigg|\\
\lesssim  &
N^{-\gamma} |k|^{\gamma-2}\ln(|k|) \sum\limits_{k_2: |k_2-N|\le 10} 1
\lesssim
N^{-\gamma} |k|^{\gamma-2}\ln(|k|) .
\end{align*}
Since $\gamma<\frac32$, it yields
\begin{align}\label{est:A2-u0-2-2}
\big\|\langle k\rangle^{\gamma}  \eqref{A2-u0-2-2}\big\|_{l^2_k}
\lesssim
N^{-\gamma}  \big\||k|^{\gamma-2}\ln(|k|) \big\|_{l^2_k\{k\ne 0\}}
\lesssim
N^{-\gamma}.
\end{align}

Together with the estimates on \eqref{est:A2-u0-2-1} and \eqref{est:A2-u0-2-2}, we obtain  
\begin{align}\label{est:A2-u0-2}
\big\|\langle k\rangle^{\gamma} \eqref{A2-u0-2}\big\|_{l^2_k}
\lesssim
\max\big\{1,N^{-\gamma}\big\}.
\end{align}

$\bullet$ {\bf Lower  and upper  bounds on  \eqref{A2-u0-3}.}
By the definition \eqref{def:intialdata-2}, we have
 \begin{align*}
 \eqref{A2-u0-3}
 = &
 i N^{-3\gamma} \sum\limits_{\substack{k_1+k_2+k_3=k\\ |k_1+N|\le 10, |k_j-N|\le 10, j=2,3}}\int_0^t\fe^{i\rho(k^2+k_1^2-k_2^2-k_3^2)} \,d\rho.
\end{align*}
Note that in the restriction: $k_1+k_2+k_3=k, |k_1+N|\le 10, |k_j-N|\le 10, j=2,3$, we have  
$$
\big|k^2+k_1^2-k_2^2-k_3^2\big|
= \big|k_1+k_2\big|\big|k_1+k_3\big|\le 100.
$$
Setting $t\in (0,10^{-3})$, then for any $\rho\in (0,t)$,
$$
\Big|\fe^{i\rho(k^2+k_1^2-k_2^2-k_3^2)} -1\Big|\le \frac12.
$$
This gives that
 \begin{align*}
\frac12 t N^{-3\gamma}\sum\limits_{\substack{k_1+k_2+k_3=k\\ |k_1+N|\le 10, |k_j-N|\le 10, j=2,3}} 1
\le
\big|\eqref{A2-u0-3}\big|
\le 2
t N^{-3\gamma}\sum\limits_{\substack{k_1+k_2+k_3=k\\ |k_1+N|\le 10, |k_j-N|\le 10, j=2,3}} 1.
\end{align*}

For the lower bound, we use the embedding $l^2\hookrightarrow l^\infty$ and obtain that
 \begin{align*}
\big\|\langle k\rangle^{\gamma}\eqref{A2-u0-3}\big\|_{l^2_k}
\ge  &
\big\|\langle k\rangle^{\gamma}\eqref{A2-u0-3}\big\|_{l^\infty_k}\\
\ge &
\frac12 t N^{-3\gamma} \left\|\langle k\rangle^{\gamma} \sum\limits_{\substack{k_1+k_2+k_3=k\\ |k_1+N|\le 10, |k_j-N|\le 10, j=2,3}} 1\right\|_{l^\infty_k}.
\end{align*}
Note that
$$
\Big\{k: k=k_1+k_2+k_3, |k_1+N|\le 10, |k_j-N|\le 10, j=2,3\Big\}
\subset
\Big\{k: |k-N|\le 30 \Big\}.
$$
Choosing $N$ large enough, we further have  
 \begin{align}\label{est:A2-u0-3-low}
\big\|\langle k\rangle^{\gamma}\eqref{A2-u0-3}\big\|_{l^2_k}
\ge &
\frac14 t N^{-2\gamma} \left\| \sum\limits_{\substack{k_1+k_2+k_3=k\\ |k_1+N|\le 10, |k_j-N|\le 10, j=2,3}} 1\right\|_{l^\infty_k}\notag\\
\ge &
5 t N^{-2\gamma}.
\end{align}

For the upper bound, we use the embedding $l^1\hookrightarrow l^2$ and obtain that
 \begin{align}\label{est:A2-u0-3-upp}
\big\|\langle k\rangle^{\gamma}\eqref{A2-u0-3}\big\|_{l^2_k}
\le &
\big\|\langle k\rangle^{\gamma}\eqref{A2-u0-3}\big\|_{l^1_k}\notag\\
\le &
2t N^{-3\gamma} \left\|\langle k\rangle^{\gamma} \sum\limits_{\substack{k_1+k_2+k_3=k\\ |k_1+N|\le 10, |k_j-N|\le 10, j=2,3}} 1 \right\|_{l^2_k(\{|k-N|\le 30\})}\notag\\
\le &
4t N^{-2\gamma}  \sum\limits_{\substack{k_1, k_2, k_3\\ |k_1+N|\le 10, |k_j-N|\le 10, j=2,3}} 1 \notag\\
\le  &
120 t N^{-2\gamma}.
\end{align}
Together with \eqref{est:A2-u0-3-low} and \eqref{est:A2-u0-3-upp}, we obtain 
 \begin{align}\label{est:A2-u0-3}
5 t N^{-2\gamma}
\le \big\|\langle k\rangle^{\gamma}\eqref{A2-u0-3}\big\|_{l^2_k}
\le   &
120 t N^{-2\gamma}.
\end{align}

Now we collect the estimates in \eqref{est:A-u0-1}, \eqref{est:A2-u0-2} and \eqref{est:A2-u0-3}, and choose $N$ large enough to obtain that for $\gamma\ge 0$,
 \begin{align*}
  \big\|A_2(u_0)\big\|_{H^{\gamma}}
=& \left\|\langle k\rangle^{\gamma}\widehat{\big(A_2(u_0)\big)}_k\right\|_{l^2_k}
\ge
\frac{\sqrt2}2 t \ln N-
C\big(tN^{-2\gamma}+1\big)
\ge
\frac12 t \ln N-C;
\end{align*}
for $\gamma<0$,
\begin{align*}
  \big\|A_2(u_0)\big\|_{H^{\gamma}}
=& \left\|\langle k\rangle^{\gamma}\widehat{\big(A_2(u_0)\big)}_k\right\|_{l^2_k}
\ge
5 tN^{-2\gamma}-
C\big(t \ln N+N^{-\gamma}\big)
\ge
tN^{-2\gamma}-N^{-\gamma}.
\end{align*}
By choosing $t=10^{-3}$, for either $\gamma\ge0$ or $\gamma<0$, we can have  
\begin{align*}
\big\|A_2(u_0)\big\|_{H^{\gamma}}
\to & +\infty,\quad \mbox{when } N\to +\infty.
\end{align*}
The proof is done by again applying \cref{lm:ill-tool}.\qed

\vskip1cm

\section{Numerical method and convergence analysis}\label{sec: num method}
In this section, we shall first give the derivation of the presented LRI scheme \eqref{NuSo-NLS}, and then we shall prove its convergence result, i.e., \cref{thm:main-N1}.

\subsection{Construction of numerical scheme}\label{section:derivation}

By the Duhamel formula of \eqref{model} and the  variable $v(t)=\fe^{-it\partial_x^2}u(t)$, we have that for any $n\geq0$,
\begin{align}\label{Duhamel-2}
v(t_{n+1})=v(t_n)+i\int_{t_n}^{t_{n+1}}\fe^{-i\rho\partial_x^2}\Big(\xi u(\rho)-\lambda|u(\rho)|^2u(\rho)\Big)\,d\rho.
\end{align}
Then we write
\begin{align}
v(t_{n+1})
=& i\int_{t_n}^{t_{n+1}}\fe^{-i\rho\partial_x^2}\Big(\xi\> \fe^{i\rho\partial_x^2} v(t_n)\Big)\,d\rho\label{T1}\\
& + v(t_n)-i\lambda\int_{t_n}^{t_{n+1}}\fe^{-i\rho\partial_x^2}\Big(|u(\rho)|^2u(\rho)\Big)\,d\rho \label{T2}\\
& + R_1^n,\notag
\end{align}
where
\begin{align*}
R_1^n\triangleq
i\int_{t_n}^{t_{n+1}}\fe^{-i\rho\partial_x^2}\Big[\xi\> \fe^{i\rho\partial_x^2}\big(v(\rho)-v(t_n)\big)\Big]\,d\rho\quad n\geq0.
\end{align*}

We first consider the integral term \eqref{T1} involving the potential. By taking the Fourier transform, we have
\begin{align*}  
\widehat{\eqref{T1}}_k=&
i\int_{t_n}^{t_{n+1}}\sum\limits_{k_1+k_2=k}\fe^{i\rho(k^2-k_2^2)}\hat \xi_{k_1}
\hat v_{k_2}(t_n)\,d\rho\\
=&i\tau\fe^{it_nk^2}\sum\limits_{k_1+k_2=k}\mathcal M_\tau\left(\fe^{i\rho(k^2-k_2^2)}\right)\hat\xi_{k_1}
\fe^{-it_nk_2^2}\hat v_{k_2}(t_n),
\end{align*}
where $\mathcal M_\tau$ is the average operator given in \eqref{time average operator}.
Then we consider the approximation $\mathcal M_\tau\left(\fe^{i\rho(k^2-k_2^2)}\right)\approx \mathcal M_\tau\left(\fe^{i\rho k^2}\right)\mathcal M_\tau\left(\fe^{-i\rho k_2^2}\right)$, and note that
\begin{align*}
&\sum_{k}\fe^{ikx}\fe^{it_nk^2}\sum\limits_{k_1+k_2=k}\mathcal M_\tau\left(\fe^{i\rho k^2}\right)\mathcal M_\tau\left(\fe^{-i\rho k_2^2}\right)\hat\xi_{k_1}
\fe^{-it_nk_2^2}\hat v_{k_2}(t_n)\\
=&\fe^{-it_n\partial_x^2} \mathcal D_{-\tau}\Big[\xi\>
\mathcal D_\tau[\fe^{it_n\partial_x^2}v(t_n)]\Big]=\fe^{-it_n\partial_x^2} \mathcal D_{-\tau}\Big[\xi\>
\mathcal D_\tau[u(t_n)]\Big].
\end{align*}
So we can write the integral term \eqref{T1} as
\begin{align}\label{T1:inter}
\eqref{T1}
=i\tau \fe^{-it_n\partial_x^2} \mathcal D_{-\tau}\Big[\xi\>
\mathcal D_\tau[u(t_n)]\Big]
+R^n_2,
\end{align}
where $R^n_2$ is the truncation term defined by
$$
\widehat{(R^n_2)}_k\triangleq
i\tau \fe^{it_nk^2}\sum\limits_{k_1+k_2=k}\left[ \mathcal M_\tau\left(\fe^{i\rho(k^2-k_2^2)}\right)- \mathcal M_\tau\left(\fe^{i\rho k^2}\right)\mathcal M_\tau\left(\fe^{-i\rho k_2^2}\right)\right]\hat \xi_{k_1}\fe^{-it_nk_2^2}
\hat v_{k_2}(t_n),
$$
for $n\geq0$.
For the term \eqref{T2}, we simply use the Lie-Trotter splitting method and we denote the remainder as
\begin{align}\label{Apprx:T2}
R^n_3 \triangleq \eqref{T2} -\fe^{-it_{n+1}\partial_x^2}\mathcal N_\tau\big[\fe^{i\tau\partial_x^2}u(t_n)\big],\quad n\geq0.
\end{align}

Plugging \eqref{T1:inter}, \eqref{Apprx:T2} back to \eqref{T1} and \eqref{T2} without the remainder terms $R^n_2,R^n_3$, and denoting $v^n\approx v(t_n)$ yield:
\begin{align}\label{numerical:vn}
v^{n+1}=i\tau \fe^{-it_n\partial_x^2} \mathcal D_{-\tau}\Big[\xi\>
\mathcal D_\tau\big[\fe^{it_n\partial_x^2}v^n\big]\Big]
+\fe^{-it_{n+1}\partial_x^2}\mathcal N_\tau\big[\fe^{i\tau\partial_x^2}\fe^{it_n\partial_x^2}v^n\big].
\end{align}
Therefore, by letting $u(t_n)\approx u^n=\fe^{it_n\partial_x^2}v^n$ with $u^0=v^0=u_0$ and noting $\fe^{i\tau\partial_x^2}\mathcal{D}_{-\tau}
=\mathcal{D}_{\tau}$, we obtain our LRI scheme presented in (\ref{NuSo-NLS}):
$$
u^{n+1}=i\tau  \mathcal D_\tau\Big[\xi\>
\mathcal D_\tau[u^n]\Big]
+\mathcal N_\tau\big[\fe^{i\tau\partial_x^2} u^n\big],\quad n\geq0.
$$
Note that the filter $P_{\leq N}$ is an additional approximation applied in  $\mathcal N_\tau$ to cut off the high frequencies. We remark that such approximation is mainly used for the technical analysis of the convergence under $L^2$-norm \cite{Ignat,Zuazua,JEMS}. It will not essentially affect the computational results in practice, and it could be omitted which will be addressed in a future work.

\subsection{Convergence analysis}
In the rest of this section, we assume $\lambda=-1$ in (\ref{model}) for simplicity of notations, and we aim to prove the error estimate given in \cref{thm:main-N1}.

Denote the part of the scheme \eqref{numerical:vn} involving the integration of the potential as
$$
\Phi(f)\triangleq  f+i\tau \fe^{-it_n\partial_x^2}  \mathcal D_{-\tau}\Big[\xi\>
\mathcal D_\tau[\fe^{it_n\partial_x^2}  f]\Big],
$$
and denote the splitting part involving the nonlinearity as
$$
 \Psi(f)\triangleq \fe^{-it_{n+1}\partial_x^2} \Big(\mathcal N_\tau\big[\fe^{i\tau\partial_x^2}\fe^{it_n\partial_x^2} f\big]-\fe^{i\tau\partial_x^2}\fe^{it_n\partial_x^2} f\Big).
$$
Then  \eqref{numerical:vn} reads
$$
v^{n+1}=\Phi\big(v^n\big)
+\Psi\big(v^n\big).
$$
Denote the error function as
$$
h^n\triangleq v^n-v(t_n),\quad n\geq0,
$$
so $h^0\equiv0$ and we have
\begin{align*}
h^{n+1}=\Phi\big(v^n\big)-\Phi\big(v(t_n)\big)+\Psi\big(v^n\big)-\Psi\big(v(t_n)\big)
+  R^n_1 +R^n_2 + R^n_3.
\end{align*}
We shall first analyze to give the stability result of the scheme (\ref{numerical:vn}) by working on the $\Phi$ part and the $\Psi$ part in a sequel. Afterwards, we shall estimate the local truncation errors $R^n_1,R^n_2,R^n_3$.

\begin{lemma}[Stability of potential part]\label{lem:Phi-stab}
 Let $\xi\in \hat l^\infty$, then for any $f\in H^{\frac12+}(\T)$ and $h=f-g\in L^2(\T)$,
\begin{align}
\big\|\Phi(f)-\Phi(g)\big\|_{L^2}
\le  (1+C\tau )\|h\|_{L^2},
\end{align}
where the constant $C>0$ depends only on $\|\xi\|_{\hat{l}^{\infty}}$.
\end{lemma}
\begin{proof}
Denote $\tilde h=\fe^{it_n\partial_x^2}h$, then $\|\tilde h\|_{L^2}=\|h\|_{L^2}$.
Note that
\begin{align*}
\Phi(f)-\Phi(g)
&=\fe^{-it_n\partial_x^2}\Big(\tilde h+ i\tau  \mathcal D_{-\tau}\Big[\xi\>
\mathcal D_\tau \tilde h\Big]\Big).
\end{align*}
Therefore,
\begin{align*}
\big\|\Phi(f)-\Phi(g)\big\|_{L^2}^2
= & \Big\|\tilde h+ i\tau  \mathcal D_{-\tau}\Big[\xi\>
\mathcal D_\tau \tilde h\Big]\Big\|_{L^2}^2\\
= & \|h\|_{L^2}^2+2\Big\langle \tilde h, i\tau  \mathcal D_{-\tau}\Big[\xi\>
\mathcal D_\tau \tilde h\Big]\Big\rangle
+\tau^2\Big\|  \mathcal D_\tau\Big[\xi\>
\mathcal D_\tau \tilde h\Big]\Big\|_{L^2}^2.
\end{align*}
where in the last step we have used the relationship: $\fe^{i\tau\partial_x^2}\mathcal{D}_{-\tau}
=\mathcal{D}_{\tau}$.

A key observation (for real-valued $\xi$) is that
\begin{align}\label{est:vanish}
\Big\langle \tilde h, i\tau  \mathcal D_{-\tau}\Big[\xi\>
\mathcal D_\tau \tilde h\Big]\Big\rangle
=&\Big\langle  \mathcal D_\tau\tilde h, i\tau \>\xi\>
\mathcal D_\tau\tilde h\Big\rangle
=\tau \mbox{Im }\int_\T \xi | \mathcal D_\tau\tilde h|^2\,dx
=0,
\end{align}
where in the second step we have used the relationship $\mathcal D_{-\tau}=\overline{\mathcal D_\tau}$.

Moreover, we claim that
\begin{align}\label{est:Dh}
\tau \Big\|  \mathcal D_\tau\Big[\xi\>
\mathcal D_\tau\tilde h\Big]\Big\|_{L^2}
\lesssim
\sqrt{\tau} \big\|\xi\big\|_{\hat{l}^\infty} \|h\|_{L^2}.
\end{align}
To prove this claim, we take the Fourier transform and write
\begin{align*}
\mathcal F_k\Big[\mathcal D_\tau\big(\xi\>
\mathcal D_\tau\tilde h\big)\Big]
=&
\sum\limits_{k_1+k_2=k}\mathcal M_\tau\big(\fe^{isk^2}\big)\mathcal M_\tau\big(\fe^{-isk_2^2}\big)\hat \xi_{k_1}
\hat{\tilde h}_{k_2}.
\end{align*}
Now by duality and Plancherel's identity  we have
\begin{align*}
 \Big\|  \mathcal D_\tau\Big[\xi\>
\mathcal D_\tau\tilde h\Big]\Big\|_{L^2}
=&\sup_{w:\|w\|_{l^2}=1}\Big\langle \sum\limits_{k_1+k_2=k}\mathcal M_\tau\big(\fe^{isk^2}\big)\mathcal M_\tau\big(\fe^{-isk_2^2}\big)\hat \xi_{k_1}
\hat{\tilde h}_{k_2}, w\Big\rangle.
\end{align*}
As before, we  omit  $\sup\limits_{w:\|w\|_{l^2}=1}$ in the front. Then we split it into the following three cases and further write
\begin{align*}
 \Big\|  \mathcal D_\tau\Big[\xi\>
\mathcal D_\tau\tilde h\Big]\Big\|_{L^2}
=&
\sum\limits_{\substack{k_1+k_2=k\\ |k|\le N_0, |k_2|\le N_0}}\mathcal M_\tau\big(\fe^{isk^2}\big)\mathcal M_\tau\big(\fe^{-isk_2^2}\big)\hat \xi_{k_1}
\hat{\tilde h}_{k_2}\>w_k\\
&\qquad +\sum\limits_{\substack{k_1+k_2=k\\ |k|\le N_0, |k_2|\ge N_0}}\mathcal M_\tau\big(\fe^{isk^2}\big)\mathcal M_\tau\big(\fe^{-isk_2^2}\big)\hat \xi_{k_1}
\hat{\tilde h}_{k_2}\>w_k\\
&\qquad+\sum\limits_{\substack{k_1+k_2=k\\ |k|\ge N_0, |k_2|\ge N_0}}\mathcal M_\tau\big(\fe^{isk^2}\big)\mathcal M_\tau\big(\fe^{-isk_2^2}\big)\hat \xi_{k_1}
\hat{\tilde h}_{k_2}\>w_k,
\end{align*}
where $N_0>0$ will be determined later.
Note that
$$
\big|\mathcal M_\tau\big(\fe^{isk^2}\big)\big|
\lesssim \min\{1,\tau^{-1}|k|^{-2}\},
$$
therefore,
 \begin{align*}
 \Big\|  \mathcal D_\tau\Big[\xi\>
\mathcal D_\tau\tilde h\Big]\Big\|_{L^2}
\lesssim &
\|\xi\|_{\hat{l}^\infty}\Big(\sum\limits_{\substack{k,k_2\\ |k|\le N_0, |k_2|\le N_0}}\big|
\hat{\tilde h}_{k_2}\big|\>|w_k|
 +\sum\limits_{\substack{k,k_2\\ |k|\le N_0, |k_2|\ge N_0}} \tau^{-1}|k_2|^{-2} \big|
\hat{\tilde h}_{k_2}\big|\>|w_k|\\
&\qquad\quad +\sum\limits_{\substack{k,k_2\\ |k|\ge N_0, |k_2|\ge N_0}}\tau^{-2}|k|^{-2}|k_2|^{-2}
 \big|
\hat{\tilde h}_{k_2}\big|\Big)\>|w_k|.
\end{align*}
This gives that
\begin{align*}
 \Big\|  \mathcal D_\tau\Big[\xi\>
\mathcal D_\tau\tilde h\Big]\Big\|_{L^2}
\lesssim &
\|\xi\|_{\hat{l}^\infty}
\Big(N_0+\tau^{-1}N_0^{-1}+\tau^{-2}N_0^{-3}\Big)\big\|\hat h_k\big\|_{l^2_k}\|w_k\|_{l^2_k}\\
\lesssim &
\|\xi\|_{\hat{l}^\infty}
\Big(N_0+\tau^{-1}N_0^{-1}+\tau^{-2}N_0^{-3}\Big)\big\|h\big\|_{L^2}.
\end{align*}
Choosing $N_0=\tau^{-\frac12}$, we obtain \eqref{est:Dh}.

Together with \eqref{est:vanish} and \eqref{est:Dh}, we have
\begin{align*}
\big\|\Phi(f)-\Phi(g)\big\|_{L^2}^2
\le & (1+C\tau) \|h\|_{L^2}^2,
\end{align*}
and thus
\begin{align*}
\big\|\Phi(f)-\Phi(g)\big\|_{L^2}
\le & (1+C\tau) \|h\|_{L^2}.
\end{align*}
This finishes the proof of the lemma.
\end{proof}

\begin{lemma}
[Stability of splitting part]\label{lem:stability split part}
  Let  $f\in L^2(\T)$ and $g\in H^{\frac12+}(\T)$, then
  \begin{align}\label{lem:stability split part result}
    \left\| \Psi(f)- \Psi(g)\right\|_{L^2}
    \leq C\tau\left(1+N\|f-g\|_{L^2}^2\right)\|f-g\|_{L^2},
  \end{align}
  for some constant $C>0$ that depends on $\|g\|_{H^{\frac12+}}$.
\end{lemma}
\begin{proof}
For simplicity, we denote
$$
\tilde f=\fe^{i\tau\partial_x^2}\fe^{it_n\partial_x^2}f,
\quad
\tilde g=\fe^{i\tau\partial_x^2}\fe^{it_n\partial_x^2}g,
\quad \tilde h=\fe^{i\tau\partial_x^2}\fe^{it_n\partial_x^2} h,\quad h=f-g.
$$
Then
$$
 \Psi(f)-\Psi(g)= \fe^{-it_{n+1}\partial_x^2} \Big(\mathcal N_\tau\big[\tilde f\big]-\mathcal N_\tau\big[\tilde g\big]-\tilde h\Big).
$$
We rewrite
\begin{align*}
  \mathcal{N}_\tau[\tilde f]-\mathcal{N}_\tau[\tilde g]
  =&\fe^{i\tau|P_{\leq N}\tilde f|^2}\tilde f-
  \fe^{i\tau|P_{\leq N}\tilde g|^2}\tilde g\\
  =&\fe^{i\tau|P_{\leq N}f|^2}\tilde h+
  \left[\fe^{i\tau|P_{\leq N}\tilde f|^2}-\fe^{i\tau|P_{\leq N}\tilde g|^2}\right] \tilde g.
\end{align*}
Therefore,
$$
 \Psi(f)-\Psi(g)
 =\fe^{-it_{n+1}\partial_x^2}
 \left[\left(\fe^{i\tau|P_{\leq N}\tilde f|^2}-1\right)\tilde h+
  \left(\fe^{i\tau|P_{\leq N}\tilde f|^2}-\fe^{i\tau|P_{\leq N}\tilde g|^2}\right) \tilde g\right].
$$
Then by  the Sobolev and Bernstein inequalities, we have
\begin{align*}
\left\|\Psi(f)-\Psi(g)\right\|_{L^2}
\lesssim & \left\|\Big[\fe^{i\tau|P_{\leq N}\tilde f|^2}-1\Big]\tilde h\right\|_{L^2}
+ \left\|\left[\fe^{i\tau|P_{\leq N}\tilde f|^2}-\fe^{i\tau|P_{\leq N}\tilde g|^2}\right] \tilde g\right\|_{L^2} \\
\lesssim & \Big\|\fe^{i\tau|P_{\leq N}\tilde f|^2}-1\Big\|_{L^\infty} \big\|\tilde h\big\|_{L^2}
+ \left\|\fe^{i\tau|P_{\leq N}\tilde f|^2}-\fe^{i\tau|P_{\leq N}\tilde g|^2}\right\|_{L^2}\big\|\tilde g\big\|_{L^\infty}.
  \end{align*}
Note that
\begin{align*}
\Big\|\fe^{i\tau|P_{\leq N}\tilde f|^2}-1\Big\|_{L^\infty}
\lesssim
\Big\|i\tau|P_{\leq N}\tilde f|^2\Big\|_{L^\infty}=&\tau \big\|P_{\leq N}\tilde f\big\|_{L^\infty}^2
\lesssim  \tau N\big\|h\big\|_{L^2}^2+\tau\big\|g\big\|_{H^{\frac12+}}^2.
\end{align*}
Similarly, we have
\begin{align*}
\left\|\fe^{i\tau|P_{\leq N}\tilde f|^2}-\fe^{i\tau|P_{\leq N}\tilde g|^2}\right\|_{L^2}
\lesssim &
\tau \big\||P_{\leq N}\tilde f|^2-|P_{\leq N}\tilde g|^2\big\|_{L^2}\\
\lesssim &
\tau \big\|P_{\leq N}\tilde h\big\|_{L^2}
\big(\big\|P_{\leq N}\tilde f\big\|_{L^\infty}+\big\|P_{\leq N}\tilde g\big\|_{L^\infty}\big)\\
\lesssim & \tau N^\frac12\big\|h\big\|_{L^2}^2+\tau  \big\|g\big\|_{H^{\frac12+}} \big\|h\big\|_{L^2}.
\end{align*}
Applying these two estimates, we further have
\begin{align*}
\left\|\Psi(f)-\Psi(g)\right\|_{L^2}
\le & C \tau \big\|h\big\|_{L^2}+C\tau N^\frac12\big\|h\big\|_{L^2}^2+ C \tau N\big\|h\big\|_{L^2}^3\\
\le & C \tau \big\|h\big\|_{L^2}+ C \tau N\big\|h\big\|_{L^2}^3,
  \end{align*}
where the constant $C>0$ depends  on $\|g\|_{H^{\frac12+}}$.
 Hence, we obtain the desired assertion.
\end{proof}

Now, we move on to  estimating the local truncation errors $R^n_1,R^n_2,R^n_3$ one by one.  To estimate $R_1^n$, we need the following result.
\begin{lemma}\label{lem:v-s-tn}
Let $\xi\in \hat b^{s,p}$ with the conditions of \cref{thm:main-N1} satisfied, so $u\in L^\infty_t H_x^{s+\gamma_p-}$\ $([0,T]\times\T)$ with $ \gamma_p=\frac32+\frac1p$. Then, for any $0\leq\beta< s+\frac1p+\frac32$ we have
\begin{align}
\big\|v(t)-v(t_n)\big\|_{L^\infty_tH^\beta_x([t_n,t_{n+1}]\times\T)}
\le C\tau\max\left\{\tau^{-\frac12(\beta-s-\frac1p)-\frac14-},1\right\},
\end{align}
where the constant $C>0$ depends only on $\|\xi\|_{\hat b^{s,p}}$ and $\|u\|_{L^\infty_t H_x^{s+\gamma_p-}}$.
\end{lemma}
\begin{proof}
By $\|v(t)-v(t_n)\|_{H^{s+\gamma_p-}}\le  2\|v\|_{L_t^\infty H_x^{s+\gamma_p-}}$, we can deduce that for any
$\beta<s+\gamma_p$,
\begin{align}\label{v-differ-HL}
  \|v(t)-v(t_n)\|_{H^\beta}&=\|P_{\leq N_0}(v(t)-v(t_n))\|_{H^\beta}+\|P_{> N_0}(v(t)-v(t_n))\|_{H^\beta}\notag\\
&\lesssim \|P_{\leq N_0}(v(t)-v(t_n))\|_{H^\beta}+N_0^{-s-\gamma_p+\beta+}\|v\|_{L_t^\infty H_x^{s+\gamma_p-}},
\end{align}
where $N_0>0$ will be determined later.
Hence, in the following  we only need to estimate $\|P_{\leq N_0}(v(t)-v(t_n))\|_{H^\beta}$.

By \eqref{Duhamel-1}, we have that
\begin{align*}
v(t)-v(t_n)=&i\int_{t_n}^t\fe^{-i\rho\partial_x^2}\Big[\xi \>\fe^{i\rho\partial_x^2}v(\rho)\Big]\,d\rho+i\int_{t_n}^t\fe^{-i\rho\partial_x^2}|u(\rho)|^2u(\rho)\,d\rho\\
\triangleq & I_1+I_2.
\end{align*}
Since $s+\gamma_p>\frac12$, we can directly have
\begin{equation}\label{est:I2}
\|I_2\|_{L^\infty_tH^{s+\gamma_p-}_x([t_n,t_{n+1}])}\lesssim\tau\|u\|_{H^{s+\gamma_p-}}^3.
\end{equation}
Then it remains to estimate $P_{\leq N_0}I_{1}$.
 For $I_1$, we take the Fourier transform to have
\begin{align*}
 \widehat{(I_1)}_k&=i\int_{t_n}^t\sum_{k=k_1+k_2}\fe^{i\rho(k^2-k_2^2)}\hat{\xi}_{k_1}\hat{v}_{k_2}(\rho)d\rho\\
 &=i\int_{t_n}^t\sum_{\substack{k_1+k_2=k\\|k_1|\lesssim |k_2|}}\fe^{i\rho(k^2-k_2^2)}\hat{\xi}_{k_1}\hat{v}_{k_2}(\rho)d\rho
 +i\int_{t_n}^t\sum_{\substack{k_1+k_2=k\\|k_1|\gg |k_2|}}\fe^{i\rho(k^2-k_2^2)}\hat{\xi}_{k_1}\hat{v}_{k_2}(\rho)d\rho\\
 &\triangleq  \widehat{(I_{11})}_k+\widehat{(I_{12})}_k.
\end{align*}
We shall estimate $P_{\leq N_0}I_{11}$ and $P_{\leq N_0}I_{12}$ in a sequel.

For $P_{\leq N_0}I_{11}$, by duality and Plancherel's identity  we have
\begin{align*}
  &\|P_{\leq N_0}I_{11}\|_{L^\infty_tH^\beta_x([t_n,t_{n+1}]\times\T)}\\
  \lesssim &
  \tau\Big\|\sup\limits_{h:\|h\|_{l^2}=1}\sum_{\substack{|k_1|\lesssim |k_2|\\|k_1+k_2|\leq N_0}} \big|\hat{\xi}_{k_1}\big|
  |\hat{v}_{k_2}(t)||h_{k_1+k_2}|\cdot \langle k_1+k_2\rangle^\beta\Big\|_{L^\infty_t([t_n,t_{n+1}])}\notag\\
\lesssim & \tau\Big\|\sup\limits_{h:\|h\|_{l^2}=1}\sum_{\substack{|k_1|\lesssim |k_2|\\|k_1+k_2|\leq N_0}} \langle k_1\rangle^s \big|\hat{\xi}_{k_1}\big|\>\langle k_2\rangle^{s+\gamma_p-}
  |\hat{v}_{k_2}(t)| \> |h_{k_1+k_2}|\langle k_1\rangle^{-s} \\
  &\qquad\qquad\qquad\qquad  \cdot\langle k_1+k_2\rangle^\beta \langle k_2\rangle^{-s-\gamma_p+}\Big\|_{L^\infty_t([t_n,t_{n+1}])}\\
 \lesssim & \tau\Big\|\sup\limits_{h:\|h\|_{l^2}=1}\sum_{\substack{|k_1|\lesssim |k_2|\\|k_1+k_2|\leq N_0}} \langle k_1\rangle^s \big|\hat{\xi}_{k_1}\big|\>\langle k_2\rangle^{s+\gamma_p-}
  |\hat{v}_{k_2}(t)| \> |h_{k_1+k_2}| \langle k_1\rangle^{-s} \\
  &\qquad \qquad\qquad\qquad \cdot\min\{N_0^{\beta},\langle k_2\rangle^\beta\} \langle k_2\rangle^{-s-\gamma_p+}\Big\|_{L^\infty_t([t_n,t_{n+1}])}.
\end{align*}
Denote $\frac1r=\frac12-\frac1p$.

If $s\le \frac1r$,
then  our strategy is to first sum up $k_1$ and then sum up $k_2$,  and by Cauchy-Schwartz's inequality, we further have
\begin{align*}
  &\big\|P_{\leq N_0}I_{11}\big\|_{L^\infty_tH^\beta_x([t_n,t_{n+1}]\times\T)}\\
  \lesssim  &
  \tau \|\xi\|_{\hat b^{s,p}}\Big\| \sum_{k_2}
 \big\|\langle k_1\rangle^{-s}\big\|_{l^r(\{k_1:|k_1|\lesssim |k_2|\})}   \min\left\{N_0^{\beta},\langle k_2\rangle^\beta\right\}\langle k_2\rangle^{-s-\gamma_p+}\\
 &\qquad\qquad\quad \cdot \langle k_2\rangle^{s+\gamma_p-} |\hat{v}_{k_2}(t)| \Big\|_{L^\infty_t([t_n,t_{n+1}])}\notag\\
   \lesssim  &
  \tau \|\xi\|_{\hat b^{s,p}}\Big\| \sum_{k_2}
\langle k_2\rangle^{-2s-\gamma_p+\frac1r+} \min\left\{N_0^{\beta},\langle k_2\rangle^\beta\right\} \langle k_2\rangle^{s+\gamma_p-} |\hat{v}_{k_2}(t)| \Big\|_{L^\infty_t([t_n,t_{n+1}])}\notag\\
  \lesssim  &
  \tau \|\xi\|_{\hat b^{s,p}}\|v\|_{L^\infty_tH^{s+\gamma_p-}_x} \Big\|
\langle k_2\rangle^{-2s-\gamma_p+\frac1r+} \min\left\{N_0^{\beta},\langle k_2\rangle^\beta\right\} \Big\|_{l^2_{k_2}}.
\end{align*}
Noticing that
\begin{equation*}
 \Big\|
\langle k_2\rangle^{-2s-\gamma_p+\frac1r+} \min\left\{N_0^{\beta},\langle k_2\rangle^\beta\right\} \Big\|_{l^2_{k_2}}
\lesssim \max\left\{N_0^{-2s-\frac{2}{p}+\beta-\frac12+},1\right\},
\end{equation*}
we have
\begin{align}\label{est:I11-case1}
  \big\|P_{\leq N_0}I_{11}&\big\|_{L^\infty_tH^\beta_x([t_n,t_{n+1}]\times\T)}
  \lesssim
  \tau \max\left\{N_0^{-2s-\frac{2}{p}+\beta-\frac12+},1\right\} \|\xi\|_{\hat b^{s,p}}\|v\|_{L^\infty_tH^{s+\gamma_p-}_x} .
\end{align}

If $s> \frac1r$, then our strategy is to first sum up $k_2$ and then sum up $k_1$,  and  by Cauchy-Schwartz's inequality, we  have
\begin{align*}
  &\big\|P_{\leq N_0}I_{11}\big\|_{L^\infty_tH^\beta_x([t_n,t_{n+1}]\times\T)}\\
  \lesssim  &
  \tau \|v\|_{L^\infty_tH^{s+\gamma_p-}_x} \sum_{k_1}  \langle k_1\rangle^s \big|\hat{\xi}_{k_1}\big|
 \langle k_1\rangle^{-s}  \left\| \min\left\{N_0^{\beta},\langle k_2\rangle^\beta\right\}\langle k_2\rangle^{-s-\gamma_p+}\right\|_{l^\infty_{\{k_2:|k_1|\lesssim |k_2|\}}}.\end{align*}
Furthermore if $s+\frac1p<1$,  we can apply the following estimate:
$$
 \left\| \min\left\{N_0^{\beta},\langle k_2\rangle^\beta\right\}\langle k_2\rangle^{-s-\gamma_p+}\right\|_{l^\infty_{\{k_2:|k_1|\lesssim |k_2|\}}}
 \lesssim 1,
$$
where we have used the relation: $\beta<s+\gamma_p$, and then we can obtain
\begin{align*}
  \big\|P_{\leq N_0}I_{11}\big\|_{L^\infty_tH^\beta_x([t_n,t_{n+1}]\times\T)}
  \lesssim  &
  \tau \|\xi\|_{\hat b^{s,p}} \|v\|_{L^\infty_tH^{s+\gamma_p-}_x} \big\|\langle k_1\rangle^{-s}\big\|_{l^{p'}_{k_1}}\\
  \lesssim  &
  \tau \|\xi\|_{\hat b^{s,p}} \|v\|_{L^\infty_tH^{s+\gamma_p-}_x}.
\end{align*}
Otherwise for $s+\frac1p\ge 1$,  we have
$$
 \big\| \min\{N_0^{\beta},\langle k_2\rangle^\beta\}\langle k_2\rangle^{-s-\gamma_p+}\big\|_{l^\infty_{\{k_2:|k_1|\lesssim |k_2|\}}}
 \lesssim \langle k_1\rangle^{s+\frac1{p'}-1-}\max\left\{N_0^{-2s-\frac{2}{p}+\beta-\frac12+},1\right\},
$$
and we can find
\begin{align*}
  &\big\|P_{\leq N_0}I_{11}\big\|_{L^\infty_tH^\beta_x([t_n,t_{n+1}]\times\T)}\\
  \lesssim  &
  \tau \max\left\{N_0^{-2s-\frac{2}{p}+\beta-\frac12+},1\right\} \|\xi\|_{\hat b^{s,p}} \|v\|_{L^\infty_tH^{s+\gamma_p-}_x} \big\|\langle k_1\rangle^{-1-}\big\|_{l^{p'}_{k_1}}\\
  \lesssim  &
  \tau \max\left\{N_0^{-2s-\frac{2}{p}+\beta-\frac12+},1\right\} \|\xi\|_{\hat b^{s,p}} \|v\|_{L^\infty_tH^{s+\gamma_p-}_x}.
\end{align*}
Therefore, we obtain that for $s>\frac1r$,
\begin{align*}
  \big\|P_{\leq N_0}I_{11}&\big\|_{L^\infty_tH^\beta_x([t_n,t_{n+1}]\times\T)}
  \lesssim
  \tau \max\left\{N_0^{-2s-\frac{2}{p}+\beta-\frac12+},1\right\} \|\xi\|_{\hat b^{s,p}} \|v\|_{L^\infty_tH^{s+\gamma_p-}_x}.
\end{align*}
This combining with \eqref{est:I11-case1} gives that for either  $s\le \frac1r$ or  $s>\frac1r$,
\begin{align}\label{est:I11}
  \big\|P_{\leq N_0}I_{11}&\big\|_{L^\infty_tH^\beta_x([t_n,t_{n+1}]\times\T)}
  \lesssim
  \tau \max\left\{N_0^{-2s-\frac{2}{p}+\beta-\frac12+},1\right\} \|\xi\|_{\hat b^{s,p}} \|v\|_{L^\infty_tH^{s+\gamma_p-}_x}.
\end{align}

For $P_{\leq N_0}I_{12}$, we have
\begin{align}
  &\|P_{\leq N_0}I_{12}\|_{L^\infty_tH^\beta_x([t_n,t_{n+1}]\times\T)}\nonumber\\
  \lesssim & \tau\Big\|\sup\limits_{h:\|h\|_{l^2}=1}
  \sum_{\substack{|k_1|\gg|k_2|\\|k_1+k_2|\leq N_0}} \big|\hat{\xi}_{k_1}\big|
  |\hat{v}_{k_2}(t)| |h_{k_1+k_2}|  \langle k_1+k_2\rangle^\beta\Big\|_{L^\infty_t([t_n,t_{n+1}])}\nonumber\\
\lesssim & \tau\Big\|\sup\limits_{h:\|h\|_{l^2}=1}\sum_{\substack{|k_1|\gg|k_2|\\|k_1+k_2|\leq N_0}}
\langle k_1\rangle^{s}\big|\hat{\xi}_{k_1}\big|   |\hat{v}_{k_2}(t)| |h_{k_1+k_2}|\cdot\langle k_1\rangle^{\beta-s}
\Big\|_{L^\infty_t([t_n,t_{n+1}])}\nonumber\\
  \leq&\tau\|\xi\|_{\hat{b}^{s,p}}\Big\|\sum_{k_2}\big\|\langle k_1\rangle^{\beta-s}\big\|_{l^r_{\{k_1:|k_1|\lesssim N_0\}}}
  |\hat{v}_{k_2}(t)|\Big\|_{L^\infty_t([t_n,t_{n+1}])}. \label{err est proof eq2}
\end{align}
Noticing
\begin{equation*}
\big\|\langle k_1\rangle^{\beta-s}\big\|_{l^r_{\{k_1:|k_1|\lesssim N_0\}}}\lesssim \left\{\begin{split}&N_0^{\beta-s+1/r+},\quad\,\,\,\, \beta-s+1/r\geq0,\\
&1,\quad \qquad \qquad\quad \beta-s+1/r<0,\end{split}\right.
\end{equation*}
(\ref{err est proof eq2}) further gives
\begin{align}\label{est:I12}
\|P_{\leq N_0}I_{12}\|_{L^\infty_tH^\beta_x([t_n,t_{n+1}]\times\T)}
\lesssim &
\tau\max\{N_0^{\beta-s+\frac1r+,1},1\}
\|\xi\|_{\hat{b}^{s,p}}\Big\|\sum_{k_2}
  |\hat{v}_{k_2}(t)|\Big\|_{L^\infty_t([t_n,t_{n+1}])}\notag\\
\lesssim &
\tau\max\Big\{N_0^{-s-\frac1p+\beta+\frac12+},1\Big\}
\|\xi\|_{\hat{b}^{s,p}}
\|v\|_{L^\infty_tH^{s+\gamma_p-}_x},
\end{align}
where we have used $s+\gamma_p>\frac12$ in the last step.

Now (\ref{est:I2}) together with \eqref{est:I11} and \eqref{est:I12} lead us to
\begin{align*}
\|P_{\leq N_0}(v(t)-v(t_n))\|_{L^\infty_tH^\beta_x([t_n,t_{n+1}]\times\T)}\leq C\tau\max\left\{N_0^{-2s-\frac{2}{p}+\beta-\frac12+}, N_0^{-s-\frac1p+\beta+\frac12+},1\right\}.
\end{align*}
Noting that $-2s-\frac{2}{p}+\beta-\frac12<-s-\frac1p+\beta+\frac12$, we have
\begin{align*}
\|P_{\leq N_0}(v(t)-v(t_n))\|_{L^\infty_tH^\beta_x([t_n,t_{n+1}]\times\T)}\leq C\tau\max\left\{N_0^{-s-\frac1p+\beta+\frac12+},1\right\}.
\end{align*}
This combining with \eqref{v-differ-HL} give
\begin{align*}
\|v(t)-v(t_n)\|_{L^\infty_tH^\beta_x([t_n,t_{n+1}]\times\T)}\leq C\left(\tau\max\left\{N_0^{-s-\frac1p+\beta+\frac12+},1\right\}+N_0^{-s-\frac1p+\beta-\frac32+}\right).
\end{align*}
For simplicity, we denote $A= -s-\frac1p+\beta$ and set $N_0=\tau^{-\frac12}$, then it follows that
\begin{align*}
\|v(t)-v(t_n)\|_{L^\infty_tH^\beta_x([t_n,t_{n+1}]\times\T)}\leq C\tau\max\left\{\tau^{-\frac A2-\frac14},1\right\},
\end{align*}
which finishes the proof of the lemma.
\end{proof}

\begin{lemma}[Local error from $R_1^n$]\label{lem:est-Rn1}
Let $\xi\in\hat{b}^{s,p}$ with the conditions of \cref{thm:main-N1} satisfied, so $u,v\in L^\infty_t H_x^{s+\gamma_p-}([0,T])$. Then, we have
\begin{align}
\big\|R_1^n\big\|_{L^2}
\le  C\tau^{1+\min\left\{s+\frac1p+\frac14-,1\right\}},
\end{align}
where the constant $C>0$ only depends on $\|\xi\|_{\hat b^{s,p}}$ and $\|u\|_{L^\infty_t H_x^{s+\gamma_p-}}$.
\end{lemma}
\begin{proof}
The analysis goes separately for the case $s+\frac1p>1$ and the case $s+\frac1p\leq1$.

\emph{Case I: $s+\frac1p>1$.}
Firstly, we consider the simpler case: $s+\frac1p>1$. In this case, by Cauchy-Schwartz's inequality, we have
$$
\big\|\xi\big\|_{L^\infty} \le \big\|\hat \xi\big\|_{l^1}
\le \big\|\langle k \rangle^{-s}\big\|_{l^{p'}} \|\xi\|_{\hat b^{s,p}}
\lesssim \|\xi\|_{\hat b^{s,p}}.
$$
By this inequality,  we find that
\begin{align*}
\big\|R_1^n\big\|_{L^2}
\lesssim &
\int_{t_n}^{t_{n+1}}\Big\|\fe^{-i\rho\partial_x^2}\Big[\xi\> \fe^{i\rho\partial_x^2}\big(v(\rho)-v(t_n)\big)\Big]\Big\|_{L^2}\,d\rho\\
\lesssim &
\int_{t_n}^{t_{n+1}}\Big\|\xi\> \fe^{i\rho\partial_x^2}\big(v(\rho)-v(t_n)\big)\Big\|_{L^2}\,d\rho\\
\lesssim &
\tau \|\xi\|_{L^\infty} \|v(\rho)-v(t_n)\|_{L^\infty_tL^2_x([t_n,t_{n+1}])}\\
\lesssim &
\tau \|\xi\|_{\hat b^{s,p}} \|v(\rho)-v(t_n)\|_{L^\infty_tL^2_x([t_n,t_{n+1}])}.
\end{align*}
By \cref{lem:v-s-tn}, it infers that
$$
\|v(\rho)-v(t_n)\|_{L^\infty_tL^2_x([t_n,t_{n+1}]\times\T)}
\le C \tau,
$$
and this further yields
\begin{align*}
\big\|R_1^n\big\|_{L^2}
\le
C\tau^2.
\end{align*}
This gives the desired estimate of the lemma.

\emph{Case II: $s+\frac1p\leq1$.} Now, we turn to consider the case: $s+\frac1p\le 1$. By taking the Fourier transform, we write
\begin{align*}
\widehat{(R_1^n)}_k=&i\int_{t_n}^{t_{n+1}}\sum\limits_{k_1+k_2=k}\fe^{i\rho(k^2-k_2^2)}\hat \xi_{k_1}
\big(\hat{v}_{k_2}(\rho)-\hat{v}_{k_2}(t_n)\big)\,d\rho\\
=&i\int_{t_n}^{t_{n+1}}\sum\limits_{\substack{k_1+k_2=k\\|k_1|\lesssim |k_2|}}\fe^{i\rho(k^2-k_2^2)}\hat \xi_{k_1}
\big(\hat{v}_{k_2}(\rho)-\hat{v}_{k_2}(t_n)\big)\,d\rho\\
&\quad +i\int_{t_n}^{t_{n+1}}\sum\limits_{\substack{k_1+k_2=k\\|k_1|\gg |k_2|}}\fe^{i\rho(k^2-k_2^2)}\hat \xi_{k_1}
\big(\hat{v}_{k_2}(\rho)-\hat{v}_{k_2}(t_n)\big)\,d\rho\\
\triangleq &
\widehat{(R_{11})}_k+\widehat{(R_{12})}_k.
\end{align*}

For $R_{11}$, by duality and Plancherel's identity we have
\begin{align*}
\|R_{11}\|_{L^2}
\lesssim &\tau \Bigg\|\sum\limits_{\substack{k=k_1+k_2\\|k_1|\lesssim |k_2|}}\big|\hat{\xi}_{k_1}\big|
\big|\hat{v}_{k_2}(\rho)-\hat{v}_{k_2}(t_n)\big|\Bigg\|_{L^\infty_t l_k^2([t_n,t_{n+1}])}\\
\lesssim &
\tau \Bigg\| \sup_{h:\|h\|_{l^2}=1} \sum\limits_{\substack{k=k_1+k_2\\|k_1|\lesssim |k_2|}} \big|\hat{\xi}_{k_1}\big|
|\hat{v}_{k_2}(\rho)-\hat{v}_{k_2}(t_n)| |h_{k_1+k_2}|\Bigg\|_{L^\infty_t([t_n,t_{n+1}])}.
\end{align*}
As before, we denote $\frac1r=\frac12-\frac1p$ and  $a=-s+\frac1r+\frac12+=-s-\frac1p+1+$. If $s\le \frac1r$,
the strategy here is to first sum up $k_1$ and then sum up $k_2$.
This gives
\begin{align}\label{est:R11-case1}
 \|R_{11}\|_{L^2}
  \lesssim  &
  \tau \|\xi\|_{\hat b^{s,p}}\Big\| \sum_{k_2}
 \big\|\langle k_1\rangle^{-s}\big\|_{l^r(\{k_1:|k_1|\lesssim |k_2|\})} |\hat{v}_{k_2}(\rho)-\hat{v}_{k_2}(t_n)|  \Big\|_{L^\infty_t([t_n,t_{n+1}])}\notag\\
   \lesssim  &
  \tau \|\xi\|_{\hat b^{s,p}}\Big\| \sum_{k_2}
\langle k_2\rangle^{-s+\frac1r+} |\hat{v}_{k_2}(\rho)-\hat{v}_{k_2}(t_n)|  \Big\|_{L^\infty_t([t_n,t_{n+1}])}\notag\\
  \lesssim  &
  \tau \|\xi\|_{\hat b^{s,p}}\|v(t)-v(t_n)\|_{L^\infty_tH^a_x([t_n,t_{n+1}])}.
\end{align}
 If $s> \frac1r$, now our strategy is to first sum up $k_2$ and then sum up $k_1$.  This gives
\begin{align*}
 \|R_{11}\|_{L^2}
  \lesssim  &
  \tau \Bigg\| \sup_{h:\|h\|_{l^2}=1} \sum\limits_{\substack{k=k_1+k_2\\|k_1|\lesssim |k_2|}} \langle k_1\rangle^s\big|\hat{\xi}_{k_1}\big|\>
\langle k_2\rangle^a|\hat{v}_{k_2}(\rho)-\hat{v}_{k_2}(t_n)| |h_{k_1+k_2}| \langle k_1\rangle^{-s} \langle k_2\rangle^{-a} \Bigg\|_{L^\infty_t}\\
\lesssim &
  \tau \|v(t)-v(t_n)\|_{L^\infty_tH^a_x([t_n,t_{n+1}]\times\T)}  \sum_{k_1}
\langle k_1\rangle^s\big|\hat{\xi}_{k_1}\big|\>  \langle k_1\rangle^{-s-a}\notag\\
   \lesssim  &
  \tau \|\xi\|_{\hat b^{s,p}}\|v(t)-v(t_n)\|_{L^\infty_tH^a_x([t_n,t_{n+1}]\times\T)}  \Big\| \langle k_1\rangle^{-s-a} \Big\|_{l^{p'}_{k_1}}.
\end{align*}
Since $s> \frac1r$, we have that $s+a>\frac1{p'}$, and thus it gives
\begin{align*}
 \|R_{11}\|_{L^2}
  \lesssim
  \tau \|\xi\|_{\hat b^{s,p}}\|v(t)-v(t_n)\|_{L^\infty_tH^a_x([t_n,t_{n+1}]\times\T)}.
\end{align*}
Combining with \eqref{est:R11-case1}, we have that for either $s\le \frac1r$ or $s>\frac1r$,
\begin{align}\label{est:R11}
 \|R_{11}\|_{L^2}
  \lesssim
  \tau \|\xi\|_{\hat b^{s,p}}\|v(t)-v(t_n)\|_{L^\infty_tH^a_x([t_n,t_{n+1}]\times\T)}.
\end{align}

For $R_{12}$, we first consider the low-frequency part to have
\begin{align*}
  \|P_{\leq N_0}R_{12}\|_{L^2}
  \lesssim &
   \tau\Bigg\|\sum_{\substack{k=k_1+k_2\\|k_1|\gg |k_2|}}\big|\hat{\xi}_{k_1}\big|
  |\hat{v}_{k_2}(\rho)-\hat{v}_{k_2}(t_n)|\Bigg\|_{L^\infty_tl_k^2([t_n,t_{n+1}]\times\{k:|k|\leq N_0\})}\\
  \lesssim &
  \tau\Bigg\|\sup\limits_{h:\|h\|_{l^2}=1}\sum_{\substack{|k_1|\gg|k_2|\\|k_1+k_2|\leq N_0}}\big|\hat{\xi}_{k_1}\big|
  |\hat{v}_{k_2}(\rho)-\hat{v}_{k_2}(t_n)||h_{k_1+k_2}|\Bigg\|_{L^\infty_t([t_n,t_{n+1}])}\\
  \lesssim&
  \tau\|\xi\|_{\hat{b}^{s,p}}\Big\|\sum_{k_2}\big\|\langle k_1\rangle^{-s}\big\|_{l_{k_1}^r(\{k_1:|k_2|\ll |k_1|\lesssim N_0\})}
  |\hat{v}_{k_2}(\rho)-\hat{v}_{k_2}(t_n)|\Big\|_{L^\infty_t}.
\end{align*}
Notice that
\begin{equation*}
\big\|\langle k_1\rangle^{-s}\big\|_{l_{k_1}^r(\{k_1:|k_2|\ll |k_1|\lesssim N_0\})}
\lesssim
\left\{\begin{split}&N_0^{-s+1/r+},\quad\,-s+1/r\geq0,\\
&\langle k_2\rangle^{-s+1/r},\quad   -s+1/r<0.\end{split}\right.
\end{equation*}
When $s+\frac1p\leq\frac12\Leftrightarrow s\leq \frac1r$, we can then have
\begin{align}\label{est:R12-low-case1}
  \|P_{\leq N_0}R_{12}\|_{L^2}
   \lesssim&
  \tau\|\xi\|_{\hat{b}^{s,p}}\Big\|\sum_{k_2}
  |\hat{v}_{k_2}(\rho)-\hat{v}_{k_2}(t_n)|\Big\|_{L^\infty_t}\notag\\
  \lesssim &
  \tau N_0^{-s+\frac1r+}\|\xi\|_{\hat{b}^{s,p}}\|v(\rho)-v(t_n)\|_{L^\infty_tH_x^{\frac12+}}.
  \end{align}
When $s+\frac1p>\frac12\Leftrightarrow s> \frac1r$, noting that $a=-s+\frac1r+\frac12+$,  we then find
\begin{align}\label{est:R12-low-case2}
  \|P_{\leq N_0}R_{12}\|_{L^2}
   \lesssim&
  \tau\|\xi\|_{\hat{b}^{s,p}}\Big\|\sum_{k_2}\langle k_2\rangle^{-s+\frac1r}
  |\hat{v}_{k_2}(\rho)-\hat{v}_{k_2}(t_n)|\Big\|_{L^\infty_t}\notag\\
  \lesssim & \tau \|\xi\|_{\hat{b}^{s,p}}\|v(\rho)-v(t_n)\|_{L^\infty_tH_x^{a}}.
  \end{align}

Now we consider the high-frequency part of $R_{12}$.
By integration-by-parts we have
\begin{align*}
  \widehat{(R_{12})}_k=&i\sum_{\substack{k=k_1+k_2\\|k_1|\gg |k_2|}}\frac{\fe^{it_{n+1}(k^2-k_2^2)}}{i(k^2-k_2^2)}
  \hat{\xi}_{k_1}[\hat{v}_{k_2}(t_{n+1})-\hat{v}_{k_2}(t_n)]
  \\
   &-i\int_{t_n}^{t_{n+1}}\sum\limits_{\substack{k_1+k_2=k\\|k_1|\gg |k_2|}}\frac{\fe^{i\rho(k^2-k_2^2)}}{i(k^2-k_2^2)}\hat \xi_{k_1}\partial_\rho
\big[\hat{v}_{k_2}(\rho)-\hat{v}_{k_2}(t_n)\big]\,d\rho\\
\triangleq &
 \widehat{(R_{121})}_k+\widehat{(R_{122})}_k.
\end{align*}
For $R_{121}$, by duality and Plancherel's identity  we have
\begin{align*}
  \big\|P_{>N_0}R_{121}\big\|_{L^2}
  \lesssim &
  \sup\limits_{h:\|h\|_{l^2}=1}\sum_{\substack{|k_1|\gg |k_2|\\|k_1+k_2|> N_0}}\frac{1}{k_1^2}
  \big|\hat{\xi}_{k_1}\big||\hat{v}_{k_2}(t_{n+1})-\hat{v}_{k_2}(t_n)||h_{k_1+k_2}|
  \\
   =&
   \sup\limits_{h:\|h\|_{l^2}=1}\sum_{\substack{|k_1|\gg |k_2|\\|k_1+k_2|> N_0}}\langle k_1\rangle^s
  \big|\hat{\xi}_{k_1}\big||\hat{v}_{k_2}(t_{n+1})-\hat{v}_{k_2}(t_n)||h_{k_1+k_2}| \langle k_1\rangle^{-s-2}\\
  \lesssim &\|v(t_{n+1})-v(t_n)\|_{L_x^{2}} \sum_{k_1:|k_1|\gtrsim N_0}\langle k_1\rangle^s  \big|\hat{\xi}_{k_1}\big|\>
  \langle k_1\rangle^{-s-2}\\
   \lesssim &
  \|\xi\|_{\hat{b}^{s,p}}\|v(t_{n+1})-v(t_n)\|_{L_x^{2}} \big\|\langle k_1\rangle^{-s-2}\big\|_{l^{p'}(\{k_1:|k_1|\gtrsim N_0\})}.
\end{align*}
Note  that $-s-2+\frac1{p'}=-s-\frac1p-1<0$. Then we further have
\begin{align}\label{est:R121}
  \big\|P_{>N_0}R_{121}\big\|_{L^2}
  \lesssim  N_0^{-s-\frac1p-1}\|\xi\|_{\hat{b}^{s,p}}\|v(t_{n+1})-v(t_n)\|_{L_x^{2}}.
  \end{align}

 For $R_{122}$, by noting $\partial_tv=i\fe^{-it\partial_x^2}[\xi u+|u|^2u]$, we have
\begin{align*}
\widehat{(R_{122})}_k=&\int_{t_n}^{t_{n+1}}\sum\limits_{\substack{k_2+k_3=\tilde{k}_2\\|k_1|\gg |\tilde{k}_2|}}\frac{\fe^{i\rho(k^2-\tilde{k}_2^2)}}{i(k^2-\tilde{k}_2^2)}\hat\xi_{k_1}
\hat\xi_{k_2}\hat{v}_{k_3}(\rho)d\rho\\
&+\int_{t_n}^{t_{n+1}}\sum\limits_{\substack{|k_1|\gg |k_2|}}\frac{\fe^{i\rho k_1^2}}{i(k^2-k_2^2)}\hat\xi_{k_1}
\widehat{(|u|^2u)}_{k_2}(\rho)d\rho\\
\triangleq  & \widehat{(R_{1221})}_k+\widehat{(R_{1222})}_k,
\end{align*}
and we shall estimate the two separately.

For $R_{1221}$,  similarly as before we have
\begin{align*}
&\|P_{>N_0}R_{1221}\|_{L^2}\\
\lesssim &
\tau\Bigg\|\sup_{h:\|h\|_{l^2}=1}\sum\limits_{\substack{k_1, k_2, k_3\\|k_1|\gg |k_2+k_3|,|k_1+k_2+k_3|>N_0}}
\frac{\big|\hat\xi_{k_1}\big|
\big|\hat\xi_{k_2}\big|}{(k_1+k_2+k_3)^2}|\hat{v}_{k_3}(t)||h_{k_1+k_2+k_3}|\Bigg\|_{L_t^\infty([0,T])}.
\end{align*}
Now we  change the variables and use the relationship $|k_1+k_2+k_3|\sim |k_1|$ to obtain that
\begin{align*}
&\|P_{>N_0}R_{1221}\|_{L^2}\\
\lesssim &
\tau\Bigg\|\sup_{h:\|h\|_{l^2}=1}\sum\limits_{\substack{k_1, \tilde{k}_2, k_3\\|k_1|\gg |\tilde k_2|, |k_1|\gtrsim N_0}}
\langle k_1\rangle^{-2} \big|\hat\xi_{k_1}\big|
\big|\hat\xi_{\tilde{k}_2-k_3}\big||\hat{v}_{k_3}| \big|h_{k_1+\tilde k_2}\big|  \Bigg\|_{L_t^\infty([0,T])}\\
=&\tau\Bigg\|\sup_{h:\|h\|_{l^2}=1}\sum\limits_{\substack{k_1, \tilde{k}_2, k_3\\|k_1|\gg |\tilde k_2|, |k_1|\gtrsim N_0}}
\langle k_1\rangle^s\big|\hat\xi_{k_1} \big| \> \langle\tilde{k}_2-k_3\rangle^s
\big|\hat\xi_{\tilde{k}_2-k_3}\big| \> |\hat{v}_{k_3}| \> \big|h_{k_1+\tilde k_2}\big|\\
&\qquad\qquad\qquad \qquad\qquad \cdot \langle k_1\rangle^{-2-s}\langle \tilde{k}_2-k_3\rangle^{-s}\Bigg\|_{L_t^\infty([0,T])}\\
\lesssim&
\tau \|\xi\|_{\hat{b}^{s,p}}\sum_{\tilde k_2, k_3}  \big\| \langle k_1\rangle^{-2-s}\big\|_{l^r(\{k_1: |k_1|\gtrsim \max\{|\tilde k_2|,N_0\}\})}
\langle\tilde{k}_2-k_3\rangle^{-s}  \langle k_1\rangle^s\big|\hat\xi_{k_1}\big| \>   |\hat{v}_{k_3}(t)|.
\end{align*}
Note that
$$
\big\| \langle k_1\rangle^{-2-s}\big\|_{l^r(\{k_1: |k_1|\gtrsim \max\{|\tilde k_2|,N_0\}\})}
\lesssim N_0^{-2s-\frac2p-\frac12+} \big|\tilde k_2\big|^{s+\frac1p-1-} .
$$
This gives
\begin{align*}
\|P_{>N_0}R_{1221}\|_{L^2}
\lesssim &
\tau N_0^{-2s-\frac2p-\frac12+} \|\xi\|_{\hat{b}^{s,p}}\sum_{\tilde k_2, k_3}\big\langle\tilde k_2\big\rangle^{s+\frac1p-1-} \big\langle \tilde{k}_2-k_3\big\rangle^{-s}    \langle \tilde{k}_2-k_3\rangle^s\big|\hat\xi_{\tilde{k}_2-k_3}\big| \>   |\hat{v}_{k_3}(t)|\\
\lesssim &
\tau  N_0^{-2s-\frac2p-\frac12+}  \|\xi\|_{\hat{b}^{s,p}}^2 \Big\|\big\langle\tilde k_2\big\rangle^{s+\frac1p-1-} \big\langle\tilde{k}_2-k_3\big\rangle^{-s}\Big\|_{l^{p'}_{\tilde{k}_2}} \sum_{ k_3}  |\hat{v}_{k_3}(t)|.
\end{align*}
Since now $s+\frac1p\le 1$, we have
$$
\Big\|\big\langle\tilde k_2\big\rangle^{s+\frac1p-1-} \big\langle\tilde{k}_2-k_3\big\rangle^{-s}\Big\|_{l^{p'}_{\tilde{k}_2}}
 \lesssim 1.
$$
This further implies
\begin{align*}
\|P_{>N_0}R_{1221}\|_{L^2}
\lesssim &
\tau N_0^{-2s-\frac2p-\frac12+}  \|\xi\|_{\hat{b}^{s,p}}^2 \sum_{ k_3}  |\hat{v}_{k_3}(t)|\\
\lesssim &
\tau N_0^{-2s-\frac2p-\frac12+}  \|\xi\|_{\hat{b}^{s,p}}^2\|v\|_{L^\infty_tH^{\frac12+}_x([0,T])}  .
\end{align*}

For $R_{1222}$,  we have  similarly
\begin{align*}
\|P_{>N_0}R_{1222}\|_{L^2}
\lesssim &
\tau\Bigg\|\sup_{h:\|h\|_{l^2}=1}\sum\limits_{\substack{|k_1|\gg |k_2|,|k_1|\gtrsim N_0}}\frac{1}{(k_1+k_2)^2}
\big|\hat\xi_{k_1}\big|\widehat{(|u|^2u)}_{k_2}| |h_{k_1+k_2}|\Bigg\|_{L_t^\infty([0,T])}\\
\lesssim &
\tau\Bigg\|\sup_{h:\|h\|_{l^2}=1}\sum\limits_{\substack{|k_1|\gg |k_2|,|k_1|\gtrsim N_0}}\langle k_1\rangle^s
\big|\hat\xi_{k_1}\big|\>\big|\widehat{(|u|^2u)}_{k_2}\big||h_{k_1+k_2}|   \langle k_1\rangle^{-2-s}\Bigg\|_{L_t^\infty([0,T])}\\
\lesssim &
\tau\|\xi\|_{\hat{b}^{s,p}}\big\||k|^{-2-s}\big\|_{l^r(|k|\gtrsim N_0)}\sum_{k_2}\big|\widehat{(|u|^2u)}_{k_2}\big|\\
\lesssim&
\tau N_0^{-2-s+\frac1r}\|\xi\|_{\hat{b}^{s,p}}\|u^3\|_{L^\infty_tH_x^{\frac12+}}
\lesssim
\tau N_0^{-s-\frac1p-\frac32}\|\xi\|_{\hat{b}^{s,p}}\|u\|_{L^\infty_tH_x^{\frac12+}}^3.
\end{align*}
Since $s+\frac1p\le 1$, we have that $-s-\frac1p-\frac32\le -2s-\frac2p-\frac12$, which implies that
\begin{align*}
\|P_{>N_0}R_{1222}\|_{L^2}
\lesssim &
\tau N_0^{-2s-\frac2p-\frac12+}\|\xi\|_{\hat{b}^{s,p}}\|u\|_{L^\infty_tH_x^{\frac12+}}^3.
\end{align*}
Collecting the estimates on $R_{1221}$ and $R_{1222}$, we have
\begin{align}\label{est:R122}
\|P_{>N_0}R_{122}\|_{L^2}\le C \tau N_0^{-2s-\frac2p-\frac12+}.
\end{align}

Together with \eqref{est:R11},  \eqref{est:R12-low-case1}, \eqref{est:R12-low-case2},  \eqref{est:R121} and \eqref{est:R122},
we find in total for $s+\frac1p\leq\frac12$ that
\begin{align*}
 \|R_1^n\|\leq& C\tau\|v(t)-v(t_n)\|_{L_t^\infty H_x^a([t_n,t_{n+1}]\times\T)}+
 C\tau N_0^{-s-\frac1p+\frac12+}\|v(t)-v(t_n)\|_{L_t^\infty H_x^{\frac12+}([t_n,t_{n+1}]\times\T)}\\
 &+CN_0^{-s-\frac1p-1}\|v(t_{n+1})-v(t_n)\|_{L^2}+C\tau N_0^{-2s-\frac2p-\frac12+},
\end{align*}
with $a=-s-\frac1p+1+$. For $\frac12<s+\frac1p\le 1$,
\begin{align*}
 \|R_1^n\|\leq& C\tau\|v(t)-v(t_n)\|_{L_t^\infty H_x^a([t_n,t_{n+1}]\times\T)}+
 CN_0^{-s-\frac1p-1}\|v(t_{n+1})-v(t_n)\|_{L^2}\\
 &+C\tau N_0^{-2s-\frac2p-\frac12+}.
\end{align*}
Now we can apply \cref{lem:v-s-tn} to the above findings to get
\begin{align*}
 &\|v(t)-v(t_n)\|_{L_t^\infty H_x^a([t_n,t_{n+1}]\times\T)}\leq C\min\left\{\tau^{\frac14+s+\frac1p-},\tau\right\},\\
& \|v(t)-v(t_n)\|_{L_t^\infty H_x^{\frac12+}([t_n,t_{n+1}]\times\T)}\le C\min\left\{\tau^{\frac12(s+\frac1p)-},\tau\right\},\\
 &\|v(t_{n+1})-v(t_n)\|_{L^2}\le  C\min\left\{\tau^{\frac34+\frac12(s+\frac1p)-},\tau\right\}.
\end{align*}
Therefore, by choosing $N_0=\tau^{-\frac12}$, we find that for $s+\frac1p\leq\frac12$,
\begin{align*}
 \|R_1^n\|_{L^2}\leq& C\tau \tau^{\frac14+(s+\frac1p)-}+C \tau \tau^{\frac12(s+\frac1p)+\frac14-}\tau^{\frac12(s+\frac1p)-}\\
 &\quad +C\tau^{\frac12(s+\frac1p)+\frac12}\tau^{\frac34+\frac12(s+\frac1p)-}+C\tau \tau^{s+\frac1p+\frac14-}\\
 \le &C\tau \tau^{s+\frac1p+\frac14-}.
\end{align*}
For $\frac12<s+\frac1p\leq1$,
\begin{align*}
 \|R_1^n\|_{L^2}\leq& C\tau \min\big\{\tau^{s+\frac1p+\frac14-},\tau\big\}
  +C\tau^{\frac12(s+\frac1p)+\frac12}\min\big\{\tau^{\frac34+\frac12(s+\frac1p)-},\tau\big\}
  +C\tau \tau^{s+\frac1p+\frac14-}\\
 \le &C\tau \min\big\{\tau^{s+\frac1p+\frac14-},\tau\big\}.
\end{align*}
Together with the two estimates above, we obtain that for $0\le s+\frac1p\leq1$,
\begin{align*}
 \|R_1^n\|_{L^2}
 \le &C\tau \min\left\{\tau^{\frac14+s+\frac1p-},\tau\right\}.
\end{align*}
This finishes the proof of the lemma.
\end{proof}

\begin{lemma}[Local error from $R_2^n$]\label{lem:est-Rn2}
Let $\xi\in \hat b^{s,p}$ with the conditions of \cref{thm:main-N1} satisfied, and $u\in L^\infty_t H_x^{s+\gamma_p-}([0,T]\times\T)$. Then,
\begin{align}
\big\|R_2^n\big\|_{L^2}
\le  C\tau^{1+\min\{\frac14+\frac1p+s-,1\}},
\end{align}
where the constant $C>0$ only depends on $\|\xi\|_{\hat b^{s,p}}$ and $\|u\|_{L^\infty_t H_x^{s+\gamma_p-}}$.
\end{lemma}
\begin{proof} Note that
$$
\big|\widehat{(R^n_2)}_k \big|
\lesssim
\tau \sum\limits_{k_1+k_2=k}\Big| \mathcal M_\tau\big(\fe^{is(k^2-k_2^2)}\big)- \mathcal M_\tau\big(\fe^{isk^2}\big)\mathcal M_\tau\big(\fe^{-isk_2^2}\big)\Big| \big|\hat \xi_{k_1}\big|
|\hat v_{k_2}(t_n)|.
$$
If $k=0$ or $k_2=0$, then
$$
 \mathcal M_\tau\big(\fe^{is(k^2-k_2^2)}\big)- \mathcal M_\tau\big(\fe^{isk^2}\big)\mathcal M_\tau\big(\fe^{-isk_2^2}\big)=0,
$$
and thus $R^n_2=0$.  Therefore, we may assume that $k\ne 0$ and $k_2\ne 0$ in the following.
By \cref{lem:average2}, we have that for any $0\le a\le 1$,
\begin{align*}
\Big| \mathcal M_\tau\big(\fe^{is(k^2-k_2^2)}\big)- \mathcal M_\tau\big(\fe^{isk^2}\big)\mathcal M_\tau\big(\fe^{-isk_2^2}\big)\Big|
\lesssim & \tau^a \min\{|k|^{2}|k_2|^{2a-2},|k_2|^{2}|k|^{2a-2}\}.
\end{align*}
In the following, we set
$$
a=\min\left\{\frac14+\frac1p+s-,1\right\}.
$$
Accordingly, we write
\begin{subequations}\label{Rn2}
 \begin{align}
\big|\widehat{(R^n_2)}_k \big|
\lesssim &
\tau \sum\limits_{\substack{k_1+k_2=k\\|k|\lesssim |k_2|}}\Big| \mathcal M_\tau\big(\fe^{is(k^2-k_2^2)}\big)- \mathcal M_\tau\big(\fe^{isk^2}\big)\mathcal M_\tau\big(\fe^{-isk_2^2}\big)\Big| \big|\hat \xi_{k_1}\big|
|\hat v_{k_2}(t_n)|\label{Rn2-1}\\
&+\tau \sum\limits_{\substack{k_1+k_2=k\\ |k|\gg |k_2|}}\Big| \mathcal M_\tau\big(\fe^{is(k^2-k_2^2)}\big)- \mathcal M_\tau\big(\fe^{isk^2}\big)\mathcal M_\tau\big(\fe^{-isk_2^2}\big)\Big| \big|\hat \xi_{k_1}\big||\hat v_{k_2}(t_n)|.
\label{Rn2-2}
 \end{align}
\end{subequations}
For \eqref{Rn2-1}, we apply the estimate:
\begin{align*}
\Big| \mathcal M_\tau\big(\fe^{is(k^2-k_2^2)}\big)- \mathcal M_\tau\big(\fe^{isk^2}\big)\mathcal M_\tau\big(\fe^{-isk_2^2}\big)\Big|
\lesssim & \tau^a \langle k\rangle^{2} |k_2|^{2a-2},
\end{align*}
and then by duality, we have
\begin{equation}\label{est:Rn2-1}
 \begin{aligned}
\|\eqref{Rn2-1}\|_{l^2_k}
\lesssim &\tau^{1+a} \Bigg\|\sum\limits_{\substack{k_1+k_2=k\\|k|\lesssim  |k_2|}}
\langle k\rangle^{2}\langle k_2\rangle^{2a-2}
\big|\hat \xi_{k_1}\big| |\hat v_{k_2}(t_n)|\Bigg\|_{l^2_k}\\
\lesssim &
\tau^{1+a}   \sup\limits_{h:\|h\|_{l^2}=1}
\sum\limits_{\substack{k_1+k_2=k\\|k_1|\lesssim  |k_2|}}
\langle k\rangle^{2}\langle k_2\rangle^{2a-2}
\big|\hat \xi_{k_1}\big| |\hat v_{k_2}(t_n)| |h_k|\\
\lesssim &
\tau^{1+a}   \sup\limits_{h:\|h\|_{l^2}=1}
\sum\limits_{\substack{k_1,k_2\\|k|\lesssim |k_2|}}
\langle k_1\rangle^{-s}\langle k_1+k_2\rangle^{2}\langle k_2\rangle^{2a-2-s-\gamma_p-} \\
&\qquad\qquad\qquad\qquad\cdot
\langle k_1\rangle^s\big|\hat \xi_{k_1}\big|\>\langle k_2\rangle^{s+\gamma_p-} |\hat v_{k_2}(t_n)| |h_{k_1+k_2}|.
\end{aligned}
\end{equation}
In the case when $s+\frac1p\le \frac12$, by Cauchy-Schwartz's inequality, we further get
\begin{align*}
\|\eqref{Rn2-1}\|_{l^2_k}
 \lesssim &
\tau^{1+a}   \sup\limits_{h:\|h\|_{l^2}=1}
\sum\limits_{k_2}
\sum\limits_{k_1: |k_1|\lesssim |k_2|}
\langle k_1\rangle^{-s}\langle k_2\rangle^{2a-s-\gamma_p-} \\
&\qquad\qquad\qquad\qquad\qquad\quad\cdot
\langle k_1\rangle^s\big|\hat \xi_{k_1}\big|\>\langle k_2\rangle^{s+\gamma_p-} |\hat v_{k_2}(t_n)||h_{k_1+k_2}| \\
\lesssim &
\tau^{1+a}  \big\|\xi\big\|_{\hat b^{s,p}}
\sum\limits_{k_2}\big\|\langle k_1\rangle^{-s}\big\|_{l^r\{k_1: |k_1|\lesssim |k_2|\}}
\langle k_2\rangle^{2a-s-\gamma_p-}
\langle k_2\rangle^{s+\gamma_p-} |\hat v_{k_2}(t_n)| \\
\lesssim &
\tau^{1+a}  \big\|\xi\big\|_{\hat b^{s,p}}
\sum\limits_{k_2}
\langle k_2\rangle^{2a-2s-\gamma_p+\frac1r-}
\langle k_2\rangle^{s+\gamma_p-} |\hat v_{k_2}(t_n)|,
\end{align*}
where $\frac1r=\frac12-\frac1p$. Note that
$$
2a-2s-\gamma_p+\frac1r<-\frac12,
$$
we further find
\begin{align*}
\|\eqref{Rn2-1}\|_{l^2_k}
\lesssim &
\tau^{1+a}  \|\xi\|_{\hat b^{s,p}}
\|v(t_n)\|_{H^{s+\gamma_p-}_x}.
\end{align*}

\noindent In the case when $s+\frac1p> \frac12$, noting that
$$
2a-s-\gamma_p<0\quad \mbox{ and } \quad 2a-2s-\gamma_p+\frac1{p'}<0,
$$
 by \eqref{est:Rn2-1} and Cauchy-Schwartz's inequality, we get
\begin{align*}
\|\eqref{Rn2-1}\|_{l^2_k}
 \lesssim &
\tau^{1+a}   \sup\limits_{h:\|h\|_{l^2}=1}
\sum\limits_{k_1}
\sum\limits_{k_2: |k_2|\gtrsim |k_1|}
\langle k_1\rangle^{2a-2s-\gamma_p-} \langle k_1\rangle^s\big|\hat \xi_{k_1}\big|\\
&\qquad\qquad\qquad\qquad\qquad\quad\cdot
\>\langle k_2\rangle^{s+\gamma_p-} |\hat v_{k_2}(t_n)||h_{k_1+k_2}| \\
\lesssim &
\tau^{1+a}  \|v(t_n)\|_{H^{s+\gamma_p-}_x}
\sum\limits_{k_1}\langle k_1\rangle^{2a-2s-\gamma_p-} \big|\hat \xi_{k_1}\big| \\
\lesssim &
\tau^{1+a}  \|\xi\|_{\hat b^{s,p}} \|v(t_n)\|_{H^{s+\gamma_p-}_x}
\big\|\langle k_1\rangle^{2a-2s-\gamma_p-}  \big\|_{l^{p'}_{k_1}}\\
\lesssim &
\tau^{1+a}  \|\xi\|_{\hat b^{s,p}} \|v(t_n)\|_{H^{s+\gamma_p-}_x}.
\end{align*}
Therefore, together with the estimates in the above two cases, we finally get that
$$
\|\eqref{Rn2-1}\|_{l^2_k}
\lesssim
\tau^{1+a}  \|\xi\|_{\hat b^{s,p}} \|v(t_n)\|_{H^{s+\gamma_p-}_x}.
$$

For \eqref{Rn2-2}, we apply the estimate:
\begin{align*}
\Big| \mathcal M_\tau\big(\fe^{is(k^2-k_2^2)}\big)- \mathcal M_\tau\big(\fe^{isk^2}\big)\mathcal M_\tau\big(\fe^{-isk_2^2}\big)\Big|
\lesssim & \tau^a|k|^{2a-2}|k_2|^{2}.
\end{align*}
Again, by duality and noting that $|\xi_1|\sim |\xi_2|$, we have
\begin{equation*}
 \begin{aligned}
\|\eqref{Rn2-2}\|_{l^2_k}
\lesssim &\tau^{1+a} \Bigg\|\sum\limits_{\substack{k_1+k_2=k\\|k|\gg  |k_2|}}
\langle k_1\rangle^{2a-2}\langle k_2\rangle^{2}
\big|\hat \xi_{k_1}\big| |\hat v_{k_2}(t_n)|\Bigg\|_{l^2_k}\\
\lesssim &
\tau^{1+a}   \sup\limits_{h:\|h\|_{l^2}=1}
\sum\limits_{\substack{k_1+k_2=k\\|k_1|\gg  |k_2|}}
\langle k_1\rangle^{2a-2}\langle k_2\rangle^{2}
\big|\hat \xi_{k_1}\big| |\hat v_{k_2}(t_n)| |h_k|\\
\lesssim &
\tau^{1+a}   \sup\limits_{h:\|h\|_{l^2}=1}
\sum\limits_{\substack{k_1,k_2\\|k|\gg |k_2|}}
\langle k_1\rangle^{-s+2a-2}\langle k_2\rangle^{-s-\gamma_p-}\langle k_1\rangle^s\big|\hat \xi_{k_1}\big| \\
&\qquad\qquad\qquad\qquad\cdot
\>\langle k_2\rangle^{s+\gamma_p-} |\hat v_{k_2}(t_n)| |h_{k_1+k_2}|.
\end{aligned}
\end{equation*}
Note that
$$
-s+2a-2+\frac1r<0\quad \mbox{ and }\quad -s-\gamma_p+\frac12<0,
$$
then by Cauchy-Schwartz's inequality, we have
\begin{align*}
\|\eqref{Rn2-2}\|_{l^2_k}
 \lesssim &
\tau^{1+a}\big\|\langle k_1\rangle^{-s+2a-2}\big\|_{l^r_{k_1}}
 \big\|\langle k_2\rangle^{-s-\gamma_p+}\big\|_{l^2_{k_2}}
\|\xi\|_{\hat b^{s,p}} \|v(t_n)\|_{H^{s+\gamma_p-}_x} \\
\lesssim &
\tau^{1+a} \|\xi\|_{\hat b^{s,p}}  \|v(t_n)\|_{H^{s+\gamma_p-}_x}.
\end{align*}

Therefore, combining with two estimates on \eqref{Rn2}, we obtain
 \begin{align*}
\Big\|\widehat{(R^{n}_{2})}_k\Big\|_{l^2_k}
\lesssim \tau^{1+a} \|\xi\|_{\hat b^{s,p}}  \|v(t_n)\|_{H^{s+\gamma_p-}_x}.
\end{align*}
This together with Plancherel's identity give the proof of the lemma.
\end{proof}

\begin{lemma}[Local error of splitting part]
\label{lem:splitting}
Under assumptions of \cref{thm:main-N1}, for the truncation term $R^n_3$ defined in \eqref{Apprx:T2}, we have
\begin{align}\label{lem:splitting est}
\big\|R_3^n\big\|_{L^2}
\le C\tau\left(\tau^{\min\left\{1,\frac12(s+\gamma_p)-\right\}}+N^{-s-\gamma_p+}\right),
\end{align}
where the constant $C>0$ only depends on $\|u\|_{L^\infty_t H_x^{s+\gamma_p-}}$.
\end{lemma}
 \begin{proof}
First of all, we rewrite \eqref{Apprx:T2} into five parts $R^n_3=R^n_{31}+\cdots+R^n_{35}$ with
\begin{align*}
 &R^n_{31}\triangleq i\int_{t_n}^{t_{n+1}}\fe^{-i\rho\partial_x^2}\left[|u(\rho)|^2u(\rho)-|\fe^{i\tau\partial_x^2}u(\rho)|^2
 \fe^{i\tau\partial_x^2}u(\rho)\right]d\rho,\\
 &R^n_{32}\triangleq i\int_{t_n}^{t_{n+1}}(\fe^{-i\rho\partial_x^2}-\fe^{-it_{n+1}\partial_x^2})\left(|\fe^{i\tau\partial_x^2}u(\rho)|^2
 \fe^{i\tau\partial_x^2}u(\rho)\right)d\rho,\\
 &R^n_{33}\triangleq i\int_{t_n}^{t_{n+1}}\fe^{-it_{n+1}\partial_x^2}\left[|\fe^{i\tau\partial_x^2}u(\rho)|^2
 \fe^{i\tau\partial_x^2}u(\rho)-|\fe^{i\tau\partial_x^2}u(t_n)|^2
 \fe^{i\tau\partial_x^2}u(t_n)\right]d\rho,
\end{align*}
and
\begin{align*}
  &R^n_{34}\triangleq i\tau\fe^{-it_{n+1}\partial_x^2}\left[|\fe^{i\tau\partial_x^2}u(t_n)|^2-|P_{\leq N}\fe^{i\tau\partial_x^2}u(t_n)|^2\right]
 \fe^{i\tau\partial_x^2}u(t_n),\\
 &R^n_{35}\triangleq\fe^{-it_{n+1}\partial_x^2}\left[\fe^{i\tau\partial_x^2}u(t_n)+i\tau|P_{\leq N}\fe^{i\tau\partial_x^2}u(t_n)|^2\fe^{i\tau\partial_x^2}u(t_n)-\mathcal{N}_\tau[\fe^{i\tau\partial_x^2}u(t_n)]\right].
\end{align*}

For $R^n_{31}$,  note that the inequality $|\fe^{ix}-1|\leq 2^{2-a}|x|^{a}$ holds for any $x\in\R$ and $0\leq a\leq1$, and $s+\gamma_p->\frac12$. Then, by the triangle inequality and Sobolev inequality we have
\begin{align*}
  &\|R^n_{31}\|_{L^2}\lesssim \tau \big\|(\fe^{i\tau\partial_x^2}-1)u\big\|_{L_t^\infty L_x^2}\|u\|_{L^\infty_t H^{\frac12+}_x}^2
  \leq C\tau^{1+\alpha},
\end{align*}
for
\begin{equation*}
  \alpha=\left\{\begin{split}
                   &1,\qquad\quad\quad s+\gamma_p>2,\\
                   &\frac{s+\gamma_p-}{2},\quad s+\gamma_p\leq2.
                \end{split}\right.
\end{equation*}
Here and after $C>0$ is  some constant dependent on $\|u\|_{L^\infty_tH_x^{s+\gamma_p-}}$.

For $R^n_{32}$, we have
\begin{align*}
  \|R^n_{32}\|_{L^2}&\leq \int_{t_n}^{t_{n+1}}\left\|(\fe^{-i(\rho-t_{n+1})\partial_x^2}-1)
  \left(|\fe^{i\tau\partial_x^2}u(\rho)|^2
 \fe^{i\tau\partial_x^2}u(\rho)\right)\right\|_{L^2}d\rho\\
 &\lesssim\int_{t_n}^{t_{n+1}}(\rho-t_{n+1})^\alpha\left\|u(\rho)\right\|_{H^{s+\gamma_p-}}^3d\rho\leq C\tau^{1+\alpha}.
\end{align*}

For $R^n_{33}$, by the triangle inequality and Sobolev inequality we have
\begin{align}\label{est:Rn-33}
\big\|R^n_{33}\big\|_{L^2}
\lesssim  &
 \int_{t_n}^{t_{n+1}}\|u(\rho)-u(t_n)\|_{L^2}\|u\|_{H^{\frac12+}_x}^2\, d\rho
.
\end{align}
Noting that
\begin{align*}
&\big\|u(\rho)-u(t_n)\big\|_{L^\infty_tL^2_x([t_n,t_{n+1}]\times\T)}\\
\lesssim  &
\big\|P_{\le N}\big(u(\rho)-u(t_n)\big)\big\|_{L^\infty_tL^2_x([t_n,t_{n+1}]\times\T)}
+\big\|P_{> N}\big(u(\rho)-u(t_n)\big)\big\|_{L^\infty_tL^2_x([t_n,t_{n+1}]\times\T)}\\
\lesssim  &
\tau \big\|P_{\le N}\partial_t u\big\|_{L^\infty_tL^2_x([t_n,t_{n+1}]\times\T)}
+N^{-s-\gamma_p+}\|u\|_{L^\infty_tH_x^{s+\gamma_p-}}\\
\lesssim  &
\tau \big\|P_{\le N}(\partial_{xx} u+\xi u-\lambda |u|^2u )\big\|_{L^\infty_tL^2_x([t_n,t_{n+1}]\times\T)}
+N^{-s-\gamma_p+}\|u\|_{L^\infty_tH_x^{s+\gamma_p-}}\\
\le & C \left(\tau N^{2-2\alpha}+ N^{-s-\gamma_p+}\right).
\end{align*}
Hence, by Young's inequality,
\begin{align*}
\big\|u(\rho)-u(t_n)\big\|_{L^\infty_tL^2_x([t_n,t_{n+1}]\times\T)}
\le
C\left(\tau^\alpha+N^{-s-\gamma_p+}\right).
\end{align*}
Inserting this inequality into \eqref{est:Rn-33}, we have
\begin{align*}
\big\|R^n_{33}\big\|_{L^2}
\le
C\tau \left(\tau^\alpha+N^{-s-\gamma_p+}\right).
\end{align*}

For $R^n_{34}$, we have
$$
\|R^n_{34}\|_{L^2}\leq C\tau\|u(t_n)-P_{\leq N}u(t_n)\|_{L^2}\leq C\tau N^{-s-\gamma_p+}.
$$

For $R^n_{35}$, we have
\begin{align*}
 \|R^n_{35}\|_{L^2}&\lesssim
 \left\|1+i\tau|P_{\leq N}\fe^{i\tau\partial_x^2}u(t_n)|^2-\exp\left({i\tau|P_{\leq N}\fe^{i\tau\partial_x^2}u(t_n)|^2}\right)\right\|_{L^2}\\
 &\lesssim
 \tau^2\left\||P_{\leq N}\fe^{i\tau\partial_x^2}u(t_n)|^4\right\|_{L^2} \leq C\tau^2.
\end{align*}
Combining the estimates above, we obtain the assertion (\ref{lem:splitting est}) and the proof is done.
\end{proof}

With the established stability results and the local error estimates, we are now ready to  give the proof of \cref{thm:main-N1}.

\begin{proof}

By \cref{main:thm1} and \cref{main:thm3-smoothdata}, we have $u\in L^\infty((0,T);H^{s+\gamma_p-})$.
Moreover, from the proof of \cref{main:thm1}, we have  in fact
\begin{align*}
\|u\|_{L^\infty_tH^{s+\frac32+\frac1p-}_x([0,T]\times\T)}
\le 2\|u_0\|_{H^{s+2}}.
\end{align*}
Therefore,
from \cref{lem:Phi-stab,lem:stability split part,lem:est-Rn1,lem:est-Rn2,lem:splitting}, we have that
\begin{align*}
\|h^{n+1}\|_{L^2}
\le &
\big\|\Phi\big(v^n\big)-\Phi\big(v(t_n)\big)\big\|_{L^2}+\big\|\Psi\big(v^n\big)-\Psi\big(v(t_n)\big)\big\|_{L^2}+
\big\| R^n_1\big\|_{L^2} +\big\|R^n_2 \big\|_{L^2}+\big\|R^n_3 \big\|_{L^2}\\
\le  &
(1+C\tau)\|h^n\|_{L^2}+C\tau N \|h^n\|_{L^2}^3+C\tau^{1+\min\left\{\frac14+\frac1p+s-,1\right\}}\\
&
+C\tau\left(\tau^{\min\left\{1,\frac12(s+\gamma_p)-\right\}}+N^{-s-\gamma_p+}\right),
\end{align*}
where the constant $C>0$ depends only on $\|\xi\|_{\hat b^{s,p}}$, $T$ and $\|u_0\|_{H^{s+2}}$. Noting that
$$
\frac12(s+\gamma_p)-\min\left\{\frac14+\frac1p+s-,1\right\}\ge \frac18,
$$
and so by taking $N= \tau^{-\frac12+\eps_0}$ for any fixed $0<\eps_0\leq\frac{1}{8(s+\gamma_p)}$,  we can further get
\begin{align}
&\big\|h^{n+1}\big\|_{L^2}
\le
C_1\tau^{1+\alpha}+(1+C_2\tau)\|h^n\|_{L^2}+C_3\tau^{\frac12+\eps_0} \|h^n\|_{L^2}^3,\quad 0\leq n<\frac{T}{\tau},\label{hn-indu}
\end{align}
where the constants $C_j>0, j=1,2,3$ depend only on $\|\xi\|_{\hat b^{s,p}}$, $T$ and $\|u_0\|_{H^{s+2}}$.
Here and after, we denote $\alpha=\min\{\frac14+\frac1p+s-,1\}$ for short.

Now we claim that there exists some $\tau_0>0$ (to be determined) such that for any $\tau\in (0,\tau_0]$,
\begin{align}\label{Hy}
\left\|h^{n}\right\|_{H^\gamma}
\le C_1 \tau^{1+\alpha} \sum\limits_{j=0}^n(1+2 C_2\tau)^j,\quad n=0,1,\ldots,\frac{T}{\tau}.
\end{align}
We prove it by induction, see \cite{WuZhao-BIT} for a similar process. Firstly, since $h^0\equiv0$, \eqref{Hy} trivially  holds for $n=0$.  Now we assume that it holds 
till some $0\le n_0\le \frac{T}{\tau}-1$, i.e.,
\begin{align}\label{Hy-1}
\left\|h^{n}\right\|_{H^\gamma}
\le C_1 \tau^{1+\alpha} \sum\limits_{j=0}^n(1+2C_2\tau)^j, \quad \forall 0\le n\le n_0.
\end{align}
From \eqref{Hy-1}, we have that for any $ 0\le n\le n_0$,
\begin{align}\label{bound-v-vn}
\left\|h^{n}\right\|_{H^\gamma}
\le C_4\tau^\alpha,
\end{align}
where $C_4=C_1 C_2^{-1}\fe^{2C_2T}$. Then by \eqref{hn-indu}, we find
\begin{align*}
\left\|h^{n_0+1}\right\|_{H^\gamma}
\le & C_1\tau^{1+\alpha}+\left(1+ C_2\tau+C_3C_4^2\tau^{\frac12+\eps_0+2\alpha}\right)
\cdot C_1 \tau^{1+\alpha} \sum\limits_{j=0}^{n_0}(1+2C_2\tau)^j.
\end{align*}
Note that $\frac12+\eps_0+2\alpha\ge 1+\frac12\eps_0$, then by choosing some $\tau_0\in(0,1]$ such that $C_3C_4^2\tau_0^{\frac12\eps_0}\le C_2$, we can further have that for any $\tau\in (0,\tau_0]$,
\begin{align*}
\left\|h^{n_0+1}\right\|_{H^\gamma}
\le & C_1\tau^{1+\alpha}+\left(1+ C_2\tau+C_3C_4^2\tau^{1+\frac12\eps_0}\right)
\cdot C_1 \tau^{1+\alpha} \sum\limits_{j=0}^{n_0}(1+2C_2\tau)^j\\
= & C_1\tau^{1+\alpha}+
 C_1 \tau^{1+\alpha} \sum\limits_{j=0}^{n_0}(1+2C_2\tau)^{j+1}\\
=&C_1\tau^{1+\alpha} \sum\limits_{j=0}^{n_0+1}(1+2C_2\tau)^{j} .
\end{align*}
This finishes the induction and proves the claim \eqref{Hy}.

Then by iteration, we have
$$
\big\|h^n\big\|_{L^2}
\le  C \tau^{\min\left\{\frac14+\frac1p+s-,1\right\}},
$$
which finishes the proof of the theorem.
\end{proof}

\vskip1cm

\section{Numerical experiments}\label{sec: num result}
In this section, we will carry out some numerical tests to verify the given theoretical results on the regularity of the solution of (\ref{model}) and on the accuracy of the proposed numerical scheme (\ref{NuSo-NLS}). The spatial discretization of  (\ref{NuSo-NLS}) will be implemented here by the Fourier pseudo-spectral method \cite{Trefethen}. In the end, some accuracy comparisons  will be made between (\ref{NuSo-NLS}) and the existing schemes  from the literature.

\subsection{On regularity of solution}\label{sec:numer1}
Firstly, we verify the regularity results on the solution of the PDE, i.e.,  \cref{main:thm1} and \cref{main:thm3-smoothdata}. To do so, we solve the NLS model (\ref{model}) by the numerical scheme (\ref{NuSo-NLS}) with very fine mesh  so that the computations are accurate. We fix
\begin{equation}
\lambda=1,\quad u_0(x)=\frac{\cos(x)}{2+\sin(2x)},\quad x\in(-\pi,\pi),\label{initial}
\end{equation}
for (\ref{model}) in this subsection, and we solve the equation till $t=2$ for the solution $u(t,x)$. The potential function in the following will be constructed as
\begin{equation}\label{num:xi}
\xi(x)=\mathrm{Re}\sum_{k=-N/2}^{N/2-1}\zeta_k\fe^{ik(x+\pi)},
\end{equation}
 by choosing $\zeta_k$ to determine its regularity and by taking an even integer $N$ as the total number of Fourier frequencies which is also the number of spatial grid points.

\begin{figure}[!ht]
	\centering
\psfig{figure =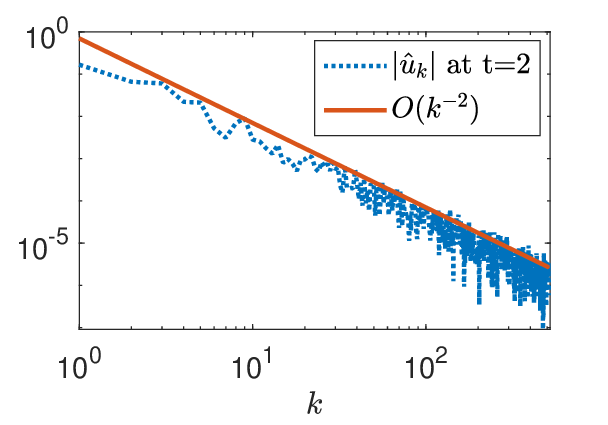, width =6cm,height=4.5cm}
	\psfig{figure =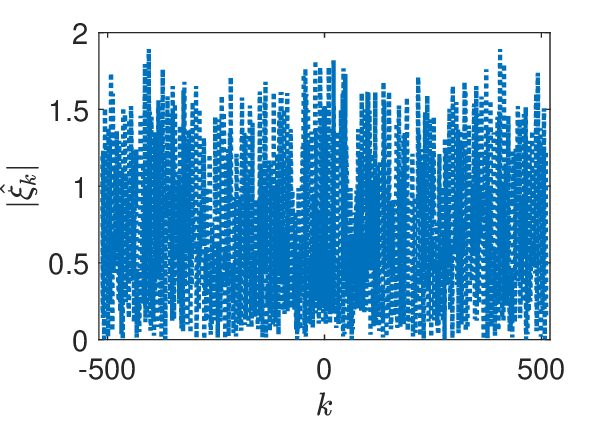, width =6cm,height=4.5cm}
	\caption{Test of \cref{main:thm1} for $s=0,p=\infty$:  modulus of Fourier coefficient of the solution (left) and the generated potential $\xi\in\hat{l}^\infty$ (right).}\label{fig:thm1}
\end{figure}
\begin{example}(Test of \cref{main:thm1} for $s=0,p=\infty$)\label{example1}
We begin with \cref{main:thm1}. Construct a $\xi(x)\in\hat{l}^\infty$ through \eqref{num:xi} by randomly generating $\zeta_k\in[-2,2]$ based on the uniform distribution. To test the regularity result, we compute the Fourier coefficient of the solution: $\hat{u}_k(t)$ at $t=2$. If $u\in H^{3/2-}$ as predicted  by \cref{main:thm1} for $s=0,p=\infty$, then $\sum_{k}|k|^{3-}|\hat{u}_k|^2<\infty$, and so $|k|^{3-}|\hat{u}_k|^2\lesssim k^{-1-}$. This implies that we expect to observe $|\hat{u}_k|\lesssim k^{-2}$  in this case. With $N=1024$, the modulus of Fourier coefficient of the solution  together with that of the generated potential are plotted against the frequency in \cref{fig:thm1}.
\end{example}

\begin{figure}[!ht]
	\centering
\psfig{figure =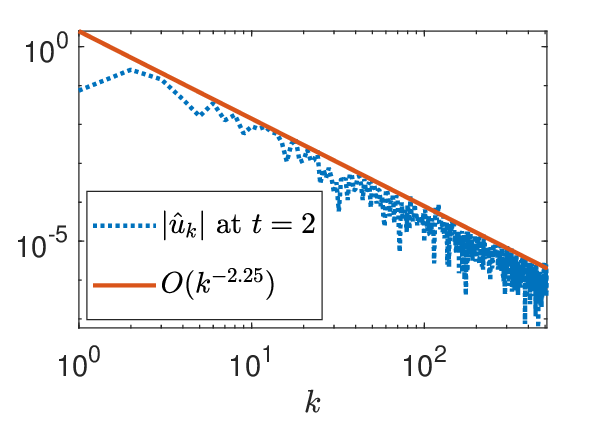, width =6cm,height=4.5cm}
	\psfig{figure =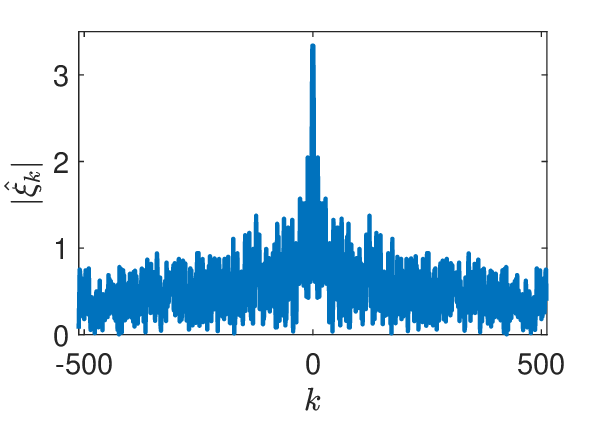, width =6cm,height=4.5cm}
	\caption{Test of \cref{main:thm1} for $s=0,p=4$: modulus of Fourier coefficient of the solution (left) and the generated potential $\xi\in\hat{l}^\infty$ (right).}\label{fig:thm1a}
\end{figure}
\begin{example}(Test of  \cref{main:thm1} for $s=0,p=4$)\label{example1a} To further test \cref{main:thm1}, we then consider a $\xi(x)\in \hat{b}^{0,4}$ through (\ref{num:xi}) by choosing
$$
\zeta_k=\left\{\begin{split}
 &(\eta_{1,k}+i\eta_{2,k})|k|^{-0.26},\quad\ \mbox{if}\ k\neq0,\\
  &1,\qquad\qquad\qquad\qquad\quad \ \, \mbox{if}\ k=0,
  \end{split}\right.$$
with $\eta_{1,k},\eta_{2,k}\in[-4,4]$ randomly generated by the uniform distribution. \cref{main:thm1} in this case predicts $u\in H^{7/4-}$, which means $\sum_{k}|k|^{7/2-}|\hat{u}_k|^2<\infty$. If so, then $|k|^{7/2-}|\hat{u}_k|^2\lesssim k^{-1-}$ and  we are expecting $|\hat{u}_k|\lesssim k^{-2.25}$ for this example.
To verify it, again we compute  $\hat{u}_k(t)$ at $t=2$ and check its decaying rate with respect to $k$.  With $N=1024$, the corresponding numerical results are shown in \cref{fig:thm1a}.
\end{example}

\begin{figure}[!ht]
	\centering
	\psfig{figure =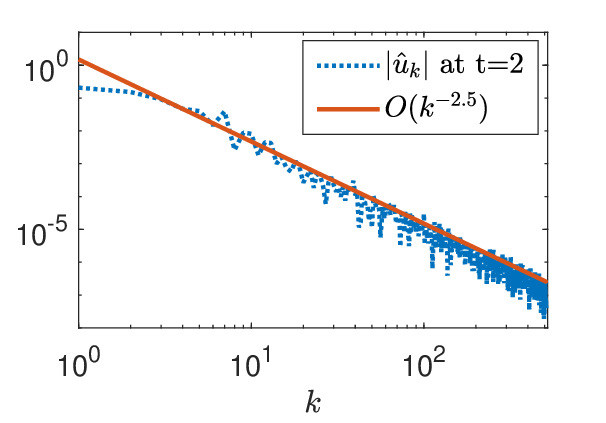, width =6cm,height=4.5cm}
	\psfig{figure =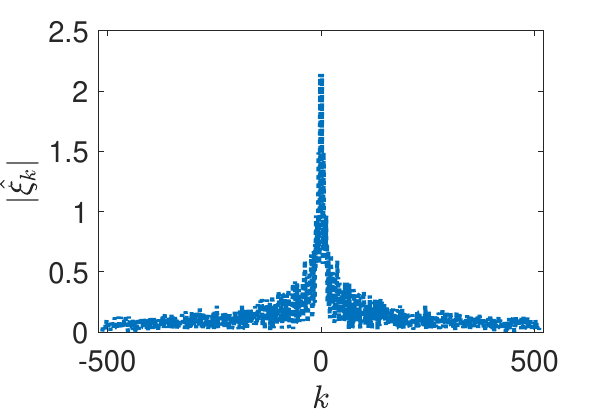, width =6cm,height=4.5cm}
	\caption{Test of \cref{main:thm3-smoothdata} for $s=0,p=2$:  modulus of Fourier coefficient of the solution (left) and the generated potential $\xi\in L^2$ (right).}\label{fig:thm4}
\end{figure}
\begin{example}(Test of  \cref{main:thm3-smoothdata} for $s=0,p=2$)\label{example2} To test \cref{main:thm3-smoothdata},
we construct a $\xi(x)\in L^2$ through (\ref{num:xi}) by choosing
$$
\zeta_k=\left\{\begin{split}
 &(\eta_{1,k}+i\eta_{2,k})|k|^{-0.6},\quad\ \mbox{if}\ k\neq0,\\
  &1,\qquad\qquad\qquad\qquad\quad  \mbox{if}\ k=0,
  \end{split}\right.$$
with $\eta_{1,k},\eta_{2,k}\in[-5,5]$ randomly generated by the uniform distribution. Now \cref{main:thm3-smoothdata}
 tells that $u\in H^{2}$, i.e., $\sum_{k}|k|^{4}|\hat{u}_k|^2<\infty$, and so $|\hat{u}_k|\lesssim k^{-2.5-}$ is expected here.
The Fourier coefficient  at $t=2$  is computed to verify the decaying rate.  With $N=1024$, the numerical results are shown in \cref{fig:thm4}.
\end{example}

All the numerical examples in this subsection illustrate that the decaying rate of the Fourier coefficient of the solution matches well with the expected value from the theorems. Thus, the theoretical results i.e., \cref{main:thm1} and \cref{main:thm3-smoothdata},
on the regularity of the solution of (\ref{model}),
 are valid and sharp. Note in addition that, the smoothness of the used initial data \eqref{initial} in the tests did  not provide more regularity for the solution than expected, and this illustrates \cref{remark1}.

\subsection{On accuracy of scheme}
Next, we test the theoretical result on the convergence order of the numerical scheme, i.e., \cref{thm:main-N1} for LRI (\ref{NuSo-NLS}). We  construct an initial data $u_0\in H^2$ for (\ref{model}) in this subsection as
\begin{equation}\label{u0 h2}
u_0(x)=\sum_{k=-N/2}^{N/2-1}\widehat{(u_0)}_k\fe^{ik(x+\pi)},\qquad
\widehat{(u_0)}_k=\left\{\begin{split}
 &\eta_{k}|k|^{-2.55},\quad \mbox{if}\ k\neq0,\\
  &0,\qquad\qquad\ \ \mbox{if}\ k=0,
  \end{split}\right.
  \end{equation}
  with $\eta_{k}$ randomly sampled from the interval $[0,1]$ by the uniform distribution.  We shall compute the numerical solution of (\ref{model}) at $t=t_n=1$, and we shall measure the relative error
\begin{equation}\label{err def}
error=\|u(t_n)-u^n\|_{L^2}/\|u(t_n)\|_{L^2}\end{equation} of the scheme. The reference solution here is obtained by using very fine mesh size.

\begin{figure}[!ht]
	\centering
\psfig{figure =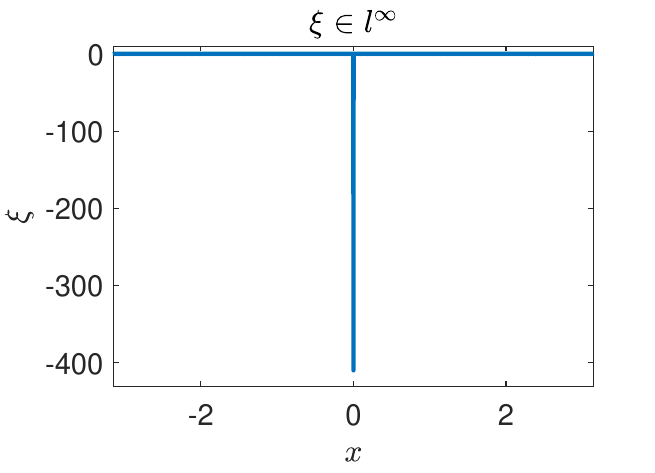, width =4.9cm,height=4.5cm}
\psfig{figure =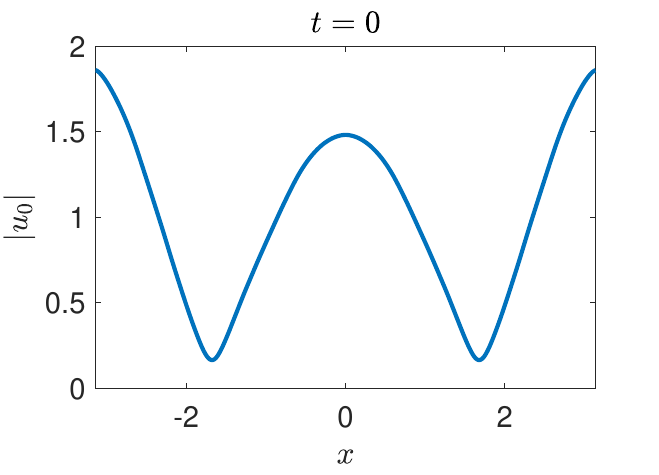, width =4.9cm,height=4.5cm}
	\psfig{figure =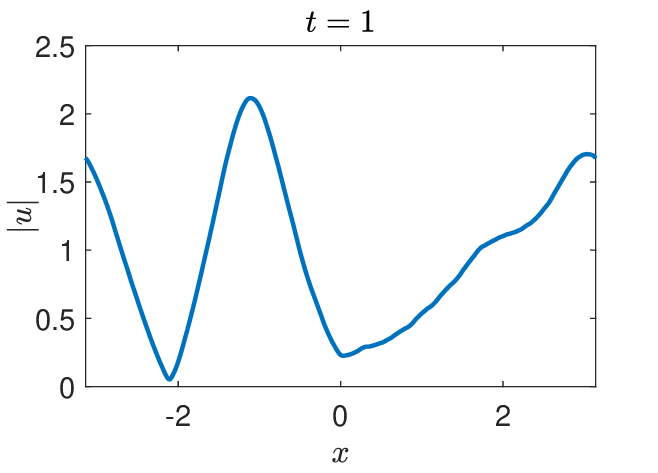, width =4.9cm,height=4.5cm}\\
	\psfig{figure =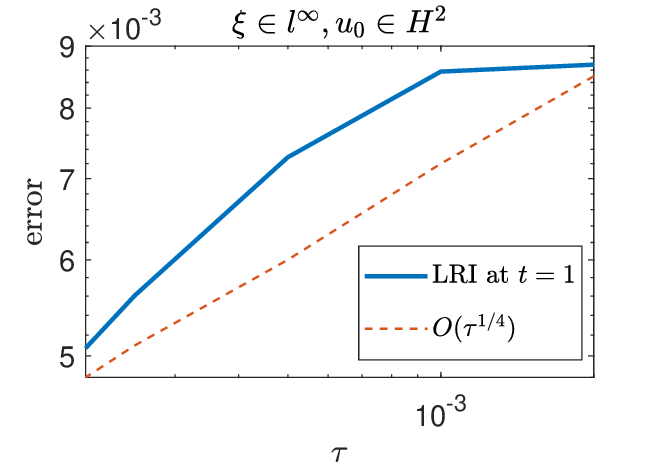, width =7.5cm,height=5cm}
	\caption{Results of \cref{ex:err1}: profiles of the potential $\xi(x)$ and the solution $|u_0(x)|$, $|u(t=1,x)|$ (1st row); error \eqref{err def} of LRI \eqref{NuSo-NLS} (2nd row).}\label{fig:thm9}
\end{figure}
\begin{example}(Lowest order in \cref{thm:main-N1})\label{ex:err1}
We take $\lambda=-2$ and the initial data (\ref{u0 h2}) for  (\ref{model}).
  A potential function $\xi\in \hat{l}^\infty$ is taken as
    $$\xi(x)=-\frac15\sum_{k=-N/2}^{N/2}\fe^{ikx}.$$
    The profiles of the potential and the solution at $t=1$ are displayed in \cref{fig:thm9}. \cref{thm:main-N1} in this case predicts the lowest convergence order for the LRI scheme (\ref{NuSo-NLS}) as $\mathcal{O}(\tau^{1/4-})$.
With the number of spatial grids $N=2048$ fixed, the  discretization error \eqref{err def} of  (\ref{NuSo-NLS}) at $t=1$ is shown in \cref{fig:thm9} under different time steps. As shown by the error curve in \cref{fig:thm9}, the averaged decreasing rate is about $1/4$.
\end{example}

\begin{figure}[h!]
	\centering
\psfig{figure =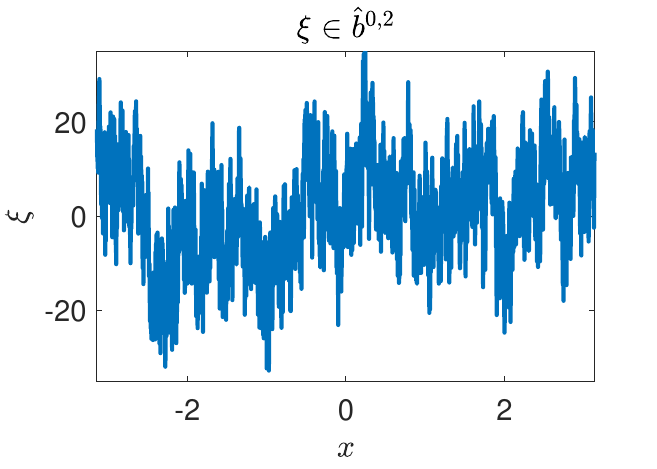, width =5cm,height=4.5cm}
\psfig{figure =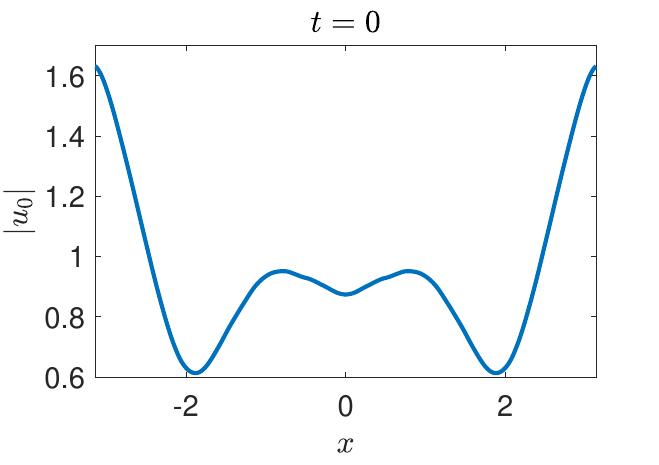, width =5cm,height=4.5cm}\psfig{figure =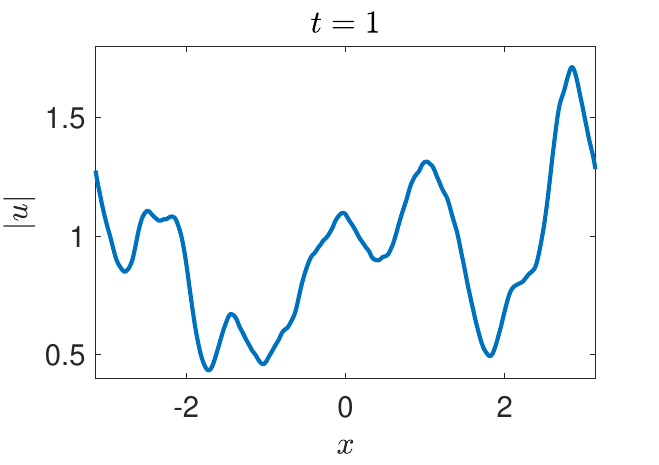, width =5cm,height=4.5cm}\\
	\psfig{figure =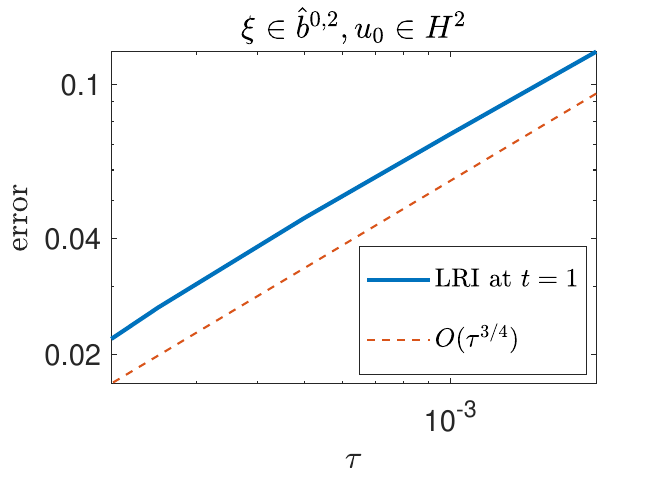, width =7.5cm,height=5cm}
	\caption{Results of \cref{ex:err2} under $\xi\in\hat{b}^{0,2}$: profiles of $\xi(x)$, $|u_0(x)|$ and $|u(t=1,x)|$ (1st row); error \eqref{err def} of LRI \eqref{NuSo-NLS} (2nd row).}\label{fig:thm10}
\end{figure}
\begin{figure}[h!]
	\centering
	\psfig{figure =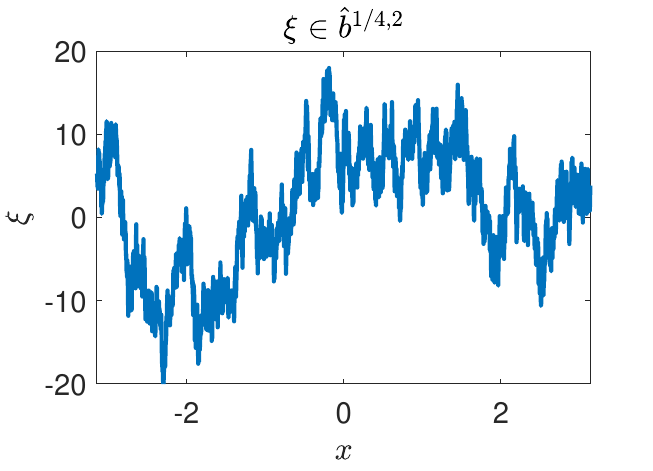, width =5cm,height=4.5cm}\psfig{figure =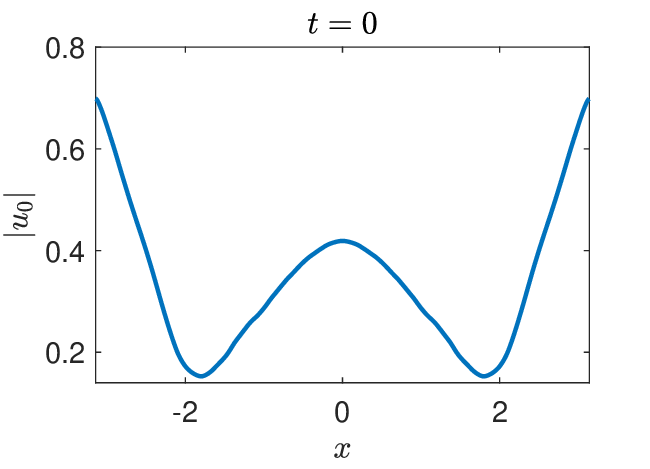, width =5cm,height=4.5cm}\psfig{figure =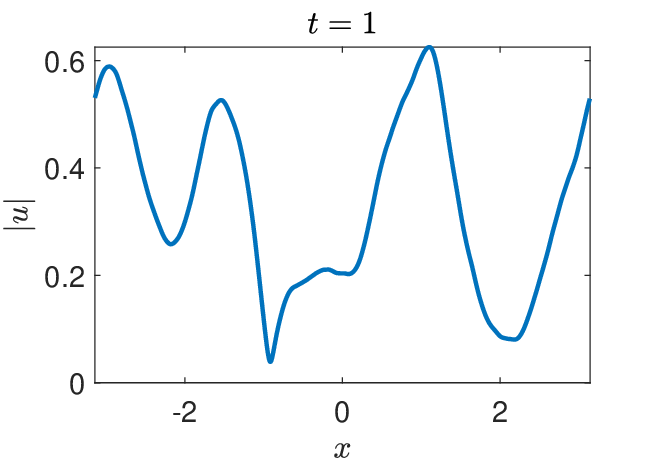, width =5cm,height=4.5cm}\\
	\psfig{figure =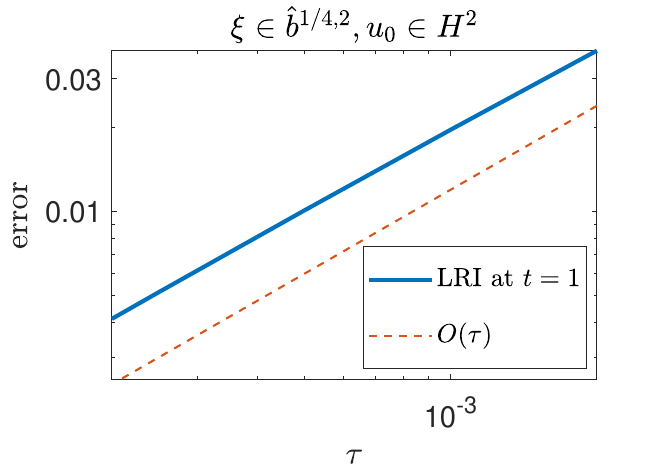, width =7.5cm,height=5cm}
	\caption{Results of \cref{ex:err2} under $\xi\in\hat{b}^{1/4,2}$ : profiles of $\xi(x)$, $|u_0(x)|$ and $|u(t=1,x)|$  (1st row); error \eqref{err def} of LRI \eqref{NuSo-NLS} (2nd row).}\label{fig:thm11}
\end{figure}
\begin{example}(Other orders in \cref{thm:main-N1}) \label{ex:err2}
Generate a $H^2$-initial data as (\ref{u0 h2}) for (\ref{model}) and  fix $\lambda=4$.
We construct the potential $\xi\in\hat{b}^{s,p}$ via
$$\xi(x)=\mathrm{Re}\sum_{k=-N/2}^{N/2-1}{\zeta}_k\fe^{ik(x+\pi)},\quad
{\zeta}_k=\left\{\begin{split}
 &(\eta_{1,k}+i\eta_{2,k})|k|^{-\delta},\quad \mbox{if}\ k\neq0,\\
  &1,\qquad\qquad\qquad\qquad \, \mbox{if}\ k=0,
  \end{split}\right.$$
  with $\eta_{1,k},\eta_{2,k}\in[-5,5]$ randomly generated by the uniform distribution.
  We consider
    $$
\delta=\left\{\begin{split}
 &0.51,\qquad \mbox{for}\quad \xi(x)\in\hat{b}^{0,2},\\
  &0.76,\qquad \mbox{for}\quad \xi(x)\in\hat{b}^{1/4,2},
  \end{split}\right.$$
  and \cref{thm:main-N1} respectively predicts the accuracy of LRI (\ref{NuSo-NLS}) as $\mathcal{O}(\tau^{3/4-})$ and $\mathcal{O}(\tau^{1-})$.
 The number of spatial grids $N=2048$ is again used and fixed for computations. For $\delta=0.51$, the  discretization error \eqref{err def} of  (\ref{NuSo-NLS}) at $t=1$  together with the profiles of the potential and solution in this case are shown in \cref{fig:thm10}. The corresponding results of $\delta=0.76$ are shown in \cref{fig:thm11}.
 The two error curves in \cref{fig:thm10} and \cref{fig:thm11} clearly decrease at the expected rates.

 In total, the observed convergence results in this subsection all match well with \cref{thm:main-N1}. This verifies its validity and indicates the sharpness of the error estimate (\ref{error-tau-1}).

\end{example}

\subsection{Accuracy comparison} At last, we conduct some numerical tests to compare  the convergence/accuracy of the proposed LRI (\ref{NuSo-NLS}) with the existing schemes in the literature for (\ref{model}). The concerned numerical schemes from the literature are listed below.
\begin{itemize}
  \item The most traditional finite difference scheme:
$i\frac{1}{\tau}(u^{n+1}-u^n)+\partial_x^2u^{n+1}+\xi u^{n+1}=\lambda|u^n|^2u^n.$
  \item The commonly used Lie-Trotter splitting scheme:
$u^{n+1}=\fe^{i\partial_x^2\tau}\fe^{-i\tau(-\xi+\lambda |u^n|^2)}u^n.$
\item The recently analyzed exponential wave integrator (EWI) \cite{Bao1}:
$u^{n+1}=\fe^{i\partial_x^2\tau}u^n-i\tau\mathcal{D}_\tau[(-\xi+
\lambda|u^n|^2)u^n].$
\item The recently proposed LRI from \cite{Bronsard}:
$u^{n+1}=\fe^{i\partial_x^2\tau}[u^n+i\tau u^n\mathcal{D}_{-\tau}\xi-i\tau\lambda(u^n)^2
\mathcal{D}_{-2\tau}\overline{u^n}].$
\end{itemize}
Note that the schemes presented above all have higher order versions, but their high order accuracy is achieved only when the setup of the NLS model is smooth enough.
Here what we would like to address is the performance of schemes under the rough setup, so we choose to focus on the first order schemes for tests and comparisons. The reference solution and the measure of error are set the same as before.

\begin{figure}[h!]
	\centering
	\psfig{figure =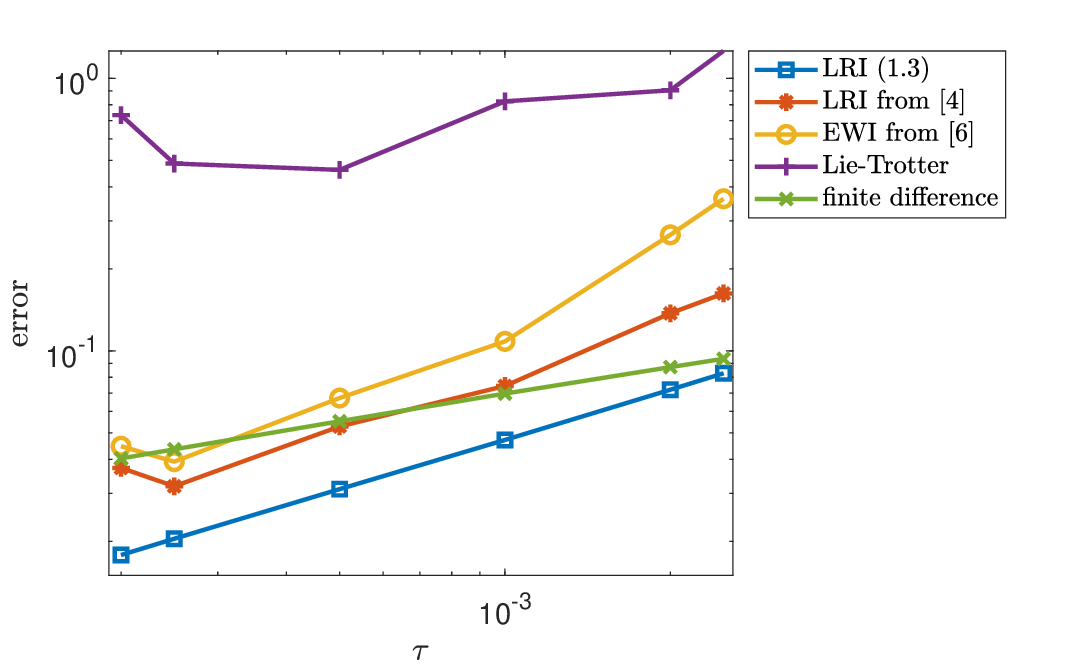, width =7.5cm,height=5cm}
	\psfig{figure =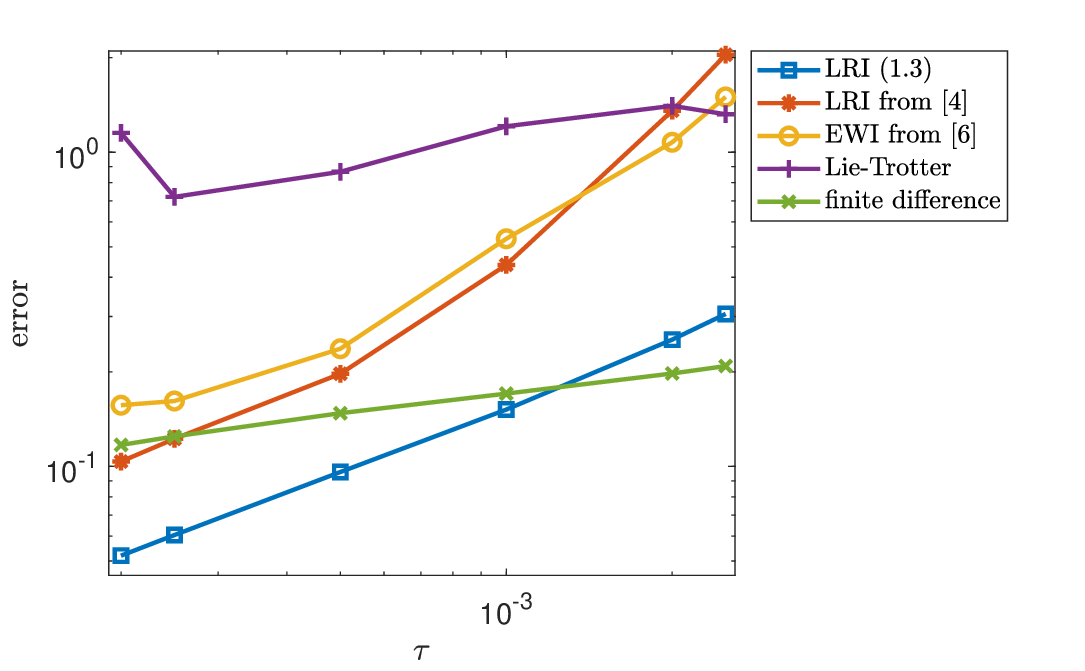, width =7.5cm,height=5cm}
	\caption{Accuracy comparison of schemes for \cref{num:comp ex1}: with the set of data \eqref{num:comp ex1 set1} (left); with the set of data \eqref{num:comp ex1 set2} (right).}\label{fig:comp1}
\end{figure}
\begin{example}\label{num:comp ex1}
We first perform comparisons of schemes within the setup of \cref{thm:main-N1}: $u_0\in H^2,s=0,p=\infty$.
Generate an initial data $u_0\in H^2$ for (\ref{model}) as
\begin{equation*}
u_0(x)=\sum_{k=-N/2}^{N/2-1}\widehat{(u_0)}_k\fe^{ik(x+\pi)},\qquad
\widehat{(u_0)}_k=\left\{\begin{split}
 &\eta_{k}|k|^{-2.51},\quad \mbox{if}\ k\neq0,\\
  &0,\qquad\qquad\ \ \mbox{if}\ k=0.
  \end{split}\right.
  \end{equation*}
  The coefficients $\eta_k,\lambda$ and the potential $\xi\in\hat{l}^\infty$ for (\ref{model}) are taken as the following two sets of data.
  The first set considers the $\delta$-type potential function:
\begin{equation}
\eta_{k}\in[0,1/3],\quad  \lambda=-2,\quad \xi(x)=-\sum_{k=-N/2}^{N/2}\left(\fe^{ikx}+\fe^{ik(x+2)}+\fe^{ik(x-2)}\right).\label{num:comp ex1 set1}
\end{equation}
The second set considers more general rough potential:
\begin{equation}
\mathrm{Re}(\eta_{k})\in[0,1/3],\  \mathrm{Im}(\eta_{k})\in[0,1/5],\quad \lambda=-0.05,\quad \xi(x)=\mathrm{Re}\sum_{k=-N/2}^{N/2-1}{\zeta}_k\fe^{ik(x+\pi)},\label{num:comp ex1 set2}
\end{equation}
with $\zeta_{k}\in[0,4]$ sampled by the uniform distribution. Under (\ref{num:comp ex1 set1}) or \eqref{num:comp ex1 set2} with $N=2048$ fixed, the error of each concerned numerical scheme  at $t=1$ is given in \cref{fig:comp1}.

It can be seen from the plots in \cref{fig:comp1} that the proposed LRI (\ref{NuSo-NLS}) is more accurate than all the other considered schemes. In contrast, the LRI \cite{Bronsard} and the EWI \cite{Bao1}  converge slowly and the accuracy order is unclear, while the Lie-Trotter splitting scheme is not working at all.
  \end{example}

\begin{figure}[h!]
	\centering
	\psfig{figure =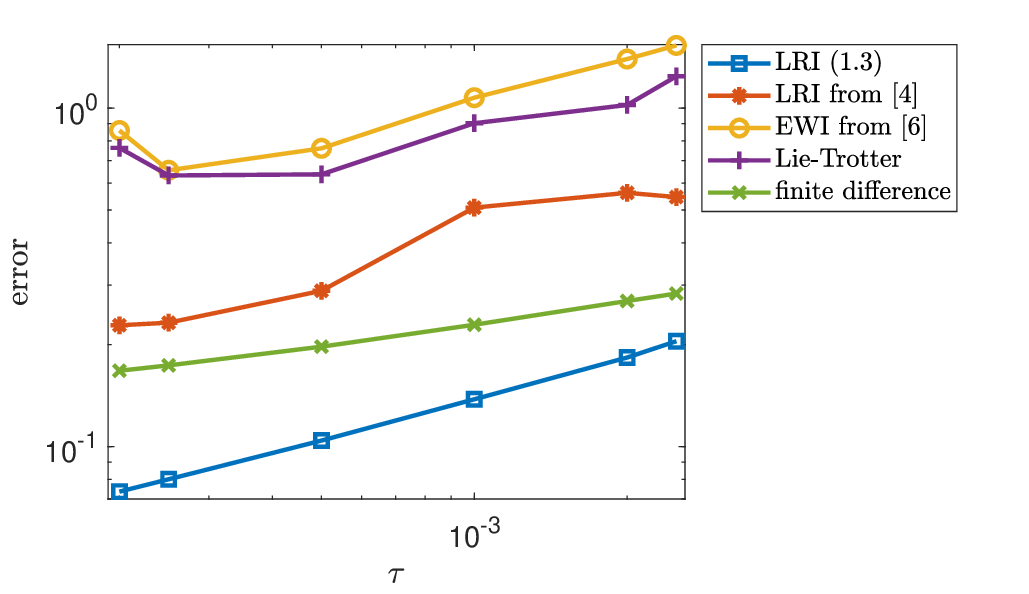, width =10cm,height=5cm}
	\caption{Accuracy comparison of schemes for \cref{num:comp ex2}.}\label{fig:comp2}
\end{figure}
\begin{example}\label{num:comp ex2}
We then test the performance of the schemes under rougher initial data for (\ref{model}) by generating a $u_0\in L^2$ as
\begin{equation*}
u_0(x)=\sum_{k=-N/2}^{N/2-1}\widehat{(u_0)}_k\fe^{ik(x+\pi)},\quad
\widehat{(u_0)}_k=\left\{\begin{split}
 &\eta_{k}|k|^{-1.1},\quad \mbox{if}\ k\neq0,\\
  &0,\qquad\qquad\  \mbox{if}\ k=0,
  \end{split}\right.
  \end{equation*}
  with the other parameters set as \eqref{num:comp ex1 set1}. With the number of spatial grid points fixed as $N=2048$, the error of each scheme  at $t=1$ is given in \cref{fig:comp2}. Although the tested setup is not covered by \cref{thm:main-N1}, we can see that the proposed LRI scheme (\ref{NuSo-NLS}) is still working well in this case, and its accuracy is much better than the others. This illustrates that (\ref{NuSo-NLS}) can have more advantages for solving the NLS equation (\ref{model}) under rough setup.
  \end{example}

\vskip1cm
\section{Conclusion}\label{sec: con}
We present new, sharp results of the cubic nonlinear Schr\"odinger equation (NLS) with a spatially rough potential, posed on a one-dimensional torus which is the mathematical model for nonlinear Anderson localization. We deal both with the well/ill-posedness analysis on the PDE level and its application for numerical discretization and convergence analysis.
In the PDE analysis, we provide insights into how the regularity of the solution is impacted by the regularity of the potential, offering quantitative and explicit characterizations. Additionally, we establish ill-posedness results to demonstrate the sharpness of our regularity characterizations and to identify the minimum required regularity of the potential for the solvability of the NLS model.
Based on our regularity results, we design a suitable numerical discretization for the model and demonstrate its convergence with an optimal error bound. The numerical experiments demonstrate the theoretical regularity results on the PDE level and also validate the established convergence rate of the proposed scheme. Furthermore, some comparisons with existing schemes are given, showcasing the superior accuracy of our scheme in the case of a rough potential.

\section*{Acknowledgements}
Financial support from the Austrian Science Fund (FWF) via the SFB
project F65 is acknowledged.
Y. Wu is partially supported by NSFC 12171356.   X. Zhao is partially supported by NSFC 12271413 and the Natural Science Foundation of Hubei Province 2019CFA007. X. Zhao thanks the Wolfgang Pauli Institute Vienna for hospitality.

\bibliographystyle{model1-num-names}

\end{document}